\documentclass[a4paper]{cedram-smai-jcm}

\usepackage{tikz}
\usepackage{subcaption}
\usepackage{graphicx}
\usepackage[T1]{fontenc} 							
\usepackage[english]{babel} 	
\usepackage{amssymb,amsmath}
\usepackage{bm} 					
\usepackage{stmaryrd} 				
\usepackage{scalerel} 				
\usepackage{nicefrac}								
\usepackage{siunitx}  		
\usepackage{booktabs} 				
\usepackage{mathtools}				
\usepackage{empheq}	
\usepackage[normalem]{ulem}
\usepackage[colorinlistoftodos,prependcaption]{todonotes}
\usepackage{array}
\usepackage{booktabs}

\newcommand\redsout{\bgroup\markoverwith{\textcolor{red}{\rule[0.5ex]{2pt}{0.5pt}}}\ULon}

\newlength\bshft
	\def\fakebold#1{\ThisStyle{\ooalign{$\SavedStyle#1$\cr%
  	\kern-\bshft$\SavedStyle#1$\cr%
  	\kern\bshft$\SavedStyle#1$}}}

\newcommand{\R}{\mathbb{R}}

\newcommand\Pk[2]{{ \mathbb{P}_{#1}{#2} }}
\newcommand\Pdk[2]{{ \mathbb{P}_{#1}^\mathrm{dc}{#2} }}
\newcommand\PPk[2]{{ \fakebold{\mathbb{P}}_{#1}{#2} }}
\newcommand\PPdk[2]{{ \fakebold{\mathbb{P}}_{#1}^\mathrm{dc}{#2} }}

\newcommand\BDMk[2]{{ \fakebold{\mathbb{BDM}}_{#1}{#2} }}

\newcommand{\helm}{{ \mathbb{P} }}
\newcommand{\stokes}{{ \mathbb{S} }}
\newcommand\HL[1]{{ \mathbb{P}{\rb{#1}} }}
\newcommand\HLh[1]{{ \mathbb{P}_h{\rb{#1}} }}
\newcommand\HLSV[1]{{ \mathbb{P}_h^\mathrm{SV}{\rb{#1}} }}
\newcommand\HLhd[1]{{ \mathbb{P}_h^\dvg{\rb{#1}} }}
\newcommand\HLhdc[1]{{ \mathbb{P}_h^0{\rb{#1}} }}

\newcommand{\ku}{{k_{\uu}}}
\newcommand{\kupr}{{k_{\uu}^{\mathrm{pr}}}}
\newcommand{\kppr}{{k_p^{\mathrm{pr}}}}
\newcommand{\kucl}{{k_{\uu}^{\mathrm{cl}}}}
\newcommand{\kpcl}{{k_p^{\mathrm{cl}}}}
\newcommand{\inv}{{ \mathrm{inv} }}

\newcommand{\dvg}{{ \mathrm{div} }}

\newcommand{\OMEGA}{{ \rb{\Omega} }}

\newcommand\Rey{\mbox{\textit{Re}}}

\newcommand{\Drm}{{ \mathrm{D} }}

\newcommand\Kin[1]{{ \mathcal{K}\rb{#1} }}

\newcommand\Ens[1]{{ \mathcal{E}\rb{#1} }}
\newcommand\Pal[1]{{ \mathcal{P}\rb{#1} }}

\DeclareMathOperator{\ip}{{\boldsymbol{\cdot}}}
\DeclareMathOperator{\Fip}{{\boldsymbol{:}}}
\DeclareMathOperator{\DIV}{\nabla\,\boldsymbol{\cdot}}
\DeclareMathOperator{\DIVh}{\nabla_\textit{h}\,\boldsymbol{\cdot}}
\DeclareMathOperator*{\argmin}{arg\,min}

\newcommand{\vort}{{\bm \omega}}
\newcommand{\ff}{{ \boldsymbol{f} }}
\newcommand{\gbld}{{ \boldsymbol{g} }}
\newcommand{\ee}{{ \boldsymbol{e} }}
\newcommand{\rr}{{ \boldsymbol{r} }}
\newcommand{\uu}{{ \boldsymbol{u} }}
\newcommand{\vv}{{ \boldsymbol{v} }}
\newcommand{\ww}{{ \boldsymbol{w} }}
\newcommand{\cc}{{ \boldsymbol{c} }}

\newcommand{\xx}{{ \boldsymbol{x} }}
\newcommand{\zz}{{ \boldsymbol{z} }}

\newcommand{\nn}{{ \boldsymbol{n} }}
\newcommand{\zero}{{ \boldsymbol{0} }}
\newcommand{\tend}{{ T }}
\newcommand{\gD}{{ \boldsymbol{g}_\Drm }}

\newcommand{\drm}{{ \mathrm{d} }}
\newcommand{\dx}{{ \,\drm\xx }}
\newcommand{\ds}{{\,\drm\boldsymbol{s}}}
\newcommand{\dtau}{{ \,\drm\tau }}

\newcommand{\drms}{{ \,\drm s }}

\newcommand{\eps}{{ \varepsilon }}

\newcommand{\lavg}{{ \big\{\hspace{-0.99ex}\big\{ }}						
\newcommand{\ravg}{{ \big\}\hspace{-0.99ex}\big\} }}		
\newcommand{\ljmp}{ \left\llbracket }	
\newcommand{\rjmp}{ \right\rrbracket }	
\newcommand\jmp[1]{{ \ljmp#1\rjmp }}										
\newcommand\avg[1]{{ \lavg#1\ravg }}

\newcommand\LTWO{{ \boldsymbol{L}^{2} }}
\newcommand\LPOne[1]{\boldsymbol{L}^{#1}} 
\newcommand\Lp[2]{{ L^{#1}{#2} }} 
\newcommand\LP[2]{{ \boldsymbol{L}^{#1}{#2} }}
\newcommand\Lpz[2]{{ L_0^{#1}{#2} }}

\newcommand\Wmp[3]{{ W^{#1,#2}{#3} }}
\newcommand\WMP[3]{{ \boldsymbol{W}^{#1,#2}{#3} }}

\newcommand\Hm[2]{{ H^{#1}{#2} }}
\newcommand\Hmz[2]{{ H_0^{#1}{#2} }}
\newcommand\HM[2]{{ \boldsymbol{H}^{#1}{#2} }}
\newcommand\HMh[2]{{ \boldsymbol{H}_h^{#1}{#2} }}

\newcommand\XX{{ \boldsymbol{X} }}
\newcommand\HDIV{{ \boldsymbol{H}{\rb{\dvg}} }}
\newcommand{\Hdiv}{{ \boldsymbol{H}{\rb{\dvg;\Omega}} }}		
			
\newcommand{\VV}{{ \boldsymbol{V} }}	
\newcommand{\WW}{{ \boldsymbol{W} }}	
	
\newcommand{\HH}{{ \boldsymbol{H} }}							
\newcommand{\Q}{{ Q }}

\newcommand{\T}{{ \mathcal{T}_h }} 
\newcommand{\F}{{ \mathcal{F}_h }}	
\newcommand{\FK}{{ \mathcal{F}_K }}	
\newcommand{\Fi}{{ \mathcal{F}_h^i }}								
\newcommand{\Fb}{{ \mathcal{F}_h^\partial }}		

\newcommand\rb[1]{{ \left(#1\right) }}
\newcommand\sqb[1]{{ \left[ #1 \right] }}
\newcommand\rsb[1]{{ \left(#1\right] }}
\newcommand\set[1]{{ \left\{ #1 \right\} }}
\newcommand\bra[1]{{ \langle #1 \rangle }}
\newcommand\abs[1]{{ \left\lvert#1\right\rvert }}
\newcommand\norm[1]{ \left\lVert#1\right\rVert }

\newcommand\nf[2]{{ \nicefrac{#1}{#2} }}

\newcommand{\tripnorm}[1]{{\left\vert\kern-\nulldelimiterspace\left\vert\kern-\nulldelimiterspace\left\vert #1
	\right\vert\kern-\nulldelimiterspace\right\vert\kern-\nulldelimiterspace\right\vert}}
	
\newcommand{\otoprule}{\midrule[\heavyrulewidth]}

\newcommand\restr[2]{{												
	\left.\kern-\nulldelimiterspace									
	#1
	\vphantom{\big|}
	\right|_{#2}
	}}				
	
\newcolumntype{N}{>{\centering\arraybackslash}m{.5in}}
\newcolumntype{G}{>{\centering\arraybackslash}m{2in}}

\DeclareSIUnit[number-unit-product = {}]\Q{~}
\DeclareSIPrefix\kilo{K}{3}
\DeclareSIPrefix\mega{M}{6}
\DeclareSIPrefix\giga{G}{9}
\DeclareSIPrefix\terra{T}{12}
\newcommand{\numQ}[1]{\SI[round-mode=figures,round-precision=3,zero-decimal-to-integer,exponent-to-prefix = true,scientific-notation = engineering]{#1}{\Q}\!}

\title[Pressure-robustness, high Reynolds numbers, Beltrami flows]{On high-order pressure-robust space discretisations, their advantages for incompressible high Reynolds number generalised Beltrami flows and beyond}

\author[N.R.\ Gauger]{\firstname{Nicolas} \middlename{R.} \lastname{Gauger}}
\address{Chair for Scientific Computing, TU Kaiserslautern, 67663 Kaiserslautern, Germany}
\email{nicolas.gauger@scicomp.uni-kl.de}

\author[A.\ Linke]{\firstname{Alexander} \lastname{Linke}}
\address{Weierstrass Institute, 10117 Berlin, Germany}
\email{alexander.linke@wias-berlin.de}
\thanks{A.\ Linke ORCID: \url{https://orcid.org/0000-0002-0165-2698}}

\author[P.W.\ Schroeder]{\firstname{Philipp} \middlename{W.} \lastname{Schroeder}}
\address{Institute for Numerical and Applied Mathematics, Georg-August-Universit\"at G\"ottingen, 37083 G\"ottingen, Germany}
\email{p.schroeder@math.uni-goettingen.de}
\thanks{P.W.\ Schroeder ORCID: \url{https://orcid.org/0000-0001-7644-4693}}

  

\definecolor{mediumblue}{RGB}{0,0,205}
\definecolor{forestgreen}{RGB}{34,139,34}
\definecolor{darkred}{RGB}{200,0,0}

\begin{document}

\begin{abstract}
An improved understanding of the divergence-free constraint for the incompressible Navier--Stokes equations leads to the observation that a semi-norm and corresponding equivalence classes of forces are fundamental for their nonlinear dynamics.
The recent concept of {\em pressure-robustness} allows to distinguish between space discretisations that discretise these equivalence classes appropriately or not.
This contribution compares the accuracy of pressure-robust and non-pressure-robust space discretisations for transient high Reynolds number flows, starting from the observation that in generalised Beltrami flows the nonlinear convection term is balanced by a strong pressure gradient. 
Then, pressure-robust methods are shown to outperform comparable non-pressure-robust space discretisations.
Indeed, pressure-robust methods of formal order $k$ are comparably accurate than non-pressure-robust methods of formal order $2k$ on coarse meshes.
Investigating the material derivative of incompressible Euler flows, it is conjectured that strong pressure gradients are typical for non-trivial high Reynolds number flows.
Connections to  vortex-dominated flows are established.
Thus, pressure-robustness appears to be a prerequisite for accurate incompressible flow solvers at high Reynolds numbers.
The arguments are supported by numerical analysis and numerical experiments.
\end{abstract}

\maketitle

\section{Introduction}	
\label{sec:Introduction}
Recently, it was revealed that entire families of convergent space discretisations for the incompressible Navier--Stokes equations
\begin{subequations} \label{eq:transient:navier:stokes}
	\begin{empheq}[left=\empheqlbrace]{alignat=2} 
  		\partial_t\uu - \nu \Delta \uu + \rb{\uu \ip \nabla} \uu + \nabla p & =  \ff, \\
    	\DIV \uu & = 0,		
	\end{empheq} 
\end{subequations}
may deliver inaccurate velocity solutions when strong pressure gradients develop, i.e.\ they suffer from a lack of {\em pressure-robustness} \cite{JLMNR:sirev, cmame:linke:merdon:2016, lr:2018}.
Nearly all classical mixed methods like the Taylor--Hood element or (`only' $\LP{2}{}$-conforming) Discontinuous Galerkin (DG) methods belong to these families.
Strong pressure gradients reflect strong forces of gradient type within the Navier--Stokes momentum balance, e.g., in the terms $\ff$, $\rb{\uu \ip \nabla} \uu$ or $\partial_t\uu$.

Indeed, the lack of {\em pressure-robustness} has been a rather hot research topic in the beginning of the history of finite element methods for CFD \cite{pfc:1989, ff:1985, st:1987, glcl:1980, DGT94, gerbeau:1997} --- sometimes called {\em poor mass conservation} --- and continued to be investigated for many years \cite{GLRW:2012, s:1997, ganesan:john:pressure:separation, SchroederEtAl18}, often in connection with the so-called \emph{grad-div stabilisation} \cite{FH88, olhl2009, CELR:2011, jjlr:2014, AkbasEtAl18, john:novo:2018}.
Also, in the geophysical fluid dynamics community and in numerical astrophysics {\em well-balanced} schemes have been proposed to overcome similar issues for related Euler and shallow-water equations, especially in connection to nearly-hydrostatic and nearly-geostrophic flows; cf., for example, \cite{CotterShipton:2012,CotterThuburn:2014, BottaKlein:2004,Klingenberg:2019}. 
~\\

However, only recently it was understood better that exactly the relaxation of the divergence constraint for incompressible flows, which was invented in classical mixed methods in order to construct discretely inf-sup stable discretisation schemes, introduces the lack of {\em pressure-robustness}, since it leads to a poor discretisation of the Helmholtz--Hodge projector \cite{lm:2018}. 
The reason is that the relaxation of the divergence constraint implies a relaxation of the $\LTWO$-orthogonality between discretely divergence-free velocity test functions and arbitrary gradient fields.
~\\

Fortunately, pressure-robust space discretisations behave in a robust manner when confronted with strong pressure gradients, and many different ways to construct such schemes have been found recently. 
To name only a few, inf-sup stable $\HM{1}{}$-conforming and divergence-free mixed methods \cite{zhang:2005}, inf-sup stable $\HDIV$-conforming DG methods \cite{CockburnEtAl07, Lehrenfeld10} and inf-sup stable $H^1$-conforming and nonconforming finite element methods (FEM), finite volume (FVM) methods, and Hybrid High Order methods (HHO) with appropriately modified velocity test functions \cite{Linke:2012, linke:cmame:2014, dPEL:2016, JLMNR:sirev, lmt:2016, cmame:linke:merdon:2016} are pressure-robust.
Moreover, also in the context of isogeometric analysis various pressure-robust discretisations have been developed \cite{BuffaEtAl11,EvansHughes13,EvansHughes13b}.
However, it is still not generally widely accepted in the numerical analysis community that {\em pressure-robustness} is simply a prerequisite for the accurate space discretisation of non-trivial Navier--Stokes flows.
~\\

Thus, the goals of this contribution are threefold:
\begin{enumerate}
\item It will be shown that the need for {\em pressure-robustness} emanates from an improved understanding of mixed methods and the divergence constraint in incompressible flows.
It is argued that the {\em divergence constraint} induces equivalence classes of forces that are connected to a semi-norm. The involved semi-norm, in turn, is connected to the Helmholtz--Hodge projector of a vector field and vanishes for arbitrary gradient fields.
\item It will be argued that exactly the quadratic nonlinearity of the incompressible Navier--Stokes equations is a major source for strong pressure gradients. An example are vortex-dominated flows with a typical balance of the centrifugal forces --- represented by the nonlinear convection term --- and the pressure gradient.
Then, the nonlinear convection term contains a strong gradient part in the sense of the Helmholtz--Hodge decomposition.
The corresponding pressure is strong and complicated to approximate due to the balance of a linear term (the pressure gradient) with a quadratic term (the nonlinear convection).
\item It will be demonstrated that pressure-robust schemes outperform non-pressure-robust schemes for entire classes of transient incompressible flows at high Reynolds numbers.
For generalised Beltrami flows and vortex-dominated flows it will be demonstrated that a pressure-robust scheme with polynomial order $k \geq 2$ for the discrete velocity will be comparably accurate to a non-pressure-robust scheme of order $2 k$ on coarse grids.
The astonishing factor $2$ in the possible reduction of the polynomial approximation order stems from the balance of the quadratic nonlinear term with the linear pressure gradient. 
\end{enumerate}

We only briefly remark that the question of an appropriate discretisation of the nonlinear convection term is intimately connected to the issue of numerical convection stabilisation techniques like upwinding or SUPG \cite{Riviere08,brooks:hughes:1982}.
With the help of generalised Beltrami flows, we will demonstrate that in real-world flows the nonlinear convection term can be strong, even if the dynamics of the flow is not convection-dominated at all, i.e.\ when measured in the appropriate semi-norm. 
Thus, our contribution opens the way to an improved understanding of convection stabilisation for incompressible Navier--Stokes flows.
Here, the notion of numerical {\em pseudo-dominant convection} is decisive, see Remark \ref{rem:pseudo:dominant:convection}.
~\\

The arguments will be supported by a comparative and paradigmatic numerical analysis of $\HM{1}{}$-conforming pressure-robust and non-pressure-robust space discretisations for transient incompressible Navier--Stokes flows. 
The analysis exploits essentially the following three observations \cite{alm:2018, lm:2018}:
\begin{itemize}
\item a pressure-robust space discretisation of the time-dependent Stokes equation for \emph{small viscosities} is essentially error-free on finite (sufficiently short w.r.t.\ the viscosity) time intervals, i.e., the approximation error of the initial values does not grow in time;
\item under the same conditions, classical space discretisations of the time-dependent Stokes problem only suffer from large gradient fields in the momentum balance (large pressures), and discrete velocity errors induced by gradient fields accumulate over time;
\item the nonlinear convection term is a major source for complicated pressure gradients.
\end{itemize}

Several numerical experiments will illustrate the theory.
In order to explicitly focus on space discretisation, in the practical examples always small time steps are chosen together with second-order time-stepping schemes.
Therefore, the error due to time discretisation is always negligible in this work.

\paragraph{Organisation of the article} 
As a basis for this work, Section \ref{sec:FluidDynamics} presents some fundamental reflections on the transient incompressible Navier--Stokes equations, which help to understand the significance of a pressure-robust space discretisation.
Among other things, we explain why equivalence classes of forces are important for Navier--Stokes flows, we introduce the notion of generalised Beltrami flows, and  we emphasise that the material derivative in incompressible Euler flows with $\ff=\zero$ is always a gradient field.
Afterwards, in Section \ref{sec:NSAndDiscretisation}, the time-dependent Navier--Stokes problem, its weak formulation, the Helmholtz--Hodge projector and its discrete counterpart are discussed in an $\HM{1}{}$-conforming FE setting.
Also for $\HM{1}{}$-conforming FEM, a comparative time-dependent $\LP{2}{}$ \emph{a priori} error analysis is presented in Section \ref{sec:ErrorAnalysisH1}.
Section \ref{sec:Taylor} discusses how and when non-pressure-robust high-order methods lose about half of their formal convergence order on coarse meshes.
In Section \ref{sec:GreshoH1}, the relevance of our considerations for vortex-dominated flows is explained showing numerical results for the Gresho vortex problem computed with $\HM{1}{}$-FEM. 
Moving to computationally much more versatile $\LP{2}{}$- and $\HDIV$-DG methods, Section \ref{sec:dGFEM} describes their space discretisation and the corresponding DG Helmholtz projectors.
The remainder of the work is dedicated to numerical experiments. 
While Section \ref{sec:ExperimentsKnownSol} deals with generalised Beltrami flows with exact solutions in 2D and 3D, in Section \ref{sec:KarmanBeltrami} we go beyond this and investigate the material derivative of a real-world flow: a von K\'arm\'an vortex street.
Finally, some conclusions are drawn and an outlook is given in Section \ref{sec:Conclusions}.

\section{Some background from fluid dynamics}	
\label{sec:FluidDynamics}

In this section, we will review some classical concepts from fluid dynamics and put them into perspective with regard to their importance in the subsequent parts of this work.

\subsection{Velocity-equivalence of forces}	
\label{sec:VelocityEquivalence}

The dynamics of the incompressible Navier--Stokes equations is driven by its vorticity equation
\begin{equation}
	\vort_t -\nu \Delta \vort + \rb{\uu \ip \nabla} \vort 
  		= \nabla \times \ff + \rb{\vort \ip \nabla} \uu,	
\end{equation}
which is formally derived from \eqref{eq:transient:navier:stokes} by applying the curl operator to the momentum balance and substituting $\vort \coloneqq \nabla \times \uu$ \cite{chorin:marsden:1993}.
Due to $\nabla \times \nabla \phi = \zero$, the two forces $\ff$ and $\ff + \nabla \phi$ induce the same velocity field $\uu$, independent of the scalar potential $\phi$. 
This leads to an equivalence class of forces, where two forces will be called {\em velocity-equivalent} if they differ only by an arbitrary gradient field, i.e.,
\begin{equation} \label{eq:equiv:classes}
	\ff \simeq \ff + \nabla \phi.
\end{equation}
Indeed, the gradient part (in the sense of the Helmholtz--Hodge decomposition) of any force $\ff$ in the Navier--Stokes momentum balance only determines the pressure gradient $\nabla p$.
In Section \ref{sec:NSAndDiscretisation} this purely formal argument is made precise by introducing the Helmholtz--Hodge projector and a semi-norm that is connected to it.
Though the concept of the {\em velocity-equivalence} of forces is relevant for all forces in the Navier--Stokes momentum balance, this contribution will mainly focus on the consequences for the nonlinear convection term $\rb{\uu \ip \nabla} \uu$ at high Reynolds numbers.

\subsection{Generalised Beltrami flows}	
\label{sec:GenBeltramiFlows}

Velocity-equivalence of forces is especially relevant in a specific, but very rich and important class of transient incompressible flows, namely generalised Beltrami flows.
E.g., as far as we know, all known {\em exact solutions} (with $\ff=\zero$) \cite{drazin:riley:2006} of the incompressible Navier--Stokes equations are Galilean-invariant to generalised Beltrami flows.
Some of them will be used for our numerical benchmarks below.
Generalised Beltrami flows are those flows, whose nonlinear convection term is velocity-equivalent to a zero-force, i.e., it holds
\begin{equation*}
	\rb{\uu \ip \nabla} \uu \simeq \zero.	
\end{equation*}
Thus, their velocity solution is likewise the solution of an incompressible Stokes problem --- with a different pressure.
The main observation for the understanding of generalised Beltrami flows is the following pointwise identity for the nonlinear convection term: 
\begin{equation} \label{eq:nonlin:identity}
  \rb{\uu \ip \nabla} \uu 
  	= \rb{\nabla \times \uu} \times \uu + \frac{1}{2} \nabla \abs{ \uu }^2 
    = \vort \times \uu + \frac{1}{2} \nabla \abs{ \uu }^2,	
\end{equation}
where $\vort \times \uu$ is usually called the \emph{Lamb vector}.

Thus, generalised Beltrami flows can be subdivided into three different subclasses:
\begin{enumerate}
\item The most famous generalised Beltrami flows are classical potential flows with $\uu = \nabla h$, where $h$ denotes a (possibly time-dependent) harmonic potential fulfilling $-\Delta h = 0$.
Since potential flows are irrotational, it holds $\vort = \nabla \times \uu = \nabla \times \rb{\nabla h} = \zero$ and the nonlinear convection term is a gradient field
\begin{equation}
	\rb{\uu \ip \nabla} \uu = \frac{1}{2} \nabla \abs{ \uu }^2,
\end{equation}
and the nonlinear convection term is balanced by a pressure gradient $\nabla p= -\frac{1}{2} \nabla \abs{ \uu }^2$.
\item The second subclass consists of Beltrami flows. 
Contrary to potential flows, they are not irrotational, i.e., it holds $\vort \not= \zero$, however it  holds $\vort \times \uu = \zero$, i.e., the vorticity vector of Beltrami flows is parallel to the velocity field.
They exist only in the three-dimensional case, because the vorticity of two-dimensional flows is always perpendicular to the velocity field. 
Again, the pressure gradient is given by $\nabla p= -\frac{1}{2} \nabla \abs{ \uu }^2$.
\item Finally, for {\em generalised Beltrami flows} the vorticity is neither zero, nor parallel to the flow field, but the Lamb vector is a gradient field
\begin{equation}
	\vort \times \uu = \nabla \phi.	
\end{equation}
Here, the pressure gradient is different, namely $\nabla p= - \nabla \left ( \frac{1}{2} \abs{ \uu }^2 + \phi \right )$.
\end{enumerate}
It should be remarked that the vorticity equation of a generalised Beltrami flow (with $\ff=\zero$) is linear and given by
\begin{equation}
  \partial_t\vort - \nu \Delta \vort  =  \zero.
\end{equation}

Thus, a generalised Beltrami flow with $\ff=\zero$ and time-independent boundary conditions presents a nearly-steady behaviour over long time intervals, at least for small kinematic viscosities $\nu \ll 1$.
As a first connection to {\em vortex-dominated flows}, we remark that the slow decay of {\em vortex structures} like the 2D planar lattice flow problem is modelled by such flows \cite{Schroeder2019}. 
For such a process, a steady Eulerian description is sufficient.

\subsection{Galilean invariance and the material derivative}
\label{sec:GalileanFlows}
In this short subsection, we briefly want to discuss the role of the divergence-free part of the nonlinear convection term $\rb{\uu \ip \nabla} \uu$. 
It enters the game whenever a steady Eulerian description is not sufficient anymore.
~\\

Recalling that the incompressible Navier--Stokes equations are Galilean-invariant, we start from a generalised Beltrami flow $(\uu_0, p_0)$, fulfilling
\begin{equation*}
  \partial_t\uu_0 - \nu \Delta \uu_0 + \rb{\uu_0 \ip \nabla} \uu_0 + \nabla p_0  =  \zero, \qquad 
  	\DIV \uu_0 = 0,		
\end{equation*}	
and add a constant velocity field $\ww_0$ such that one obtains a new flow field
\begin{equation*}
	\uu(t, \xx) = \ww_0 + \uu_0(t, \xx - t \ww_0).	
\end{equation*}
Below, the corresponding pressure will be demonstrated to be
\begin{equation*}
	p(t, \xx) = p_0(t, \xx - t \ww_0).	
\end{equation*}
Then, one computes
\begin{equation*}
\partial_t \uu(t, \xx)  
	= \partial_t \sqb{ \uu_0(t, \xx - t \ww_0) } 
    = \partial_t \uu_0(t, \xx - t \ww_0)
    - \rb{\ww_0 \ip \nabla} \uu_0(t, \xx - t \ww_0)
\end{equation*}
and
\begin{subequations} \label{eq:transformed:nonlin:conv}
\begin{align}
  \rb{\uu(t, \xx) \ip \nabla} \uu(t, \xx) 
  	& = \rb{\ww_0 + \uu_0(t, \xx-t \ww_0) \ip \nabla} \uu_0(t, \xx-t \ww_0) \\
    & = \rb{\ww_0 \ip \nabla} \uu_0(t, \xx - t \ww_0)
    	+ \rb{\uu_0(t, \xx-t \ww_0) \ip \nabla} \uu_0(t, \xx-t \ww_0).	
\end{align}
\end{subequations}
Therefore, for the material derivative of $\uu$ it holds
\begin{equation*}
 	\frac{\Drm \uu(t, \xx)}{\Drm t} \coloneqq
   		\partial_t \uu(t, \xx) +  \rb{\uu(t, \xx) \ip \nabla} \uu(t, \xx)
  		=  \partial_t \uu_0(t, \xx - t \ww_0) + 
   			\rb{\uu_0(t, \xx - t \ww_0) \ip \nabla} \uu_0(t, \xx - t \ww_0),	
\end{equation*}
which is invariant under the Galilean transformation.
Since it further holds
\begin{equation*}
	-\nu \Delta \uu(t, \xx)=-\nu \Delta \uu_0(t, \xx - t \ww_0)	
\end{equation*}
and $\DIV \uu(t, \xx) = 0$, the pair $(\uu(t, \xx), p(t, \xx))$ does indeed fulfil the incompressible Navier--Stokes equations with $\ff=\zero$.
~\\

Besides the gradient part $\rb{\uu_0(t, \xx-t \ww_0) \ip \nabla} \uu_0(t, \xx-t \ww_0)$ (due to the generalised Beltrami property of $\uu_0$), the transformed nonlinear convection term \eqref{eq:transformed:nonlin:conv} does contain the new contribution $\rb{\ww_0 \ip \nabla} \uu  = \nabla \uu_0(t, \xx - t \ww_0) \ww_0$.
For this contribution, it holds
\begin{equation*}
 	\DIV \rb{ \nabla \uu_0(t, \xx - t \ww_0) \ww_0 }
  		= \DIV \rb{ \rb{\ww_0 \ip \nabla} \uu }
  	 	= \rb{\ww_0 \ip \nabla} (\DIV \uu) = 0.	
\end{equation*}
Thus, the nonlinear convection term of flows that are Galilean-invariant to a generalised Beltrami flow contains {\em both a  divergence-free and a gradient-field} part.
The corresponding vorticity equation remains linear, but contains an additional linear convection term
\begin{equation}
  \partial_t\vort - \nu \Delta \vort + \rb{\ww_0 \ip \nabla} \vort =  \zero.
\end{equation}
By means of an example in the next subsection we will demonstrate that the divergence-free part of the nonlinear convection term is responsible for the transport of geometric structures in the flow (like vortices, vortex filaments, \ldots), while the gradient field part prevents the dispersion of geometric structures, thereby ensuring conservation of mass.
Thus, the gradient field part of the nonlinear convection term is of major importance for incompressible high Reynolds number flows, although it represents a challenge for non-pressure-robust space discretisations.

\subsection{Steady solutions of the incompressible Euler equations, Galilean invariance and the Euler material derivative}
\label{sec:EulerFlows}

In this subsection we will slightly go beyond generalised Beltrami flows with $\ff=\zero$ and discuss the limit case $\mathrm{Re} \to \infty$, leading to the incompressible Euler equations with $\ff=\zero$; that is,
\begin{subequations} \label{eq:transient:euler}
	\begin{empheq}[left=\empheqlbrace]{alignat=2} 
  		\partial_t\uu + \rb{\uu \ip \nabla} \uu + \nabla p & =  \zero, \\
    	\DIV \uu & = 0.		
	\end{empheq} 
\end{subequations}
First, we want to remind the reader that vortices, vortex lines and vortex filaments are the building blocks of fluid dynamics \cite{chorin:marsden:1993}. 
Vortex-like solutions can be obtained as steady solutions of the incompressible Euler equations \eqref{eq:transient:euler}, for which it holds
\begin{equation*}
	\rb{\uu \ip \nabla} \uu = - \nabla p.	
\end{equation*}
Thus, every steady solution $\uu$ of the incompressible Euler equations has a nonlinear convection term which is velocity-equivalent to a zero force,
\begin{equation*}
	 \rb{\uu \ip \nabla} \uu \simeq \zero,	
\end{equation*}
in the sense of corresponding equivalence classes of forces as introduced above.
In the frictionless Euler setting, there exist even steady solutions with a compact support like the famous Gresho vortex, see Section \ref{sec:GreshoH1}.
Indeed, for steady solutions of the incompressible Euler equations with a compact support mainly the centrifugal force and the pressure gradient balance, similar to a tornado. 
Self-evidently, the centrifugal force is modelled by the (quadratic) convection term of the incompressible Euler and Navier--Stokes equations.
We remark that steady solutions of the incompressible Euler equations can contain an immense amount of kinetic energy, which is contained in {\em rotational degrees of freedom}, though.
This rotational kinetic energy can be unleashed, whenever vortices or vortex filaments interact with each other
or interact with the boundary of the domain.
~\\

Since all considerations from Subsection \ref{sec:GalileanFlows} about the Galilean invariance of the incompressible Navier--Stokes equations are valid for the Euler equations as well, the divergence-free part of the nonlinear convection term leads to a transport of structures; cf.\ the example of the Gresho vortex in Section \ref{sec:GreshoH1} in the case $\ww_0 \neq \zero$.
~\\

Moreover, looking at \eqref{eq:transient:euler}, we recognise that for the material derivative of incompressible Euler flows with $\ff=\zero$ it holds
\begin{equation}
   \frac{\Drm \uu}{\Drm t} =
     \partial_t\uu + \rb{\uu \ip \nabla} \uu = -\nabla p,
\end{equation}
i.e., for the Euler material derivative one obtains
\begin{equation}
  \frac{\Drm \uu}{\Drm t} \simeq \zero.
\end{equation}
Thus, strong forces of gradient field type are typical for incompressible Euler and Navier--Stokes flows at high Reynolds numbers, e.g., due to a force balance of the nonlinear centrifugal force and a strong, nontrivial pressure gradient.
~\\

The next subsection serves to illustrate that strong and `complicated' gradient fields in the Navier--Stokes momentum balance lead to numerical errors for {\em non-pressure-robust} space discretisations, while {\em pressure-robust} discretisations behave well.

\subsection{Hydrostatics: Complicated pressure}
\label{sec:ComplicatedPressure}

Let us explain in more detail, what we usually mean by `complicated pressures'.
The most important point is that `complicated' is always meant \emph{compared to the velocity}.
In this sense, a complicated pressure in a particular flow is always a relative concept.
~\\

In hydrostatics (flow at rest, no-flow), the pressure usually balances an external (gradient) force field, which makes it a perfect example for presenting `complicated pressures'.
For ease of presentation, suppose we want to solve the incompressible Stokes problem ($\nu=1$), with a right-hand side forcing term $\ff=\nabla\phi / \int_\Omega \phi$.
Here, the normalisation is made to ensure comparable situations for different potentials $\phi$.
The left-hand side column in Figure~\ref{fig:no-flow-stokes} shows different potentials $\phi=y^\gamma$ for $\gamma=1,2,4,9$ in a domain $\Omega$ which resembles a glass geometry.
Note that the pressure behaves analogously as these potentials and thus, they can be considered `complicated' compared to the exact velocity solution $\uu=\zero$ in hydrostatics problems.
~\\

\begin{figure}[h]
\centering
	\includegraphics[width=0.9\textwidth]{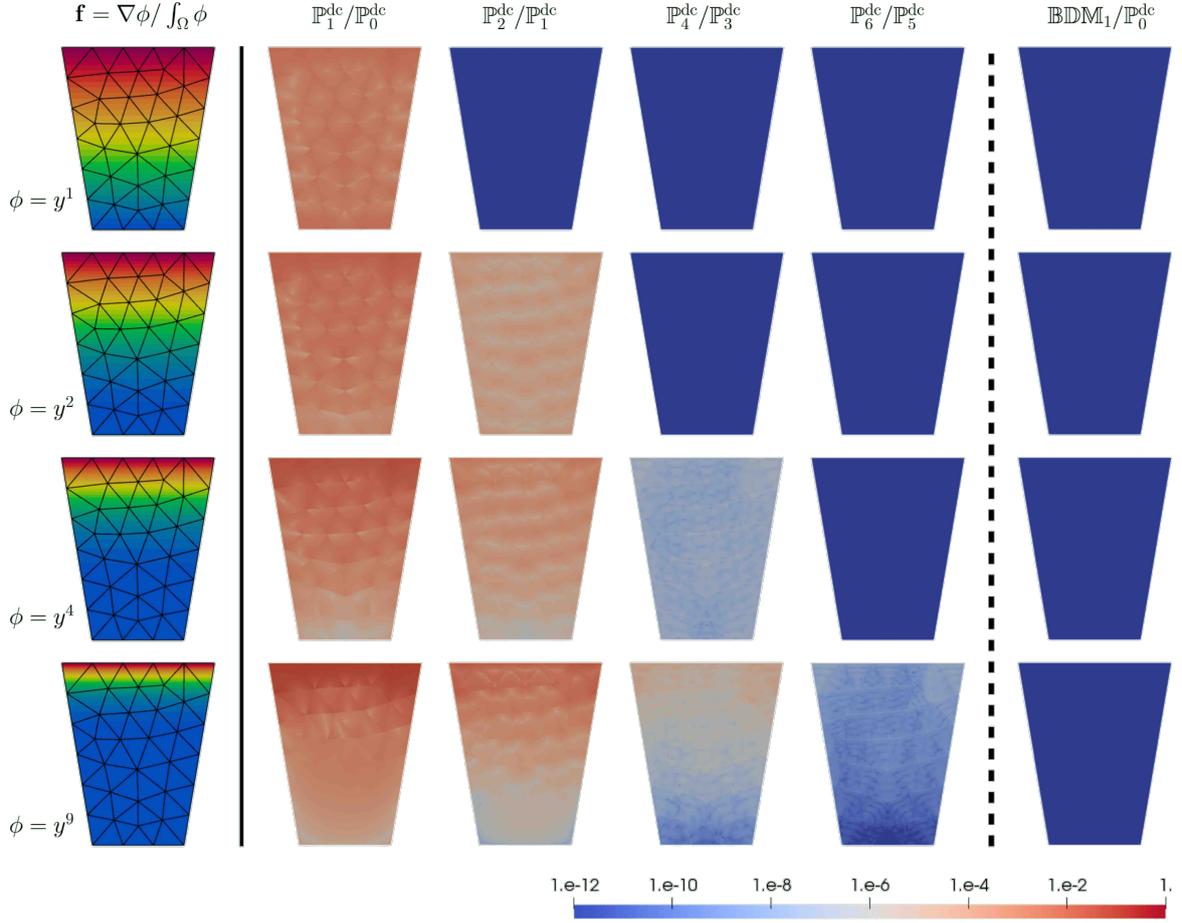} 
\caption{Stokes no-flow problem in a glass, demonstrating the concept of complex pressures and the advantages of using pressure-robust methods. The left-hand side column shows the potential $\phi=y^\gamma$ for $\gamma=1,2,4,9$ and the underlying triangular mesh for all computations. The other columns show velocity magnitude $\abs{\uu_h}$ for the non-pressure-robust $\protect\PPdk{k}{}/\Pdk{k-1}{}$ method and the pressure-robust and divergence-free $\protect\BDMk{1}{}/\Pdk{0}{}$ method.}
\label{fig:no-flow-stokes}
\end{figure}

Now, all other plots in Figure~\ref{fig:no-flow-stokes} show the velocity solution of different numerical methods for the particular problem.
The chosen methods are $\LTWO$- ($\PPdk{k}{}/\Pdk{k-1}{}$) and $\HDIV$-based ($\BDMk{1}{}/\Pdk{0}{}$) DG methods on triangular meshes of different order $k$ (polynomial order of the discrete velocity approximations), 
In the present context, it suffices to know that the former are non-pressure-robust whereas the latter are pressure-robust and divergence-free; cf.\ Section~\ref{sec:dGFEM} for more details.
~\\

One can see that the low-order ($k=1$) pressure-robust method computes the correct velocity solution $\uu_h=\zero$ \emph{independent} of the pressure/potential, even though the discrete pressure space only consists of piecewise constants.
The non-pressure-robust method, on the other hand, only leads to $\uu_h=\zero$ if $k-1\geqslant \gamma$, as in this situation the pair $\rb{\uu,p}$ is contained in the discrete FE spaces.
Moreover, one can see that whenever the non-pressure-robust method gives $\uu_h\neq\zero$, the quality of the solution decreases as $\gamma$ increases, i.e.\ as the pressure becomes more and more complicated.
On the other hand, increasing the order $k$ of the discretisation improves the solution; this is simply $k$-convergence.

\section{Time-dependent Navier--Stokes problem and ${\emph{\textbf{H}}^{\mathbf{1}}}$ discretisation} 
\label{sec:NSAndDiscretisation}

After a very brief introduction to the governing equations on the continuous level, we introduce the spatial $\HM{1}{}$-conforming discretisation schemes which will be used for the error analysis in the first part of this work.
They consist of an exactly divergence-free, pressure-robust method and a classical non-pressure-robust FEM.

\subsection{Infinite-dimensional Navier--Stokes equations}

We consider the time-dependent incompressible Navier--Stokes problem, which reads
\begin{subequations}\label{eq:TINS}
	\begin{empheq}[left=\empheqlbrace]{alignat=2} 
		\partial_t\uu - \nu\Delta \uu + \rb{\uu\ip\nabla}\uu +\nabla p &= \ff \qquad\quad 												&&\text{in }\rsb{0,\tend}\times\Omega,	\\
		\DIV\uu &= 0 				&&\text{in }\rsb{0,\tend}\times\Omega, 			\\
		\uu\rb{0,\xx} &=\uu_0\rb{\xx} 	&&\text{for } \xx\in\Omega.		
	\end{empheq} 
\end{subequations}

For the space dimension $d\in\set{2,3}$, $\Omega\subset\R^d$ denotes a connected bounded Lipschitz domain and $\tend$ is the end of time considered in the particular problem.
Since in the numerical analysis below we want to compare the best possible convergence rates for pressure-robust and classical space discretisations in the $\LPOne{2}{}$-norm, we will assume for technical reasons that $\Omega$ is convex, leading to elliptic regularity.
Moreover, $\uu \colon\sqb{0,\tend}\times\Omega\to\R^d$ indicates the velocity field, $p\colon\sqb{0,\tend}\times\Omega\to\R$ is the (zero-mean) kinematic pressure, $\ff\colon\sqb{0,\tend}\times\Omega\to\R^d$ represents external body forces and $\uu_0\colon \Omega\to\R^d$ stands for a suitable initial condition for the velocity. 
The underlying fluid is assumed to be Newtonian with constant (dimensionless) kinematic viscosity $0<\nu\ll 1$. 
We impose either the general Dirichlet boundary condition $\uu=\gD$ on $\rsb{0,\tend}\times\partial\Omega$, or periodic boundary conditions (or a mixture of them).

\paragraph{Notation}
In what follows, for $K\subseteq\Omega$ we use the standard Sobolev spaces $\Wmp{m}{p}{\rb{K}}$ for scalar-valued functions with associated norms $\norm{\cdot}_{\Wmp{m}{p}{\rb{K}}}$ and seminorms $\abs{\cdot}_{\Wmp{m}{p}{\rb{K}}}$ for $m\geqslant 0$ and $p\geqslant 1$. 
We obtain the Lebesgue space $\Wmp{0}{p}{\rb{K}}=\Lp{p}{\rb{K}}$ and the Hilbert space $\Wmp{m}{2}{\rb{K}}=\Hm{m}{\rb{K}}$. 
Additionally, the closed subspaces $\Hmz{1}{\rb{K}}$ consisting of $\Hm{1}{\rb{K}}$-functions with vanishing trace on $\partial K$ and the set  $\Lpz{2}{\rb{K}}$ of $\Lp{2}{\rb{K}}$-functions with zero mean in $K$ play an important role. 
The $\Lp{2}{\rb{K}}$-inner product is denoted by $\rb{\cdot,\cdot}_K$ and, if $K=\Omega$, we sometimes omit the domain completely when no confusion can arise. 
Furthermore, with regard to time-dependent problems, given a Banach space $\XX$ and a time instance $t$, the Bochner space $\Lp{p}{\rb{0,t;\XX}}$ for $p\in\sqb{1,\infty}$ is used.
In the case $t=\tend$, we frequently use the abbreviation $\Lp{p}{\rb{\XX}}=\Lp{p}{\rb{0,\tend;\XX}}$. 
Further, $C^1(0, t; \XX)$ denotes the function space mapping $\sqb{0, t}$ into $\XX$, which is continuously differentiable in time w.r.t.\ the norm  $\norm{\uu}_{C^1(0, t; \XX)} \max_{s \in  \sqb{0,\tend}} (\norm{\uu}_\XX + \norm{\uu_t}_\XX)$.
Spaces and norms for vector- and tensor-valued functions are indicated with bold letters. 
For example, for a vector-valued function $\vv=\rb{v_1,\dots,v_n}^\dag$, we consider $\norm{\vv}_\LP{p}{\OMEGA}^p=\sum_{i=1}^n \norm{v_i}_\Lp{p}{\OMEGA}^p =\int_\Omega \abs{\vv}_p^p \dx$, where $\abs{\vv}_p^p=\sum_{i=1}^n \abs{v_i}^p$. 
The vorticity of a 2D velocity field $\uu=\rb{u_1,u_2}^\dag$ is defined as $\omega =\partial_{x_1}u_2 - \partial_{x_2}u_1$.
~\\

Depending on the particular boundary conditions, let $\VV/\Q$ be the continuous solution spaces for velocity and pressure, respectively.
Note that it holds $\VV\subset\HM{1}{\OMEGA}$ and $\Q\subset\Lp{2}{\OMEGA}$.
For the numerical analysis in this and the next section, we will always choose $\VV=\HH^1_0\OMEGA$ and $Q_h=L^2_0\OMEGA$.
The subspace of weakly divergence-free functions is defined as
\begin{equation*}
	\VV^\dvg 
		= \set{\vv\in\VV\colon \rb{q,\DIV\vv}=0,~\forall\,q\in\Q}.
\end{equation*}

A weak velocity solution $\uu \in \Lp{2}{\rb{0,\tend;\VV^\dvg}}$ of \eqref{eq:TINS} fulfils that for all test functions $\vv \in \VV^\dvg$ holds
\begin{equation} \label{eq:weak:NSE}
  \frac{\drm}{\drm t} \rb{\uu(t), \vv}
  	+ \nu \rb{\nabla \uu(t), \nabla \vv}
  	+	 \rb{(\uu(t) \ip \nabla) \uu(t), \vv}
  	= \bra{\ff(t), \vv}_{\HH^{-1}, \HH^1_0}
\end{equation}

in the sense of distributions in $\mathcal{D}'(]0, T[)$ and such that $\uu(0) = \uu_0$ \cite{BoyerFabrie13}.
Note that the pressure $p$ is not part of the weak formulation of the incompressible Navier--Stokes problem, see Remark \ref{rem:weak:eq:forces}. 
For the numerical analysis, we will further assume the regularity $\uu \in \Lp{1}{\rb{\WMP{1}{\infty}{}}}$, ensuring, e.g., uniqueness of the weak solution in time \cite{bt:2013, SchroederLube17b}. 
Further (technical) regularity assumptions will be made at appropriate places in the contribution.
Then, $(\uu, p)$ fulfils
\begin{subequations} \label{eq:WeakContinuous}
	\begin{empheq}[left=\empheqlbrace]{align} 
	\text{Find }\rb{\uu,p}&\colon\rsb{0,\tend}\to\VV\times\Q
		\text{ with }\uu\rb{0}=\uu_{0}
		\text{ s.t., }\forall\,\rb{\vv,q}\in\VV\times\Q,\\
	\rb{\partial_t\uu,\vv}
		&+\nu \rb{\nabla\uu,\nabla\vv} 
		+ \rb{\rb{\uu\ip\nabla}\uu,\vv}
		- \rb{p,\DIV\vv}
		+ \rb{q,\DIV\uu}
		= \rb{\ff,\vv}.
	\end{empheq} 
\end{subequations}

\subsection{Helmholtz--Hodge decomposition in $\textbf{L}^\textbf{2}$}

In order to understand the significance of pressure-robustness for the discretisation theory of the incompressible Navier--Stokes equations \eqref{eq:TINS}, the concept of the Helmholtz--Hodge projector is introduced. 
Since the numerical analysis below is essentially an $\LP{2}{}$ analysis (assuming all forces $\partial_t\uu(t)$, $(\uu(t) \ip \nabla) \uu(t)$, \ldots to be in $\LP{2}{}$), we will restrict our considerations to the Helmholtz--Hodge decomposition in $\LP{2}{}$.
Below in this section, some functional analytic prerequisites are summarised that show that only the divergence-free parts, i.e., the Helmholtz--Hodge projectors of the forces in the Navier--Stokes momentum balance influence the velocity solution of the incompressible Navier--Stokes equations, see also \cite{JLMNR:sirev}. 
~\\

The space of square-integrable divergence-free (solenoidal) vector fields is defined by
\begin{equation} \label{eq:def:lp:sigma}
  \LPOne{2}_{\sigma}\OMEGA 
  	\coloneqq \set{ \ww \in \LPOne{2}\OMEGA\colon -\rb{\ww, \nabla \phi} = 0,~\forall\,\phi \in H^1(\Omega) }.
\end{equation}

\begin{remark} \label{rem:helm:bcs}
First note that for $\phi \in C^\infty_0\OMEGA$ the mapping $\phi \mapsto -(\ww, \nabla \phi)$ denotes the distributional divergence of $\ww$.
Thus, vector fields in $\LPOne{2}_{\sigma}\OMEGA$ are divergence-free \cite{JLMNR:sirev}.
Further note that definition \eqref{eq:def:lp:sigma} implies that $\restr{\ww\ip\nn}{\partial\Omega}=0$, where $\nn$ denotes the outer unit normal vector on $\partial\Omega$, since test functions $\phi \in H^1$ do not vanish on the boundary of $\Omega$.
In this context, please also note that a Helmholtz--Hodge decomposition is made unique only by prescribing certain boundary conditions. 
The reason is that any gradient of a harmonic function with $-\Delta h = 0$ is irrotational and divergence-free at the same time. 
Thus, the boundary conditions determine whether $\nabla h$ is called 'divergence-free' or 'gradient field'. 
In our setting, all gradients of harmonic functions are called 'gradient fields', and vector fields in $\LPOne{2}_{\sigma}\OMEGA$ are orthogonal to all gradient fields in $\LPOne{2}$.
\end{remark} 

\begin{remark} \label{rem:lp}
Our considerations regard the Helmholtz--Hodge decomposition in $\LP{2}{}$  of $(\uu \ip \nabla) \uu$. Since we assume that  $\uu \in \WW^{1, \infty}$, it holds $(\uu \ip \nabla) \uu \in \LP{2}{}$.
\end{remark}

Due to the special choice of the boundary conditions and Remark \ref{rem:helm:bcs} one obtains the following theorem, for which the proof will be repeated for completeness and readability of the manuscript from \cite{JLMNR:sirev}:
\begin{theorem}[Helmholtz--Hodge decomposition in $\textbf{L}^\textbf{2}$]
For every vector field $\vv \in \LP{2}{\OMEGA}$, there exists a unique Helmholtz--Hodge decomposition 
\begin{equation} \label{helmholtz:hodge:decomposition}
	\vv = \ww + \nabla \psi,
\end{equation}
where it holds $\ww \in \LPOne{2}_{\sigma}(\Omega)$, and $\psi \in H^1\OMEGA$ and $\ww$ and $\nabla \psi$ are $\LPOne{2}$-orthogonal.
Then, $\ww \eqqcolon \helm(\vv)$ is called the Helmholtz--Hodge projector of $\vv$. 
\end{theorem}

\begin{proof} A potential $\psi \in H^1(\Omega) / \mathbb{R}$ in the Helmholtz--Hodge decomposition is obtained by: for all $\chi \in H^1(\Omega) / \mathbb{R}$ holds
\begin{equation} \label{eq:weak:helm:problem}
  (\nabla \psi, \nabla \chi) = (\vv, \nabla \chi).
\end{equation}

This Neumann problem for $\psi$ is uniquely solvable \cite{JLMNR:sirev}.
Then, define $\ww \coloneqq \vv - \nabla \psi$.
One can test $\ww$ with arbitrary gradient fields $\nabla (\phi + C)$, where $C$ denotes an arbitrary real number and where it holds $\phi \in H^1\OMEGA/\mathbb{R}$.
Then, one obtains $(\ww, \nabla (\phi + C)) = (\ww, \nabla \phi)= (\vv - \nabla \psi, \nabla \phi) = 0$, due to \eqref{eq:weak:helm:problem}.
Thus, it holds $\ww \in \LPOne{2}_{\sigma}\OMEGA$.
Due to the definition of $\LPOne{2}_{\sigma}\OMEGA$, $\ww$ and $\nabla \psi$ are orthogonal in $\LPOne{2}$.
Assuming that $\vv = \ww_1 + \nabla \psi_1=\ww_2 + \nabla \psi_2$ are two Helmholtz--Hodge decompositions of $\vv$, then it holds
\begin{equation*}
  \ww_1 - \ww_2 = \nabla (\psi_2 - \psi_1)	
\end{equation*}

with $\ww_1 - \ww_2 \in \LPOne{2}_{\sigma}\OMEGA$.
Testing this equality by $\nabla (\psi_1-\psi_2)$ yields by the $\LPOne{2}$-orthogonality of \eqref{eq:weak:helm:problem}
\begin{equation*}
  \norm{\nabla (\psi_2 - \psi_1)}^2_{\LPOne{2}\OMEGA}=0,
\end{equation*}

and one concludes $\ww_1=\ww_2$ and $\psi_1=\psi_2$ using $\psi_1, \psi_2 \in H^1(\Omega)/\mathbb{R}$.
Thus, the Helmholtz--Hodge decomposition is unique.
\end{proof}

\begin{remark}
Formally, the Helmholtz--Hodge decomposition of $\vv\in\LP{2}{\OMEGA}$ can be written as the solution of the PDE problem
\begin{subequations} \label{eq:HelmholtzLeray}
	\begin{empheq}[left=\empheqlbrace]{alignat=2} 
		 \HL{\vv} + \nabla \psi &= \vv \qquad\quad 												&&\text{in }\Omega,	\\
		\DIV\HL{\vv} &= 0 				&&\text{in }\Omega, 			\\
		\HL{\vv}\ip\nn &=0 	&&\text{on } \partial\Omega.		
	\end{empheq} 
\end{subequations}	
\end{remark}

The most important property of the Helmholtz--Hodge projector for our contribution is given as follows:
\begin{lemma} \label{lem:helm:vanishes:for:gradients}%
For all $\psi \in H^1\OMEGA$,
it holds
\begin{equation*}
	\helm(\nabla \psi) = \zero.	
\end{equation*}
\end{lemma}

\begin{proof}
Note that $\nabla \psi = \zero + \nabla \psi$ is the unique Helmholtz--Hodge decomposition of $\nabla \psi$. 
Thus, it follows $\helm(\nabla \psi) = \zero$.
\end{proof}

Finally, it is emphasised that the velocity solution $\uu$ of \eqref{eq:weak:NSE} is completely determined by testing the momentum equation with divergence-free velocity test functions $\vv \in \VV^\dvg$ and by its initial value $\uu_0$.
Assuming smoothness of $\uu$ in space and time, $\uu$ fulfils for all $\vv \in \VV^\dvg$
\begin{subequations}
\begin{align} \label{eq:helm:equiclasses}
  	\rb{\partial_t\uu,\vv}
		&+\nu \rb{\nabla\uu,\nabla\vv} 
		+ \rb{\rb{\uu\ip\nabla}\uu,\vv}
		 = \rb{\ff,\vv} \qquad \Leftrightarrow \\
  &\rb{\helm(\partial_t\uu),\vv}
		-\nu \rb{\helm(\Delta \uu), \vv} 
		+ \rb{\helm(\rb{\uu\ip\nabla}\uu),\vv}
		 = \rb{\helm(\ff),\vv}.
\end{align}	
\end{subequations}

\begin{remark} \label{rem:weak:eq:forces}
Equation \eqref{eq:helm:equiclasses} shows that the velocity solution $\uu$ of the incompressible Navier--Stokes equations is not determined by the forces $\ff$ in the momentum equation themselves, but by their Helmholtz--Hodge projectors $\helm(\ff)$. 
Therefore, two forces $\ff$ and $\gbld$ that differ by a gradient field $\ff=\gbld + \nabla \phi$, lead to the same velocity solution. 
Thus, the velocity solution of the incompressible  Navier--Stokes equations is naturally determined by equivalence classes of forces, where it holds
\begin{equation*}
\ff \simeq \gbld \qquad \Leftrightarrow \qquad
   \helm(\ff) = \helm(\gbld).	
\end{equation*}
\end{remark}

\subsection{$\emph{\textbf{H}}^{\mathbf{1}}$ finite element methods}
\label{sec:H1FEM}

Let $\VV_h/\Q_h$ be the considered discretely inf-sup stable velocity/pressure FE pair, where $\VV_h\subset\VV$ and $\Q_h\subset\Q$.
We assume that the discrete velocity space contains polynomials up to degree $\ku$ and the discrete pressure space contains polynomials up to degree $k_p$. 
Note that for most discretely inf-sup stable schemes it holds $k_p = \ku-1$. 
An example is the Taylor--Hood finite element family $\PPk{k}{}/\Pk{k-1}{}$ with $k \geqslant 2$ \cite{John16}. 
However, for the famous mini element it holds $k_p=\ku (= 1)$ \cite{abf:1984}.
In the following numerical analysis, $C>0$ always denotes a generic constant, whose value is independent of the mesh size but possibly dependent on the mesh-regularity.
~\\

In order to approximate \eqref{eq:TINS}, or equivalently \eqref{eq:WeakContinuous}, the following generic semi-discrete FEM is considered:
\begin{subequations} \label{eq:H1-FEM}
	\begin{empheq}[left=\empheqlbrace]{align} 
	\text{Find }\rb{\uu_h,p_h}\colon\rsb{0,\tend}\to\VV_h\times\Q_h
		\text{ with }\uu_h\rb{0}=\uu_{0h}
		\text{ s.t., }\forall\,\rb{\vv_h,q_h}\in\VV_h\times\Q_h&,\\
	\rb{\partial_t\uu_h,\vv_h}
		+\nu \rb{\nabla\uu_h,\nabla\vv_h} 
		+ \rb{\rb{\uu_h\ip\nabla}\uu_h + \frac{1}{2}\rb{\DIV\uu_h}\uu_h,\vv_h}
		- \rb{p_h,\DIV\vv_h}
		&= \rb{\ff,\vv_h} \\
	\rb{q_h,\DIV\uu_h} &= 0.
	\end{empheq} 
\end{subequations}
Now, the choice of the FE spaces decides whether an exactly divergence-free, pressure-robust method is applied or not. 
For example, the Scott--Vogelius element $\PPk{k}{}/\Pdk{k-1}{}$ is discretely inf-sup stable on shape-regular, barycentrically refined meshes for $k \geqslant d$ \cite{zhang:2005}, or on meshes without singular vertices for $k\geqslant 2d$; then it yields an exactly divergence-free and thus pressure-robust method.
On the other hand, for example, a classical $\PPk{k}{}/\Pk{k-1}{}$ Taylor--Hood element is a non-pressure-robust method.
Note that the explicit skew-symmetrisation of the convective term is only necessary for a non-divergence-free FEM.
~\\

The subspace of discretely divergence-free functions is given by
\begin{align*}
	\VV_h^\dvg 
		= \set{\vv_h\in\VV_h\colon (q_h,\DIV\vv_h)=0,~\forall\,q_h\in\Q_h},
\end{align*}

where we note that for exactly divergence-free (and thus pressure-robust) FEM, $\vv_h\in\VV_h^\dvg$ follows $\DIV\vv_h=0$, i.e., $\VV_h^\dvg \subset \LPOne{2}_{\sigma}\OMEGA$.
In this context, a frequently used tool in finite element error analysis is the discrete Stokes projector, defined by
\begin{align*}
	&\stokes_h\colon\VV^\dvg \to \VV_h^\dvg, \quad
		\stokes_h(\vv) = \argmin_{\vv_h\in\VV_h^\dvg} \norm{\nabla\rb{\vv-\vv_h}}_\LP{2}{\OMEGA}, \quad
		\int_\Omega \nabla\sqb{\stokes(\vv)-\vv}\Fip\nabla\vv_h \dx =0,~\forall\,\vv_h\in\VV_h^\dvg.
\end{align*}

The Stokes projector possesses optimal approximation properties due to discrete inf-sup stability \cite{alm:2018}.
Last but not least, the Lagrange interpolation into the $H^1$-conforming subspace of the discrete pressure-space $Q_h$ is denoted by
\begin{equation} \label{eq:lagrange:press:interpol}
  L_h: C(\bar{\Omega}) \to Q_h \cap H^1(\Omega).
\end{equation}

\begin{remark} \label{rem:h1:press:approx}
For discrete discontinuous pressure spaces with $k_p \geqslant 1$ it holds for all $q \in H^{k_p+1}$
\begin{equation*}
  \norm{\nabla (q - L_h q)}_{L^2} \leqslant  C h^{k_p}
   \abs{q}_{H^{k_p+1}\OMEGA},	
\end{equation*}

where $C$ does only depend on the shape-regularity of the triangulation.
\end{remark}

\begin{lemma}[Convergence of the nonlinear convection term as $h \to 0$] \label{lem:conv:nonlinear:term}%
Assume that $\uu_h \to \uu \in \WW^{1, \infty}(\Omega)$ converges strongly in $\HH^{1}(\Omega)$ and that $\norm{\uu_h}_{\WW^{1, \infty}(\Omega)} \leqslant  C$ is uniformly bounded.
Then, $(\uu_h \ip \nabla) \uu_h \to (\uu \ip \nabla) \uu$ converges strongly in $\LPOne{2}(\Omega)$.
Further, it holds that $\helm((\uu_h \ip \nabla) \uu_h) \to \helm((\uu \ip \nabla) \uu)$ converges strongly for the Helmholtz--Hodge projector in $\LPOne{2}(\Omega)$.
\end{lemma}

\begin{proof}
One can derive
\begin{align*}
  \norm{ (\uu_h \ip \nabla) \uu_h - (\uu \ip \nabla) \uu }_{\LPOne{2}(\Omega)}
    & \leqslant  
 \norm{ ((\uu_h - \uu) \ip \nabla) \uu_h}_{\LPOne{2}(\Omega)}
 + \norm{ (\uu \ip \nabla) (\uu_h - \uu)}_{\LPOne{2}(\Omega)} \\
 & \leqslant 
 \norm{ \uu_h - \uu }_{\LPOne{2}(\Omega)} \norm{ \nabla \uu_h }_{\LPOne{\infty}(\Omega)}
 + \norm{ \uu }_{\LPOne{\infty}(\Omega)} \norm{ \nabla (\uu_h - \uu)}_{\LPOne{2}(\Omega)}.
\end{align*}

Due to $\uu_h \to \uu$ strongly in $\HH^1(\Omega)$,  $\norm{ \uu_h - \uu }_{\LPOne{2}(\Omega)}$ and $\norm{ \nabla (\uu_h - \uu)}_{\LPOne{2}(\Omega)}$ converge to zero.
Further, $\norm{\nabla \uu_h}_{\LPOne{\infty}(\Omega)}$ and $\norm{\uu}_{\LPOne{\infty}(\Omega)}$ are assumed to be bounded.
The convergence  of the Helmholtz--Hodge projectors $\helm((\uu_h \ip \nabla) \uu_h) \to \helm((\uu \ip \nabla) \uu)$ in $\LPOne{2}(\Omega)$ is an immediate consequence of the well-posedness of the Neumann problem \eqref{eq:weak:helm:problem} in $H^{1}(\Omega) / \mathbb{R}$.
\end{proof}

\begin{remark}[Pseudo-dominant convection] \label{rem:pseudo:dominant:convection}
When confronted with generalised Beltrami flows at high Reynolds numbers, Lemma \ref{lem:conv:nonlinear:term} is decisive in order to understand the behaviour of space discretisations of the incompressible Navier--Stokes equations.
While $\helm( (\uu \ip \nabla) \uu ))$ vanishes for generalised Beltrami flows since $(\uu \ip \nabla) \uu$ is a gradient field, on the discrete level $\helm( (\uu_h \ip \nabla) \uu_h )) \to \zero$ holds since $(\uu_h \ip \nabla) \uu_h$ \emph{only converges} to a gradient field as $h\to 0$.
Thus, one can observe some kind of \emph{pseudo-dominant} convection at high Reynolds numbers \cite{cmame:linke:merdon:2016}, i.e., the infinite-dimensional generalised Beltrami problem is not convection-dominated due to $\helm( (\uu \ip \nabla) \uu )) = \zero$, but the discretised problem experiences some non-negligible, artificial convective force.
A similar effect can be observed also for the linear Stokes problem when one uses numerical quadratures in the discretisation of forces of gradient fields, see \cite[Subsection 6.2]{lmn:2018}.
\end{remark}

\subsection{Discrete $\emph{\textbf{H}}^{\mathbf{1}}$-conforming Helmholtz--Hodge projector}
\label{sec:DiscreteH1Projectors}

A discrete Helmholtz--Hodge projector in $\LTWO$ is defined straightforward as the $\LTWO$ projector onto $\VV_h^\dvg$:
\begin{equation} \label{eq:disc:helm}
  \helm_h\colon\LP{2}{\OMEGA} \to \VV_h^\dvg, \quad
		\helm_h(\vv) = \argmin_{\vv_h\in\VV_h^\dvg} \norm{\vv-\vv_h}_\LP{2}{\OMEGA}
		, \quad
		\int_\Omega \sqb{\helm_h(\vv)-\vv}\ip\vv_h \dx =0,~\forall\,\vv_h\in\VV_h^\dvg.
\end{equation}

\begin{remark}
Under the assumptions of elliptic regularity of $\Omega$, shape-regular meshes and discrete inf-sup stability of the method, the corresponding discrete Helmholtz projector has optimal approximation properties	
\begin{equation}
	\norm{\uu-\helm_h(\uu)}_\LP{2}{\OMEGA}
		+ h\norm{\nabla\sqb{\uu-\helm_h(\uu)}}_\LP{2}{\OMEGA}
		\leqslant  C h^{k+1} \norm{\uu}_\HM{k+1}{\OMEGA}.
\end{equation}

The proof follows directly from \cite[Lemma 11]{alm:2018}.
\end{remark}

\begin{lemma}[$\WMP{1}{\infty}{}$ stability of discrete Helmholtz--Hodge projector] \label{lem:W1inftyHelmholtz}%
	Assume elliptic regularity of $\Omega$, shape-regular meshes, discrete inf-sup stability of the method and that the exact solution $\uu$ is sufficiently smooth. 
	Then, the corresponding discrete Helmholtz--Hodge projector fulfils 
	\begin{equation}
		\norm{\nabla\helm_h(\uu)}_\LP{\infty}{\OMEGA}
			\leqslant  C\norm{\nabla\uu}_\LP{\infty}{\OMEGA} 
				+ h^{\ku-\nf{d}{2}} \norm{\uu}_\HM{\ku+1}{\OMEGA},
	\end{equation}
	
	where $C>0$ is independent of $h$ and $\ku$ denotes the polynomial order of discrete velocities in $\VV_h$.
\end{lemma}

\begin{proof}
The first step is to use the Stokes projector and the triangle inequality to obtain
	\begin{equation*}
		\norm{\nabla \helm_h(\uu)}_\LP{\infty}{\OMEGA}
			\leqslant  \norm{\nabla \stokes_h(\uu)}_\LP{\infty}{\OMEGA} 
				+ \norm{\nabla \helm_h(\uu) - \nabla \stokes_h(\uu)}_\LP{\infty}{\OMEGA}.
	\end{equation*}
	
Note that the $\WMP{1}{\infty}{}$ stability of the Stokes projector has been shown in \cite{GiraultEtAl15}.
For shape-regular decompositions $\T$, the discrete space $\VV_h$ (and thus also $\VV_h^\dvg$) satisfies the local inverse inequality \cite[Lemma 1.138]{ErnGuermond04} 
\begin{equation} \label{eq:LocInvEq}
	\forall\vv_h\in\VV_h\colon\quad
	\norm{\vv_h}_{\WMP{\ell}{p}{\rb{K}}}
	\leqslant  C_\inv h_K^{m-\ell+d\rb{\frac{1}{p}-\frac{1}{q}}}\norm{\vv_h}_{\WMP{m}{q}{\rb{K}}}, \quad
	\forall K\in\T,
\end{equation}

where $0\leqslant  m\leqslant  \ell$ and $1\leqslant  p,q\leqslant \infty$.
Choosing $\ell=m=1$, $p=\infty$ and $q=2$, the inverse estimate can be applied to further estimate the right-hand side as
\begin{align*}
	\norm{\nabla \helm_h(\uu) - \nabla \stokes_h(\uu)}_\LP{\infty}{\OMEGA}
		&\leqslant  C_\inv h^{-\nf{d}{2}} \norm{\nabla \sqb{\helm_h(\uu) - \stokes_h(\uu)}}_\LP{2}{\OMEGA} \\
		&\leqslant  C_\inv h^{-\nf{d}{2}} \sqb{
		\norm{\nabla \sqb{\uu - \helm_h(\uu)}}_\LP{2}{\OMEGA}
			+ \norm{\nabla \sqb{\uu - \stokes_h(\uu) }}_\LP{2}{\OMEGA}
		} \\
		&\leqslant  C h^{\ku-\nf{d}{2}} \norm{\uu}_\HM{\ku+1}{\OMEGA},
\end{align*}

where the optimal approximation properties of both $\helm_h$ and $\stokes_h$ are essential.
\end{proof}

Combining Lemmas \ref{lem:W1inftyHelmholtz} and \ref{lem:conv:nonlinear:term} yields a result which is essential for good convergence properties of the Galerkin method on pre-asymptotic meshes, see also Theorems \ref{thm:pr} and \ref{thm:classical}:
\begin{remark} \label{rem:p:uhnablauh:conv}
Assuming $\ku \geqslant d/2$, it holds for $h \to 0$ that
\begin{equation*}
  \helm( (\helm_h(\uu) \ip \nabla) \helm_h(\uu))
   \stackrel{\LPOne{2}{}}{\to}
    \helm((\uu \ip \nabla) \uu).	
\end{equation*}
\end{remark}

Now, let us first consider the situation where $\helm_h$ belongs to a pressure-robust (divergence-free) method.
\begin{lemma} \label{lem:DiscHelmDivFreeH1}%
For pressure-robust (divergence-free) $\HM{1}{}$ methods, for all $\psi \in \Hm{1}{\OMEGA}$ it holds
\begin{equation*}
  \helm_h(\nabla \psi) = \zero.	
\end{equation*}
\end{lemma}

\begin{proof}
Since for divergence-free methods $\DIV\vv_h=0$ holds for all $\vv_h \in \VV_h^\dvg$, one obtains
\begin{equation*}
	(\nabla \psi, \vv_h) = -(\psi, \DIV \vv_h) = 0,	
\end{equation*}
for all $\vv_h \in \VV_h^\dvg$. 
\end{proof}

Finally, for non-pressure-robust methods, the situation is not exactly the same as in the infinite-dimensional case.
In fact, for the steady Navier--Stokes problem, i.e., for the elliptic problem, one has to estimate the consistency error of $\helm_h(\nabla \phi)$ in a discrete $\HMh{-1}{}$ norm, which yields an $\mathcal{O}(h^{k_p+1})$ consistency error \cite{lm:2018}, where $k_p$ denotes the formal approximation order of the discrete pressure space in the $L^2$ norm.
However, for the fully time-dependent \emph{a priori} error analysis below we will have to estimate this consistency error in the stronger $\LP{2}{\OMEGA}$ norm, which was seemingly done for the first time in \cite{lr:2018}.

\begin{lemma} \label{lem:DiscHelmLerayNonDivFreeH1}%
For non-pressure-robust $\HM{1}{}$ methods with $k_p \geqslant 1$, it holds for all gradient fields $\nabla \psi$ with $\psi \in \Hm{k_p+1}{\OMEGA}$
\begin{equation}
	\norm{\helm_h(\nabla \psi)}_\LP{2}{\OMEGA} 
		\leqslant  C h^{k_p} \abs{\psi}_\Hm{k_p + 1}{\OMEGA}.
\end{equation}
\end{lemma}

\begin{proof}
Given $\psi \in \Hm{k_p+1}{\OMEGA}$, it holds for all discretely divergence-free $\vv_h\in\VV_h^\dvg$
\begin{equation*}
	0 = -(L_h \psi, \DIV \vv_h) = \rb{\nabla (L_h \psi), \vv_h},	
\end{equation*}
  
due to $L_h \psi \in Q_h \cap H^1(\Omega)$.
Thus, one obtains
\begin{equation*}
	\rb{\nabla \psi,\vv_h}
		= \rb{\nabla (\psi - L_h \psi) ,\vv_h}
		\leqslant  \norm{\nabla (\psi - L_h \psi)}_{\LP{2}{\OMEGA}} \norm{\vv_h}_{\LPOne{2}(\Omega)}.
\end{equation*}

The result is proven using Remark \ref{rem:h1:press:approx}.
\end{proof}

\begin{remark}
The numerical experiments in \cite{lr:2018} indicate that Lemma \ref{lem:DiscHelmLerayNonDivFreeH1} is sharp.
Indeed, for non-pressure-robust, inf-sup stable mixed methods with discontinuous $\Pk{0}{}$ pressures, i.e., $k_p=0$, like the nonconforming Crouzeix--Raviart and the conforming Bernardi--Raugel element, it is demonstrated that it holds
\begin{equation*}
  \norm{\helm_h(\nabla q)}_\LP{2}{\OMEGA}=\mathcal{O}(1),	
\end{equation*}

leading to a pressure-induced locking phenomenon for the time-dependent Stokes equations in the presence
of large pressure gradients.
\end{remark}

\begin{remark}
Remark \ref{rem:p:uhnablauh:conv} assures that at least for $\ku \geqslant d/2$ one obtains for $h \to 0$ the convergence result
\begin{equation*}
  \helm( (\helm_h(\uu) \ip \nabla) \helm_h(\uu))
    \stackrel{\LPOne{2}{}}{\to}
    \helm((\uu \ip \nabla) \uu).
\end{equation*}

However, the quality of the space discretisation \eqref{eq:H1-FEM} is determined by whether and how
\begin{equation*}
  \helm_h( (\helm_h(\uu) \ip \nabla) \helm_h(\uu))
    \stackrel{\LPOne{2}{}}{\to}
    \helm((\uu \ip \nabla) \uu)
\end{equation*}

converges  for $h\to 0$. 
For $k_p \geqslant 2$, convergence in $\LPOne{2}{}$ is assured.
But the convergence speed of pressure-robust methods can be much faster than in classical, non-pressure-robust space discretisations due to Lemma \ref{lem:DiscHelmLerayNonDivFreeH1}, when $(\uu \ip \nabla) \uu$ contains a large gradient part in the sense of the Helmholtz--Hodge decomposition \eqref{helmholtz:hodge:decomposition}.
This is the main reason for the superiority of pressure-robust methods for Beltrami flows, where space discretisations suffer from an artificial, pseudo-dominant convection on coarse meshes, see Remark \ref{rem:pseudo:dominant:convection}.
Exactly this artificial pseudo-dominant convection is reduced by pressure-robust methods.
\end{remark}

\section{A special $\emph{\textbf{H}}^{\mathbf{1}}$ finite element error analysis}
\label{sec:ErrorAnalysisH1}

In the following, we present a numerical error analysis for the time-dependent Navier--Stokes problem which is based on a new understanding of the velocity error in the time-dependent Stokes problem, see \cite{lr:2018} and \cite[Section 4]{alm:2018}.
The key point is that in the case of low viscosities, pressure-robust space discretisations do not show an increase in the velocity error as long as the time interval is small compared with $\nu^{-1}$.
On the other hand, in non-pressure-robust methods the only source of error is a dominating pressure gradient in the momentum balance \cite[Theorem 5.2]{AhmedLinkeMerdon18}; namely in the special case $\ff = \mathrm{const}$ one gets (for $\nu \ll 1$) $\uu_h \approx \helm_h(\uu) + t\helm_h(\nabla p)$ for $t\in\sqb{0,\tend}$.
~\\

In the following, we will give two different error estimates for the Navier--Stokes equations in the pressure-robust and the non-pressure-robust case.
The convergence analysis is inspired by novel discrete velocity error estimates for the transient Stokes equations \cite{lr:2018, lm:2018, alm:2018}, which estimate the difference between the discrete velocity $\uu_h(t)$ and the (discretely divergence-free) $\LTWO$ best approximation $\helm_h(\uu(t))$.

\subsection{Pressure-robust space discretisation}

\begin{theorem}[Pressure-robust estimate] \label{thm:pr}%
For the discrete velocity for all $t \in \sqb{0, \tend}$ the following representation is chosen
\begin{equation*}
	\uu_h(t) = \helm_h(\uu(t)) + \ee_h(t),	
\end{equation*}

and the time-dependent evolution of $\ee_h(t)$ is considered.
Then, assuming $\uu \in \Lp{2}{\rb{\WMP{1}{\infty}{} \cap \HM{3}{}}}$, $\partial_t\uu \in \Lp{1}{\rb{\LP{2}{}}}$, on the time interval $\sqb{0, \tend}$, $\ee_h$ can be estimated by
\begin{align*}
  \norm{\ee_h(\tend)}_\LP{2}{}^2
    + \nu \norm{\nabla \ee_h}_\Lp{2}{\rb{\LP{2}{}}}^2 
   &\leqslant 
     e^{1 + 2\norm{\nabla \helm_h(\uu)}_\Lp{1}{\rb{\LP{\infty}{}}}} \times \left (
     \nu \norm{\nabla \sqb{\stokes_h(\uu) - \helm_h(\uu)}}_\Lp{2}{\rb{\LP{2}{}}}^2 \right. \\
     &  \left. + 2 \tend \int_0^\tend 
  			\norm{ \nabla \helm_h(\uu)}_\LP{\infty}{}^2\norm{\uu - \helm_h(\uu)}_\LP{2}{}^2 
  			+ \norm{\uu}_\LP{\infty}{}^2 \norm{\nabla \sqb{\uu - \helm_h(\uu)}}_\LP{2}{}^2
   		\dtau
     \right ).
\end{align*}
\end{theorem}

\begin{remark}
Due to the explicit $2\tend$ dependence of the error in Theorem \ref{thm:pr}, in the case $\tend \gg \nu^{-1}$ this estimate can become pessimistic.
Note that for our numerical examples in Section \ref{sec:ExperimentsKnownSol}, we indeed consider short time intervals for which Theorem \ref{thm:pr} is meaningful.
In the literature, usually one can find 'long-term' estimates, e.g., \cite{BochevGunzburgerLehoucq07,BurmanFernandez07,ArndtEtAl15,SchroederLube17b,SchroederEtAl18}, which are sharper for $T \gg \nu^{-1}$, but which are pessimistic for short time intervals.
\end{remark}

\begin{remark}
Concerning the regularity assumption $\uu \in \Lp{2}{\rb{\HM{3}{}}}$ in Theorem~\ref{thm:pr}, let us remark that we have chosen it in such a way that the stability of the discrete Helmholtz projector in Lemma~\ref{lem:W1inftyHelmholtz} is ensured.
This is basically a technical detail and with additional effort, one can presumably reduce the necessary regularity at this point.
However, in the error analysis of this work, we are only interested in relatively smooth flows anyway.
\end{remark}

\begin{ProofOf}{Theorem \ref{thm:pr}}
Note that $\DIV\zz_h=0$ for all $\zz_h\in\VV_h^\dvg$ due to $\VV_h^\dvg \subset \LPOne{2}_{\sigma}\OMEGA$ assuming pressure-robustness.
Due to Galerkin orthogonality, for all $\zz_h \in \VV_h^\dvg$  it holds
\begin{align*}
  \rb{\partial_t \uu_h, \zz_h} 
  	+ \nu \rb{\nabla \uu_h, \nabla \zz_h}
    + \rb{\rb{\uu_h \ip \nabla} \uu_h, \zz_h} 
    	& = \rb{\partial_t\uu, \zz_h} 
    	+ \nu \rb{\nabla \uu, \nabla \zz_h}  
    	+ \rb{\rb{\uu\ip\nabla}\uu,\zz_h}\\
        	&  = \rb{ \partial_t \helm_h(\uu), \zz_h} 
        		+ \nu \rb{\nabla \stokes_h(\uu), \nabla \zz_h}
        		+ \rb{\rb{\uu\ip\nabla}\uu,\zz_h}.
\end{align*}

Here, it was used $(\nabla p, \zz_h)=0$, which is equivalent to $\helm_h(\nabla p) = \zero$, proved in Lemma \ref{lem:DiscHelmDivFreeH1} for pressure-robust space discretisations.
Using the representation $\uu_h(t) = \helm_h(\uu(t)) + \ee_h(t)$ leads to
\begin{align*}
  \rb{\partial_t \ee_h, \zz_h}
  	+ \nu \rb{\nabla \ee_h, \nabla \zz_h}
  	&+ \rb{\rb{\sqb{\helm_h(\uu) + \ee_h} \ip \nabla} \sqb{\helm_h(\uu) + \ee_h}, \zz_h} \\
  		&= \nu \rb{\nabla (\stokes_h(\uu) - \helm_h(\uu)), \nabla \zz_h}
  			+ \rb{\rb{\uu\ip\nabla}\uu,\zz_h},
\end{align*}

for all $\zz_h \in \VV_h^\dvg$, where the initial value for the ODE system is chosen as $\ee_h(0) = \zero$.
For the discrete nonlinear term, we obtain
\begin{align*}
  \rb{\rb{\sqb{\helm_h(\uu) + \ee_h} \ip \nabla} \sqb{\helm_h(\uu) + \ee_h}, \zz_h}
   	 = &\rb{ \rb{\ee_h \ip \nabla} \ee_h, \zz_h} 
   	 	+ \rb{ \rb{\helm_h(\uu)  \ip \nabla} \ee_h, \zz_h } \\
    &+ \rb{ \rb{\ee_h \ip \nabla} \helm_h(\uu), \zz_h}
    	+ \rb{ \rb{\helm_h(\uu) \ip \nabla} \helm_h(\uu), \zz_h}.
\end{align*}

Further, due to the skew-symmetry of the first two terms plus $\DIV\ee_h=\DIV\helm_h(\uu)=0$, testing with $\zz_h = \ee_h$ leads to
\begin{align*}
 \frac{1}{2} \frac{\drm}{\drm t}  \norm{\ee_h}_\LP{2}{}^2 
 	&+ \nu \norm{\nabla \ee_h}_\LP{2}{}^2 \\
  	&\!\!\!\!\!\!\!\!\!\!= \nu \rb{\nabla \sqb{\stokes_h(\uu) -\helm_h(\uu)}, \nabla \ee_h}
  		- \rb{ \rb{\ee_h \ip \nabla} \helm_h(\uu), \ee_h}
  		- \rb{ \rb{\helm_h(\uu) \ip \nabla} \helm_h(\uu), \ee_h} 
  		+ \rb{\rb{\uu\ip\nabla}\uu,\ee_h}.	
\end{align*}

Due to Lemmas \ref{lem:conv:nonlinear:term} and \ref{lem:W1inftyHelmholtz}, the last two terms on the right-hand side can be combined and estimated by
\begin{align*}
	\big( \rb{\helm_h(\uu) \ip \nabla} \helm_h(\uu) &- \rb{\uu \ip \nabla} \uu, \ee_h \big) 
  	 	= \rb{ \rb{\sqb{\helm_h(\uu) - \uu} \ip \nabla} \helm_h(\uu) 
  			+ \rb{\uu \ip \nabla} \sqb{\helm_h(\uu) - \uu}, \ee_h} \\
   		& \leqslant  \norm{ \nabla \helm_h(\uu)}_\LP{\infty}{}
   				\norm{\uu - \helm_h(\uu)}_\LP{2}{} \norm{\ee_h}_\LP{2}{}
     		+ \norm{\uu}_\LP{\infty}{}
    			\norm{\nabla \sqb{\uu - \helm_h(\uu)}}_\LP{2}{}   \norm{\ee_h}_\LP{2}{} \\
   & \leqslant   \tend \rb{
   			\norm{ \nabla \helm_h(\uu)}_\LP{\infty}{}^2
     		\norm{\uu - \helm_h(\uu)}_\LP{2}{}^2 + \norm{\uu}_\LP{\infty}{}^2 
       		\norm{\nabla \sqb{\uu -  \helm_h(\uu)}}_\LP{2}{}^2
          } 
          + \frac{1}{2\tend} \norm{\ee_h}^2_\LP{2}{}.
\end{align*}

Here, the weight $(2\tend)^{-1}$ in front of $\norm{\ee_h}^2_\LP{2}{}$ ensures that later on, the argument of the exponential Gronwall term does not catch any explicit $\tend$ dependence.
The other convection term can be treated simply by the generalised H\"older inequality; that is, 
\begin{align*}
\rb{ \rb{\ee_h \ip \nabla} \helm_h(\uu), \ee_h} 
	\leqslant  \norm{ \nabla \helm_h(\uu)}_\LP{\infty}{} \norm{\ee_h}_\LP{2}{}^2.	
\end{align*}

Using Young's inequality for the remaining term involving Stokes and Helmholtz--Hodge projectors, after rearranging, one obtains the overall estimate
\begin{align*}
 	\frac{\drm}{\drm t}  \norm{\ee_h}_\LP{2}{}^2 
 		+ \nu \norm{\nabla \ee_h}_\LP{2}{}^2
  		 \leqslant  &\nu \norm{\nabla \sqb{\stokes_h(\uu) -\helm_h(\uu)}}_\LP{2}{}^2
  		+ \rb{ \frac{1}{\tend} + 2\norm{\nabla \helm_h(\uu)}_\LP{\infty}{} } \norm{\ee_h}_\LP{2}{}^2 \\
 		&+ 2\tend \rb{
   			\norm{ \nabla \helm_h(\uu)}_\LP{\infty}{}^2
     		\norm{\uu - \helm_h(\uu)}_\LP{2}{}^2 + \norm{\uu}_\LP{\infty}{}^2 
       		\norm{\nabla \sqb{\uu -  \helm_h(\uu)}}_\LP{2}{}^2
          } .
\end{align*}

In such a situation, Gronwall's lemma \cite[Lemma A.54]{John16} in differential form states that for $t\in\sqb{0,\tend}$,
\begin{equation}
	\frac{\drm}{\drm t}  \norm{\ee_h\rb{t}}_\LP{2}{}^2 
		\leqslant  \alpha\rb{t} + \beta\rb{t} \norm{\ee_h\rb{t}}_\LP{2}{}^2 	
		\quad \Rightarrow \quad 
	\norm{\ee_h\rb{t}}_\LP{2}{}^2
		\leqslant  \int_0^t
			\alpha\rb{s}	 \exp\rb{\int_s^t \beta\rb{\tau} \dtau } \drms.
\end{equation}

In order to apply this estimate, one sets
\begin{equation*}
	\beta(t) =\frac{1}{\tend} + 2\norm{ \nabla \helm_h(\uu) }_\LP{\infty}{}	
\end{equation*}
   
and
\begin{align*}
  \alpha(t) 
  	= &-\nu \norm{\nabla \ee_h}_\LP{2}{}^2 
  		+ \nu \norm{\nabla \sqb{\stokes_h(\uu)-\helm_h(\uu)} }_\LP{2}{}^2 \\
    &+ 2\tend \rb{
   			\norm{ \nabla \helm_h(\uu)}_\LP{\infty}{}^2
     		\norm{\uu - \helm_h(\uu)}_\LP{2}{}^2 + \norm{\uu}_\LP{\infty}{}^2 
       		\norm{\nabla \sqb{\uu -  \helm_h(\uu)}}_\LP{2}{}^2
          }, 	
\end{align*}

and computes for $t \geqslant s$
\begin{equation*}
	\exp\rb{\int_s^t \beta(\tau) \dtau} 
		= \exp\rb{ \frac{t-s}{T}   + 2\norm{\nabla \helm_h(\uu)}_{\Lp{1}{\rb{s,t; \LP{\infty}{}}}} }.
\end{equation*}

Using $0 \leqslant  s \leqslant  t \leqslant  T$ and $0 \leqslant  \frac{t-s}{T} \leqslant  1$, one obtains
\begin{align*}
  c \exp\rb{ \frac{t-s}{T}   + 2\norm{\nabla \helm_h(\uu)}_{\Lp{1}{\rb{s,t; \LP{\infty}{}}}} } 
  	\leqslant  \begin{cases}
      	c \exp\rb{1 + 2\norm{\nabla \helm_h(\uu)}_{\Lp{1}{\rb{\LP{\infty}{}}}} } 
      		& \text{for $c \geqslant 0$}, \\
      	c 
      		& \text{for $c<0$}.
    \end{cases}	
\end{align*}

Now, actually applying Gronwall's lemma yields the estimate
\begin{align*}
	&\norm{\ee_h\rb{\tend}}_\LP{2}{}^2
	\leqslant  \int_0^\tend \alpha(s) \exp\rb{\int_s^\tend \beta(\tau) \dtau } \drms
  		 \leqslant   -\nu \norm{\nabla \ee_h}_\Lp{2}{\rb{\LP{2}{}}}^2
       		+ e^{1 + 2\norm{\nabla \helm_h(\uu)}_\Lp{1}{\rb{\LP{\infty}{}}}} \times \\
      &
     \rb{
     \nu \norm{\nabla \sqb{\stokes_h(\uu) - \helm_h(\uu)}}_\Lp{2}{\rb{\LP{2}{}}}^2 
     + 2 \tend \int_0^\tend 
  			\norm{ \nabla \helm_h(\uu)}_\LP{\infty}{}^2\norm{\uu - \helm_h(\uu)}_\LP{2}{}^2 
  			+ \norm{\uu}_\LP{\infty}{}^2 \norm{\nabla \sqb{\uu - \helm_h(\uu)}}_\LP{2}{}^2
   		\dtau
     }.
\end{align*}

Rearranging concludes the proof.     
\end{ProofOf}

\begin{remark}
The estimate in Theorem \ref{thm:pr} is pressure-robust, since if $\uu(t) \in \VV_h^\dvg$ holds for all $t \in  \sqb{0,\tend}$, then it also holds $\uu_h(t) = \uu(t)$ due to $\uu = \stokes_h(\uu)=\helm_h(\uu)$, i.e., the pressure $p$ does not spoil the discrete velocity solution $\uu_h$.
\end{remark}

\begin{remark}
Note that under the assumption that $\ku$ is large enough and $h$ is small enough, one can apply Lemma \ref{lem:W1inftyHelmholtz} to obtain $\norm{ \nabla \helm_h(\uu)}_\LP{\infty}{} \leqslant  C \norm{ \nabla \uu}_\LP{\infty}{}$.
\end{remark}

\subsection{Classical space discretisation}

\begin{theorem}[Non-pressure-robust estimate] \label{thm:classical}%
For the discrete velocity for all $t \in \sqb{0, \tend}$ the following representation is chosen
\begin{equation*}
	\uu_h(t) = \helm_h(\uu(t)) + \ee_h(t),	
\end{equation*}

and the time-dependent evolution of $\ee_h(t)$ is considered.
Then, assuming $\uu \in \Lp{2}{\rb{\WMP{1}{\infty}{} \cap \HM{3}{}}}$, $\partial_t\uu \in \Lp{1}{\rb{\LP{2}{}}}$, $p \in \Lp{2}{\rb{\Hm{1}{}}}$, on the time interval $\sqb{0, \tend}$, $\ee_h$ can be estimated by
\begin{align*}
  \norm{\ee_h(\tend)}_\LP{2}{}^2
    &+ \nu \norm{\nabla \ee_h}_\Lp{2}{\rb{\LP{2}{}}}^2 
   \leqslant  
     e^{1 + 4\norm{\nabla \helm_h(\uu)}_\Lp{1}{\rb{\LP{\infty}{}}}} \times  
     \Big(
     \nu \norm{\nabla \sqb{\stokes_h(\uu) - \helm_h(\uu)}}_\Lp{2}{\rb{\LP{2}{}}}^2 \\
     &+3\tend \norm{\nabla \sqb{p - L_h(p)}}_\Lp{2}{\rb{\LP{2}{}}}^2 
      + 3 \tend \int_0^\tend \big[ 
  			\norm{\helm_h(\uu)}_\LP{\infty}{}^2\norm{\DIV\helm_h(\uu)}_\Lp{2}{}^2 \\
  			& \qquad 
  			+ 2\norm{\nabla \helm_h(\uu)}_\LP{\infty}{}^2\norm{\uu - \helm_h(\uu)}_\LP{2}{}^2 
  			+ 2\norm{\uu}_\LP{\infty}{}^2 \norm{\nabla \sqb{\uu - \helm_h(\uu)}}_\LP{2}{}^2
   		\big]\dtau
     \Big).
\end{align*}
\end{theorem}

\begin{ProofOf}{Theorem \ref{thm:classical}}
Following the same procedure as for the pressure-robust case, one obtains
\begin{align*}
 	\frac{1}{2} \frac{\drm}{\drm t} \norm{\ee_h}_\LP{2}{}^2 
 		&+ \nu \norm{\nabla \ee_h}_\LP{2}{}^2
  		 = \nu \rb{\nabla \sqb{\stokes_h(\uu) -\helm_h(\uu)}, \nabla \ee_h}
  			+ \rb{\nabla \sqb{p - L_h(p)}, \ee_h} 
  			+ \rb{\rb{\uu\ip\nabla}\uu,\ee_h} \\
  			& - \rb{ \rb{\ee_h \ip \nabla} \helm_h(\uu) + \frac{1}{2}\rb{\DIV\ee_h}\helm_h(\uu), \ee_h}
  			- \rb{ \rb{\helm_h(\uu) \ip \nabla} \helm_h(\uu) + \frac{1}{2}\rb{\DIV\helm_h(\uu)}\helm_h(\uu), \ee_h}  .
\end{align*}

Note that in the non-pressure-robust case  the pressure contribution does not vanish, since it holds $\VV_h^\dvg \not\subset \LPOne{2}_{\sigma}\OMEGA$ and the consistency error $\norm{\helm_h(\nabla p)}_{\LPOne{2}\OMEGA}$ estimated in Lemma \ref{lem:DiscHelmLerayNonDivFreeH1} unavoidably appears in the estimate.
Further, the full skew-symmetric convective term has to be taken into account.
In comparison to the pressure-robust case, there are only two new contributions from the skew-symmetrisation of the convective term which have to be estimated.
For the first one,
\begin{align*}
	\rb{\frac{1}{2}\rb{\DIV\helm_h(\uu)}\helm_h(\uu), \ee_h}
		&= \rb{\frac{1}{2}\rb{\DIV\sqb{\uu-\helm_h(\uu)}}\helm_h(\uu), \ee_h}
		\leqslant  \norm{\helm_h(\uu)}_\LP{\infty}{} \norm{\DIV\sqb{\uu-\helm_h(\uu)}}_\Lp{2}{} \norm{\ee_h}_\LP{2}{} \\
		&\leqslant  
			\frac{3\tend}{2}\norm{\helm_h(\uu)}_\LP{\infty}{}^2\norm{\DIV\sqb{\uu-\helm_h(\uu)}}_\Lp{2}{}^2
			+ \frac{1}{6\tend}\norm{\ee_h}_\LP{2}{}^2
		\end{align*}

can be obtained. 
Applying Young's inequality slightly differently than in the pressure-robust case leads to
\begin{align*}
	\big( \rb{\helm_h(\uu) \ip \nabla} \helm_h(\uu) &- \rb{\uu \ip \nabla} \uu, \ee_h \big) \\
		 &\leqslant   \frac{6\tend}{2} \rb{
   			\norm{ \nabla \helm_h(\uu)}_\LP{\infty}{}^2
     		\norm{\uu - \helm_h(\uu)}_\LP{2}{}^2 + \norm{\uu}_\LP{\infty}{}^2 
       		\norm{\nabla \sqb{\uu -  \helm_h(\uu)}}_\LP{2}{}^2
          } 
          + \frac{1}{6\tend} \norm{\ee_h}^2_\LP{2}{}.
\end{align*}

With the help of integration by parts, the additional contribution in the remaining additional convective term can be estimated as
\begin{equation*}
	\rb{\frac{1}{2}\rb{\DIV\ee_h}\helm_h(\uu),\ee_h}
		= -\frac{1}{2}\rb{\rb{\nabla\helm_h(\uu)}\ee_h,\ee_h}
		\leqslant  \norm{ \nabla \helm_h(\uu)}_\LP{\infty}{} \norm{\ee_h}_\LP{2}{}^2.
\end{equation*}

For completeness, from the pressure-robust case, we repeat
\begin{align*}
\rb{ \rb{\ee_h \ip \nabla} \helm_h(\uu), \ee_h} 
	\leqslant  \norm{ \nabla \helm_h(\uu)}_\LP{\infty}{} \norm{\ee_h}_\LP{2}{}^2.	
\end{align*}

Finally, the additional pressure term can be bounded as follows:
\begin{equation*}
	\rb{\nabla \sqb{p - L_h(p)}, \ee_h}
		\leqslant  \frac{3\tend}{2} \norm{\nabla \sqb{p - L_h(p)}}_\LP{2}{}^2
			+ \frac{1}{6\tend} \norm{\ee_h}_\LP{2}{}^2
\end{equation*}

This yields the overall estimate
\begin{align*}
 	\frac{\drm}{\drm t}  &\norm{\ee_h}_\LP{2}{}^2 
 		+ \nu \norm{\nabla \ee_h}_\LP{2}{}^2 \\
  		 &\leqslant  \nu \norm{\nabla \sqb{\stokes_h(\uu) -\helm_h(\uu)}}_\LP{2}{}^2
  		+ \rb{ \frac{1}{\tend} + 4\norm{\nabla \helm_h(\uu)}_\LP{\infty}{} } \norm{\ee_h}_\LP{2}{}^2
  		+ 3\tend \norm{\nabla \sqb{p - L_h(p)}}_\LP{2}{}^2 \\
 		&\quad+ 3\tend \rb{
 			\norm{\helm_h(\uu)}_\LP{\infty}{}^2\norm{\DIV\helm_h(\uu)}_\Lp{2}{}^2
   			+ 2\norm{ \nabla \helm_h(\uu)}_\LP{\infty}{}^2 \norm{\uu - \helm_h(\uu)}_\LP{2}{}^2 
     		+ 2\norm{\uu}_\LP{\infty}{}^2 \norm{\nabla \sqb{\uu -  \helm_h(\uu)}}_\LP{2}{}^2
          } .
\end{align*}

Similarly as in the pressure-robust case, Gronwall's lemma concludes the proof.
\end{ProofOf}

\section{Consistency errors and the accuracy of low/high-order methods}
\label{sec:Taylor}

The main argument of this contribution is that pressure-robust space discretisations allow to reduce the (formal) approximation order of the algorithms, without compromising the accuracy, since the discrete Helmholtz--Hodge projector $\norm{\helm_h(\nabla p)}_\LP{2}{}$ of classical, non-pressure-robust discretisations suffers from a consistency error. 
This section will now interpret the numerical analysis of Section \ref{sec:ErrorAnalysisH1} according to this point of view. 
The main difference between Theorems \ref{thm:pr} (pressure-robust) and \ref{thm:classical} (non-pressure-robust) is the term
\begin{equation} \label{eq:pr:term}
  3\tend \norm{\nabla \sqb{p - L_h(p)}}_\Lp{2}{\rb{\LP{2}{}}}
\end{equation}

in the $\LP{2}{}$ estimate for the non-pressure-robust error $\ee_h$, which is a direct consequence of the consistency error $\norm{\helm_h(\nabla p)}_\LP{2}{}$ in Lemma \ref{lem:DiscHelmLerayNonDivFreeH1}.
~\\

In the following, let $\kucl$/$\kpcl$ be the order of the discrete velocity/pressure polynomials for a classical, non-pressure robust method and, analogously, $\kupr$/$\kppr$ the orders for a pressure-robust FE discretisation.

\subsection{Lowest-order discretisations} \label{subsec_low:order:discs}

An obvious, seemingly yet unknown conclusion from \eqref{eq:pr:term} is that for non-pressure-robust space discretisations, for $\kpcl=0$ no convergence order at all can be expected on pre-asymptotic meshes in presence of non-negligible pressures $p$.
The reason is simple: discrete $\Pk{0}{}$ pressure do not have any approximation property w.r.t.\ the $H^1$ norm of $p$.
~\\

Recently, in \cite{lr:2018} it was numerically confirmed that the estimate \eqref{eq:pr:term} is sharp.
The discretely inf-sup stable, non-pressure-robust Crouzeix--Raviart element indeed shows an error behaviour $\ee_h = \mathcal{O}(1)$ on pre-asymptotic meshes for $\nu \ll 1$, i.e., no convergence order at all was observed --- a classical locking phenomenon. 
Furhtermore, in the time-dependent Stokes problem, it was shown that classical, non-pressure-robust methods with $\kpcl=\kucl-1$ even lose two orders of convergence in the $\LP{2}{}$ norm w.r.t.\ comparable pressure-robust methods.
E.g., while the pressure-robust Scott--Vogelius element with $(\kupr=2, \kppr=1)$ converges with the optimal order $3$ in the $\LP{2}{}$ norm, the classical, non-pressure-robust Taylor--Hood method $(\kucl=2, \kpcl=1)$ converges only with order $1$ in the $\LP{2}{}$-norm \cite{lr:2018}. 
~\\

Thus, non-pressure-robust space discretisations need higher-order discrete pressure spaces in order to get reasonable convergence orders for their discrete velocities --- since high-order discrete pressure approximations reduce the consistency error $\norm{\helm_h(\nabla p)}_\LP{2}{}$ by a simple Taylor expansion. 
Due to inf-sup stability, usually it holds $\kucl \geqslant \kpcl$ (usually $\kpcl = \kucl-1$), and high-order discrete pressure spaces require high-order discrete velocity spaces as well.
As a conclusion, pressure-robust methods with $\kppr=\kupr-1$ for the time-dependent Stokes problem converge with two more orders of convergence in the $\LP{2}{}$ norm than non-pressure-robust methods with $\kpcl = \kucl-1$, if the pressure $\nabla p$ is non-negligible \cite{lr:2018}.

\subsection{Beltrami flows}

The considerations for time-dependent Stokes problems with $\nu \ll 1$ and large pressure gradients $\nabla p$ are now applied to high-Reynolds number Beltrami flows, where it holds $\ff=\zero$ and $-(\uu \ip \nabla) \uu = -\frac{1}{2} \nabla \abs{\uu}^2=\nabla p$, i.e., the dominant nonlinear convection term induces a large pressure gradient.

\subsubsection{Polynomial potential flows} \label{sub:sec:sec:poly}

A simple consideration allows to show that pressure-robust discretisations sometimes allow to reduce the formal approximation order from $\kucl = 2 k + 1$ (non-pressure-robust, $\kpcl=\kucl-1$) to $\kupr=k$ without compromising the accuracy for the discrete velocities at all.
~\\

Let us assume that for all $t \in  \sqb{0,\tend}$, $(\uu, p)$ is a time-dependent polynomial potential flow with $h(t) \in \Pk{k+1}{}$ for all $t$, i.e., it holds for all times that $\uu(t)=\nabla h(t)$, $\nabla p(t)=-\frac{1}{2}\nabla\abs{\uu(t)}^2$ and $\Delta h(t)=0$. 
Then, $(\uu(t), p(t)) \in \PPk{k}{} \times \Pk{2k}{}$ fulfils (for all fixed $\nu > 0$) the time-dependent Navier--Stokes equations \eqref{eq:transient:navier:stokes} with $\ff=\zero$, and $\uu$ is indeed a Beltrami flow, see Section \ref{sec:GenBeltramiFlows}.
Note that one has to impose, e.g., time-dependent inhomogeneous Dirichlet velocity boundary conditions.
A pressure-robust space discretisation of order $\kupr=k$ will deliver the exact velocity solution $\uu_h(t)=\uu(t)$ for all $t\in  \sqb{0,\tend}$ according to Theorem \ref{thm:pr} on every shape-regular mesh.
On the contrary, non-pressure-robust space discretisations only deliver the exact velocity solution $\uu_h(t)=\uu(t)$, if it also holds $p(t) \in Q_h$ for all $t \in  \sqb{0,\tend}$ according to the consistency error \eqref{eq:pr:term} in Theorem \ref{thm:classical}. 
Thus, classical space discretisations (with $\kpcl=\kucl-1$) require $\kpcl=2 k$, i.e., $\kucl=\kpcl+1 = 2k+1$.
~\\

This observation has been already published in \cite{cmame:linke:merdon:2016}. 
Therefore, we simply refer to the numerical results therein.

\subsubsection{Non-$C^{\infty}$ Beltrami flows}

In the following, it is assumed that $\uu$ is a time-dependent Beltrami flow with $\uu \in \Lp{\infty}{\rb{0,\tend; \HH^{k+1}}}$ for $k \in \mathbb{N}$ and $k \geqslant 1+d/2$ and $\uu \notin \Lp{\infty}{\rb{0,\tend; \HH^{k+1+\eps}}}$ for any $\eps>0$.
According to Theorem \ref{thm:pr} for $\nu \ll 1$, one gets for pressure-robust space discretisations with $\kupr=k$ a convergence order $k \leqslant  \kupr \leqslant  k+1$ for the $\LP{2}{}$ norm; note that the nonlinearity of the problem may lead to a reduction of the convergence order on pre-asymptotic meshes, as compared to the time-dependent Stokes problem \cite{lm:2018}.
~\\

Turning to non-pressure-robust methods, one notes that due to $\nabla p=-\frac{1}{2}\nabla \abs{\uu}^2$ and $\uu \in \Lp{\infty}{\rb{0,\tend; \HH^{k+1}}}$ it holds $p \in \Lp{\infty}{\rb{0,\tend; H^{k+1}}}$.
The result can be proven by the Leibniz formula yielding for the $\kappa$-th derivative of $\frac{1}{2} (f(x))^2$  (with $\kappa=0, 1, \ldots, k+1$)
\begin{equation} \label{eq:leibniz:formula}
  \frac{1}{2} |(f(x))^2|^{(\kappa)} =
    \frac{1}{2} \sum_{i=0}^{\kappa} \binom{\kappa}{i} f^{(\kappa-i)} f^{(i)}.
\end{equation}

According to the Sobolev imbedding theorem for $d \leq 3$, one concludes: for $f \in H^1$ it holds also $f \in L^6$ and for $f \in H^2$ it holds also $f \in L^{\infty}$. 
Therefore, all the terms in \eqref{eq:leibniz:formula} are either in $L^3$ or in $L^2$, if at least $f \in H^2$.
Substituting $f=\abs{\uu}$ and searching for partial derivatives delivers the regularity for the pressure $p$.
 
Due to the spatial regularity $(\uu, p) \notin \HH^{k+1+\eps} \times H^{k+1+\eps}$, one concludes that any space discretisation method with $\kucl > k$ cannot converge with a better convergence order than the pressure-robust method with $\kupr=k$. 
Also the pressure-dependent velocity error contribution \eqref{eq:pr:term} has a consistency error, which cannot be expected to be better than of order $k$.
~\\

Prescribing a certain accuracy of the velocity error of the pressure-robust method like
\begin{equation*}
  \norm{\ee_h}_\Lp{\infty}{\rb{0,\tend; \LP{2}{}}}
    \leqslant  \frac{\eps}{\norm{\uu}_\Lp{\infty}{\rb{0,\tend; \LP{2}{}}}}	
\end{equation*}

for $\uu \not= \zero$ and $\eps \ll 1$ can be fulfilled for mesh sizes $h$ which are fine enough, i.e., $h < h(\eps)$; in the case $\uu(t)\equiv \zero$, every mesh allows for the exact solution $\uu_h(t)=\uu(t)$ for pressure-robust methods.
Now, the solution $\uu$ can be represented as $\uu = \uu_h + (\uu-\uu_h) \eqqcolon \uu_h + \rr$.
~\\

Then, it holds further $\nabla p=- \nabla ( \frac{1}{2}\abs{\uu_h}^2 + \uu_h\ip \rr + \frac{1}{2} \abs{\rr}^2)$ with $\rr \in \Lp{\infty}{\rb{\HH^{k+1}}}$, $\frac{1}{2} \abs{\rr}^2 \in \Lp{\infty}{\rb{\HH^{k+1}}}$ and $\norm{\rr}_\Lp{\infty}{\rb{\LP{2}{}}} \approx {\eps}/{\norm{\uu}_\Lp{\infty}{\rb{\LP{2}{}}}}$.
Due to $\frac{1}{2} \abs{\uu_h}^2 \in \Pk{2k}{(\T)}$ one can get (depending on the flow field $\uu$ and its Sobolev semi norms $|\uu|_{\HH^i}$ for $i=0, 1, \ldots, 2 k-1)$ that the approximation of $\frac{1}{2} \abs{\uu_h}^2 $ by piecewise polynomials from $\Pk{i}{(\T)}$ may lead to a non-negligible error for all $i=0, 1, \ldots, 2k-1$.
Note that the higher derivatives of $-\frac{1}{2} |\uu|^2$ are given by a weighted sum of products of low- and high-order derivatives of $\uu$ according to \eqref{eq:leibniz:formula}.
Then, only the choice $\kpcl=2k$ and $\kucl=2k+1$ can ensure that for the classical, non-pressure-robust method also holds  $\norm{\nabla \sqb{p - L_h(p)}}_\Lp{2}{\rb{\LP{2}{}}}=\mathcal{O}(\eps)$.
~\\

Thus, the considerations for purely polynomial Beltrami flows can also be extended to more general Beltrami flows, and pressure-robust discretisations of formal order $\kupr=k$ can be comparably accurate as classical, non-pressure-robust discretisations of order $\kucl=2 k+1$.

\subsubsection{Analytic Beltrami flows}

For analytic, but non-polynomial Beltrami flows no clear statements beyond Subsection \ref{subsec_low:order:discs} can be given about how much pressure-robust methods allow to reduce the formal order of the approximations without compromising the accuracy.
However, in the numerical examples below we will exclusively compare numerical results on analytic flows, where high-order methods profit from exponential convergence.
Nevertheless, one can confirm that for moderate formal approximation orders up to $\kucl=6$ pressure-robust methods allow for halving the approximation order without compromising the accuracy on coarse meshes.

\subsection{Generalised Beltrami flows} \label{subsec:gen:beltrami:flows}

Actually, pure Beltrami flows are only difficult for classical, non-pressure-robust space discretisations whenever the nonlinear convection term is approximated by the so-called convective form $(\uu_h \ip \nabla) \uu_h$ (or a skew-symmetric variant thereof). 
Alternatively, one can exploit \eqref{eq:nonlin:identity} which tells us that it holds
\begin{equation} \label{eq:rot:form}
  \helm((\uu \ip \nabla) \uu) =
    \helm((\nabla \times \uu) \times \uu).
\end{equation}

Therefore, velocity solutions of the incompressible Navier--Stokes equations \eqref{eq:transient:navier:stokes} also fulfil
\begin{equation*}
  \partial_t\uu - \nu \Delta \uu + \rb{\nabla \times \uu} \times \uu
   + \nabla \pi^\mathrm{rot}  =  \ff, \qquad \DIV \uu = 0,	
\end{equation*}

where the pressure $p$ has been replaced by the new pressure variable
\begin{equation}
  \pi^\mathrm{rot} \coloneqq p + \frac{1}{2} \nabla |\uu|^2.
\end{equation}

This is the so-called rotational or vector-invariant form \cite{lo:2002} of the incompressible Navier--Stokes equations which is equivalent to the convective form.
~\\ 

It should be noted that pressure-robust (exactly divergence-free) methods, like the Scott--Vogelius element, deliver exactly the same discrete velocities for the convective and the rotational form on every mesh, since they appropriately handle the equivalence classes of forces leading to \eqref{eq:rot:form} \cite{bclr:2010}.
On the other hand, classical, non-pressure robust methods, which do not respect the equivalence classes of forces exactly, deliver different discrete velocities.
For Beltrami flows, one obtains
\begin{equation*}
  \pi^{\mathrm{rot}} = -\frac{1}{2} \nabla |\uu|^2 + \frac{1}{2} \nabla |\uu|^2
   = 0	
\end{equation*}

and the issue of a lack of pressure-robustness in classical space discretisations does not play any role for high Reynolds number Beltrami flows using the the rotational form of the Navier--Stokes equations \cite{cmame:linke:merdon:2016}.
~\\

So, why not simply using classical, non-pressure-robust methods in connection with the rotational form for the simulation of high Reynolds number flows?
The reason is that there exist generalised Beltrami flows, --- where $(\uu \ip \nabla) \uu$ is a gradient field --- which can be accurately simulated with classical methods with the convective form, but not accurately with the rotational form at high Reynolds numbers. 
We also refer to e.g.\ \cite{Zang:1991} for numerical investigations showing that the rotational form of the incompressible Navier--Stokes can be inaccurate in FEM discretisations.
~\\

The easiest example is quadratic, planar Hagen--Poiseuille flow in a channel. 
Here, it holds $(\uu \ip \nabla) \uu = \zero$, and the nonlinear convection term is a trivial gradient field, which is, e.g., always a discrete solution of the non-pressure-robust Taylor--Hood element $\PPk{2}{}/\Pk{1}{}$ for all Reynolds numbers on all meshes.
However, using the quadratic Taylor--Hood element in connection with the rotational form will lead to enormous velocity errors on coarse meshes and high Reynolds numbers, since the corresponding pressure $\pi^\mathrm{rot}$ will be a fourth order polynomial, again.
~\\

Moreover, in complicated flows like in a K\'arm\'an vortex street, see Section \ref{sec:KarmanBeltrami}, usually there dominate different types of flows locally: at the inlet a situation similar to a Hagen--Poiseuille flow dominates, where the convective form is accurate for classical discretisations, and in front of the obstacle a situation like a  generalised Beltrami flow may prevail, where other forms of the convective term could be more accurate for non-pressure-robust discretisations.
~\\

In conclusion, pressure-robust discretisations respect \eqref{eq:rot:form} exactly on the discrete level and are thus appropriate for all types of situations.
However, if for a generalised Beltrami flow the classical discretisation in convective form is compared to a pressure-robust discretisation, the pressure-robust discretisation may have a dramatic speedup (for Beltrami-type flows), but it can also happen that there is no speedup at all, since Hagen--Poiseuille with the trivial nonlinear convection term $(\uu \ip \nabla) \uu = \zero$ is also a generalised Beltrami flow.
Thus, the speedup question is not really decidable, but large speedups are achievable for generalised Beltrami flows, see, e.g., the numerical results in Subsection \ref{sec:2DLatticeFlow}.

\section{Gresho vortex problem with $\emph{\textbf{H}}^{\mathbf{1}}$-FEM}
\label{sec:GreshoH1}

A frequently used 2D model problem for investigating how well a discretisation preserves structures is the so-called `Gresho vortex' \cite{GreshoChan90} (originally called `triangle vortex').
Centred at $\cc=\rb{c_1,c_2}^\dag\in\R^2$ and with a (constant) translational velocity $\ww_0\in\R^2$, the problem setup in Cartesian coordinates is fully described by the initial condition
\begin{align} \label{eq:GreshoInitialVortex}
	\uu_0\rb{\xx}
	= \ww_0
	+ \begin{cases}
		\rb{-5\widetilde{x}_2,5\widetilde{x}_1}^\dag,\quad &0\leqslant r < 0.2,\\
		\rb{-\frac{2\widetilde{x}_2}{r}+5\widetilde{x}_2, \frac{2\widetilde{x}_1}{r}-5\widetilde{x}_1}^\dag, &0.2\leqslant r < 0.4,\\
		\rb{0,0}^\dag, &0.4\leqslant r,	
	\end{cases}
\end{align} 

with Euclidean distance from the vortex centre $r=\abs{\xx-\cc}_2$ and $\rb{\widetilde{x}_1,\widetilde{x}_2}=\rb{x_1-c_1,x_2-c_2}$.
The initial Gresho vortex \eqref{eq:GreshoInitialVortex} has a constant vorticity in its core for $r<0.2$ (similar to a rigid body rotation) but then, for $0.2 \leqslant r < 0.4$ it decreases linearly and vanishes for $r\geqslant 0.4$ with discontinuities at $r\in\set{0.2,0.4}$.
In Figure~\ref{fig:Gresho_initial_state} one can get an impression of how the initial velocity field described by \eqref{eq:GreshoInitialVortex} looks like (left-hand side); on the right-hand side, the (discontinuous) initial vorticity is shown.
When $\ww_0\equiv\zero$ is chosen, the name `standing vortex problem' is also used sometimes.
~\\

\begin{figure}[h]
\centering
	\includegraphics[width=0.31\textwidth]{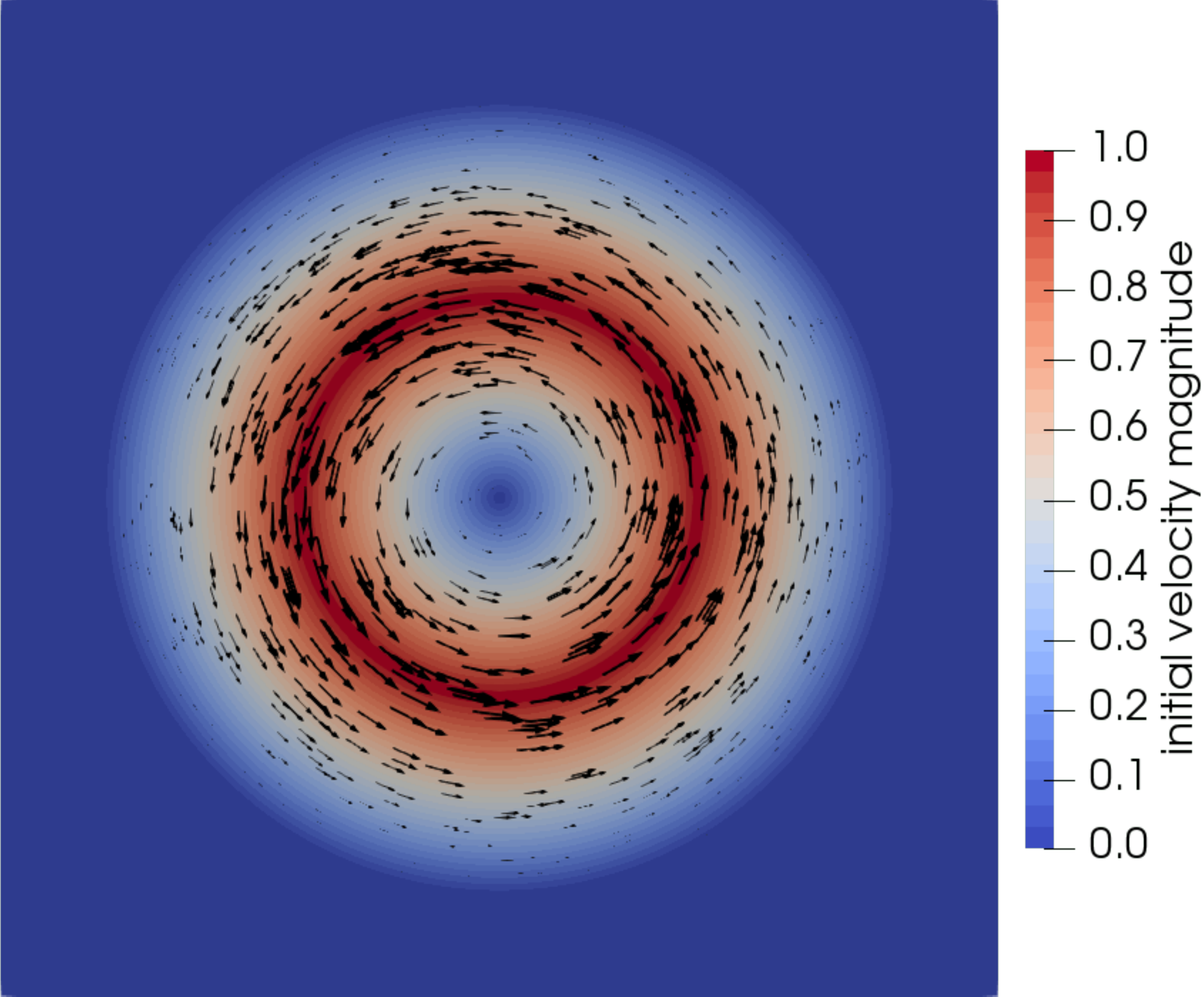} \hspace{10pt}
	\includegraphics[width=0.31\textwidth]{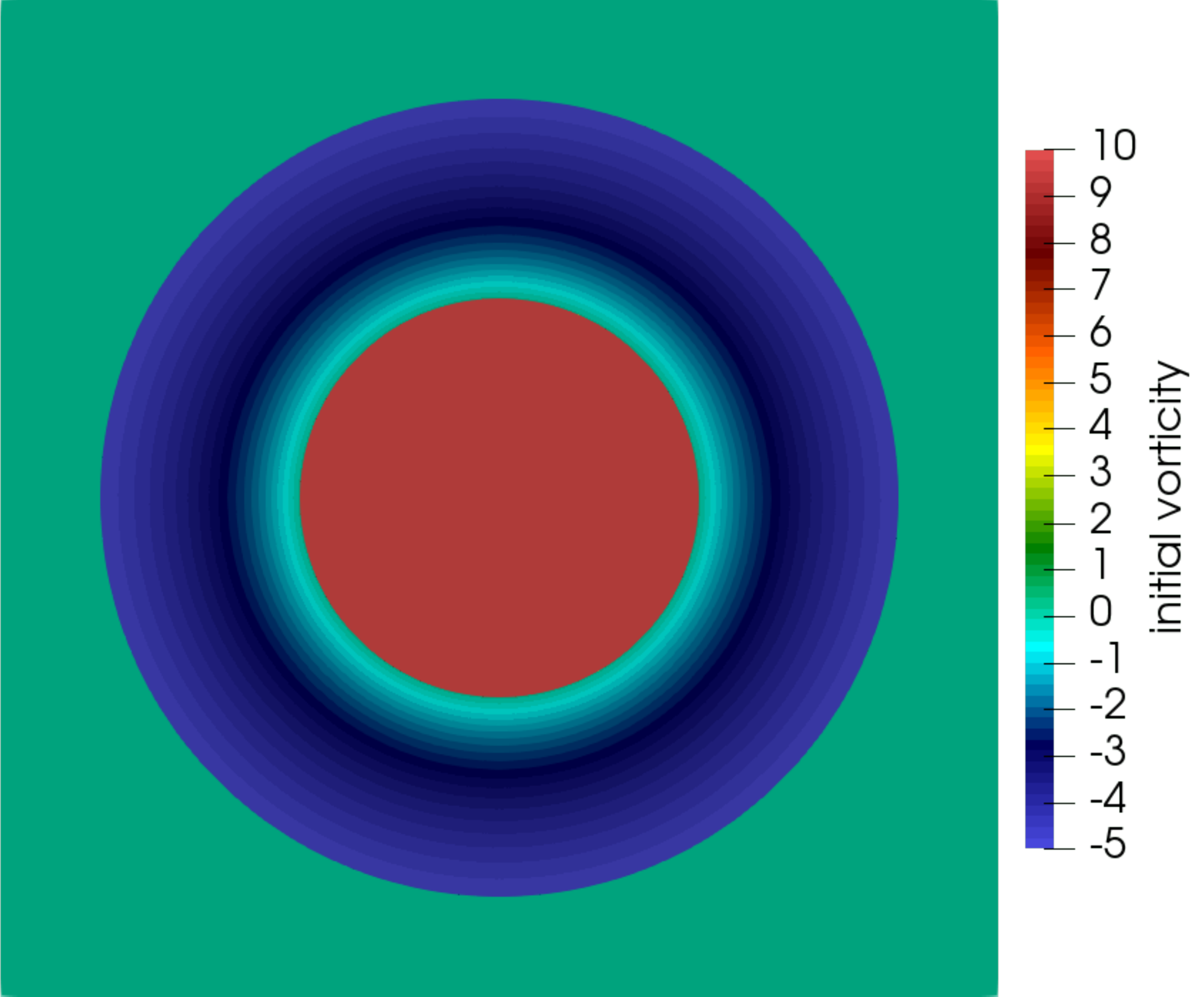}	
\caption{Initial state of Gresho vortex problem. Velocity magnitude $\abs{\uu_0}$ (left) and vorticity $\omega_0$ (right).}
\label{fig:Gresho_initial_state}
\end{figure}

Our actual simulations are done in the $x_1$- and $x_2$- periodic box $\Omega=\rb{0,1}^2$ with centre $\cc=\rb{0.5,0.5}^\dag$ and up to an end time $\tend=3$.
For a wind $\ww_0=\rb{0,0}^\dag$ we obtain the standing Gresho vortex problem, while, e.g., choosing $\ww_0=\rb{1/3,1/3}^\dag$ corresponds to the moving Gresho vortex setting which we will consider in this work.
The full time-dependent Navier--Stokes problem \eqref{eq:TINS} with $\ff\equiv\zero$ governs the problem with initial condition \eqref{eq:GreshoInitialVortex} and we fix $\nu=\num{E-5}$.
For the spatial discretisation, unstructured triangular meshes are employed, and for the time-stepping we use a constant time step $\Delta t=\num{e-4}$ with the second-order Runge--Kutta (RK) variant ARS(2,2,2) of the implicit-explicit (IMEX) method introduced in \cite{AscherEtAl97}; cf., for example, \cite{Lehrenfeld10} for more details about time integration for PDEs in this context.
Stokes-type subproblems are treated implicitly and the associated matrix for solving linear systems is called $M^\ast$, whereas convection is treated explicitly. All computations in this work have been done with the high-order finite element library \texttt{NGSolve} \cite{Schoeberl14}. 
~\\

In the following, we compare the exactly divergence-free and pressure-robust $\PPk{k}{}/\Pdk{k-1}{}$ Scott--Vogelius method SV$_k$ with the classical non-pressure-robust $\PPk{k}{}/\Pk{k-1}{}$ Taylor--Hood method TH$_k$ for $k\in\set{4,8}$.
The discrete formulation for both methods is given by \eqref{eq:H1-FEM}.
Note again that while SV$_k$ is energy-stable without any modifications to the convection term, TH$_k$ is explicitly used with a skew-symmetrisation involving the divergence of discrete velocities.
~\\

In Table~\ref{tab:DOFs}, the relevant numbers of degrees of freedom (DOFs) and numbers of non-zero entries (NZEs) of $M^\ast$, which result from different triangular meshes are collected.
The reference solution refSV$_k$ (no convection stabilisation) is used to assess the quality of the other less-resolved simulations.
Here, the counted DOFs indicate the costs for explicit operator applications (convection), while the NZEs allow to assess the effort involved in implicit linear solves for the Stokes subproblems.
Note that four different meshes are chosen in order to ensure a (relatively) fair comparison between the different methods and polynomial orders.
~\\

\begin{table}[h]
\caption{Overview of number of mesh elements, DOFs and NZEs of $M^\ast$ for the Gresho problem, based on discretisations with Scott--Vogelius SV$_k$ and Taylor--Hood TH$_k$ of different order $k\in\set{4,8}$. A reference solution refSV$_8$ is also computed. The abbreviations `K' and `M' denote thousands and millions, respectively.}
\label{tab:DOFs}
\centering
\noindent
\begin{tabular}{N N N N N N}
\toprule
	\multicolumn{1}{N}{\textbf{}} & 
		\multicolumn{3}{c}{Scott--Vogelius} & 
		\multicolumn{2}{c}{Taylor--Hood} 
		\\  
\cmidrule[0.05em](lr){2-4}
\cmidrule[0.05em](ll){5-6}
	\multicolumn{1}{c}{Name} 
		& refSV$_8$
		& SV$_8$ 
		& SV$_4$
		& TH$_8$	 
		& TH$_4$ 
		\\ 
\cmidrule[0.05em](lr){1-1}
\cmidrule[0.05em](lr){2-4}
\cmidrule[0.05em](ll){5-6}
	\multicolumn{1}{l}{$\#\set{\mathrm{trigs}}$} 
		& \numQ{23446}
		& \numQ{330}
		& \numQ{1346}
		& \numQ{432}
		& \numQ{2004}  
		\\
	\multicolumn{1}{l}{$\#\set{\uu\,\mathrm{DOFs}}$} 
		& \numQ{1500544}
		& \numQ{21120}
		& \numQ{21536}
		& \numQ{27648}
		& \numQ{32064}  
		\\
	\multicolumn{1}{l}{$\#\set{p\,\mathrm{DOFs}}$} 
		& \numQ{844056}
		& \numQ{11880}
		& \numQ{13460}
		& \numQ{10584}
		& \numQ{9018}  
		\\
	\multicolumn{1}{l}{$\#\set{\text{nze}(M^\ast)}$	} 
		& \numQ{239864425}
		& \numQ{3377120}
		& \numQ{1302724}
		& \numQ{4247969}
		& \numQ{1713399}  
		\\				
\bottomrule
\end{tabular}
\end{table}

Let us begin analysing the flow by first regarding the state at $t=3$, computed with the reference refSV$_8$ method in Figure~\ref{fig:Standing_Gresho_ref}.
On the left-hand side, the vorticity is shown and it becomes clear that the viscosity in the Navier--Stokes problem smoothes out the discontinuities from the initial condition.
For the other two figures, the discrete Helmholtz--Hodge projection \eqref{eq:disc:helm}, computed with the same SV method, has been used to obtain the decomposition $\ff_h=\rb{\uu_h\ip\nabla_h}\uu_h=\HLSV{\ff_h}+\nabla\phi_h$.
One can observe that the gradient part $\nabla\phi_h$ (middle) is clearly dominating the divergence-free part $\HLSV{\ff_h}$ of $\ff_h=\rb{\uu_h\ip\nabla_h}\uu_h$.
Furthermore, $\HLSV{\ff_h}$ is very small and thus, interestingly, this flow behaves approximately like a generalised Beltrami flow.
~\\

\begin{figure}[h]
\centering
	\includegraphics[width=0.31\textwidth]{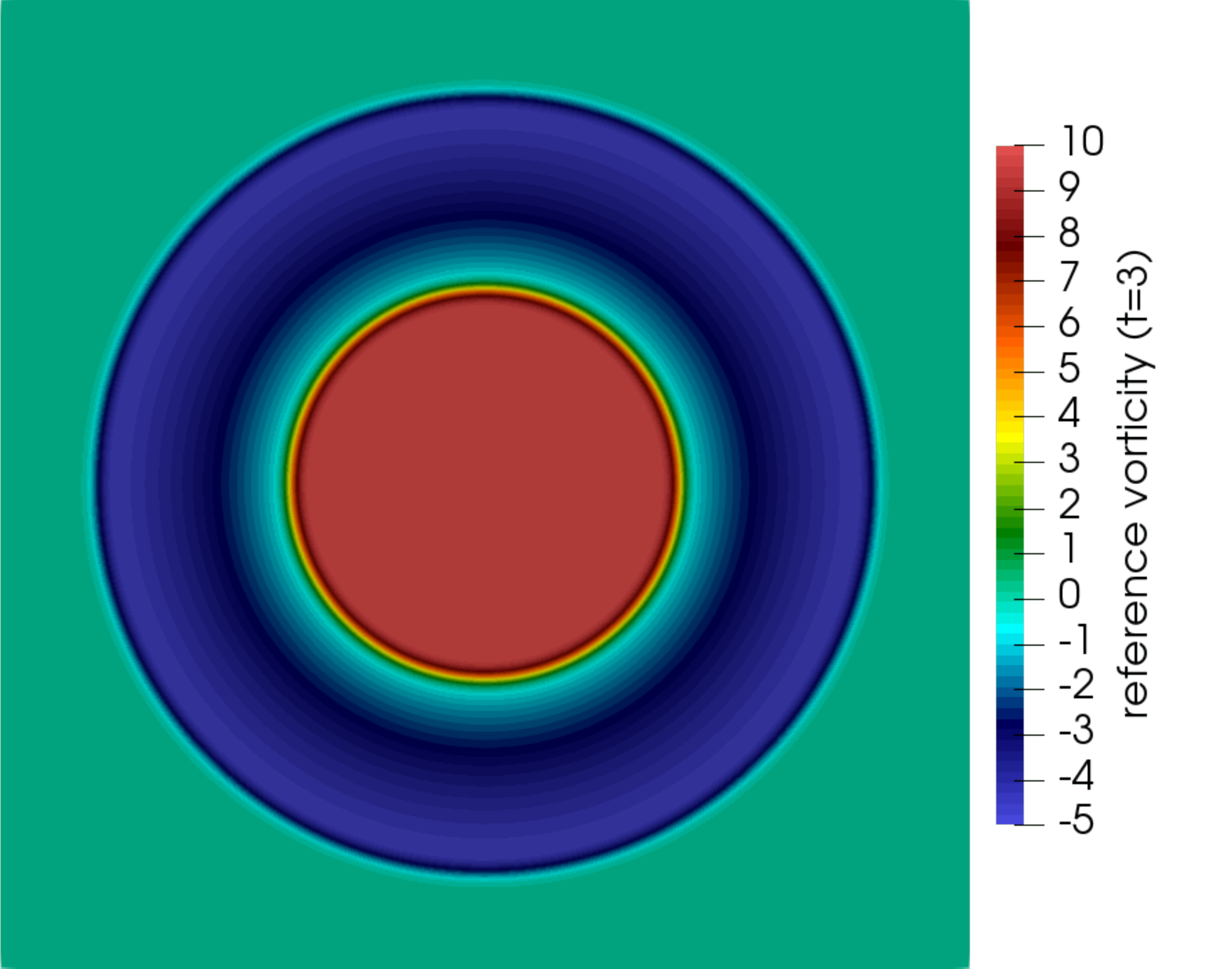} 
	\includegraphics[width=0.31\textwidth]{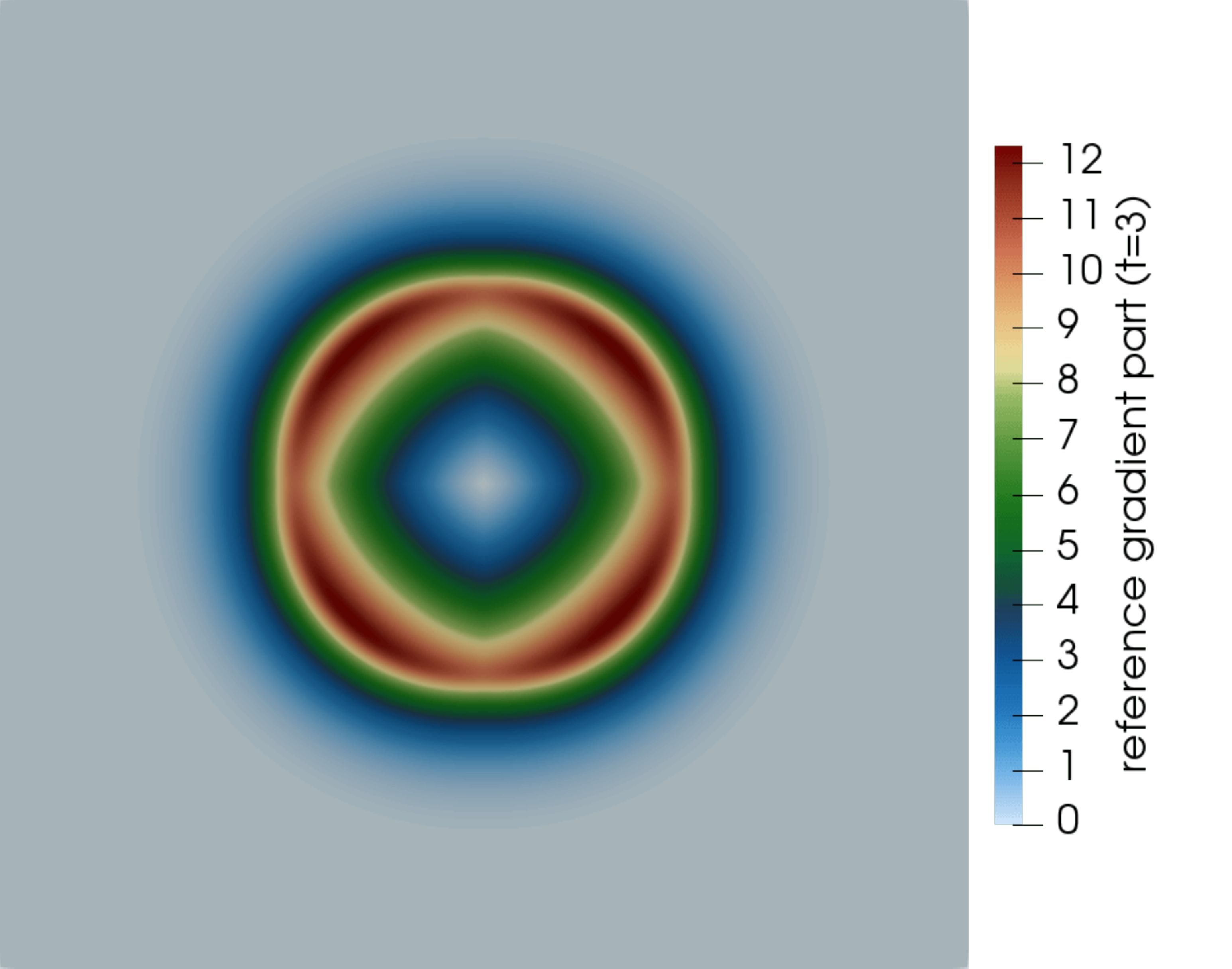}
	\includegraphics[width=0.31\textwidth]{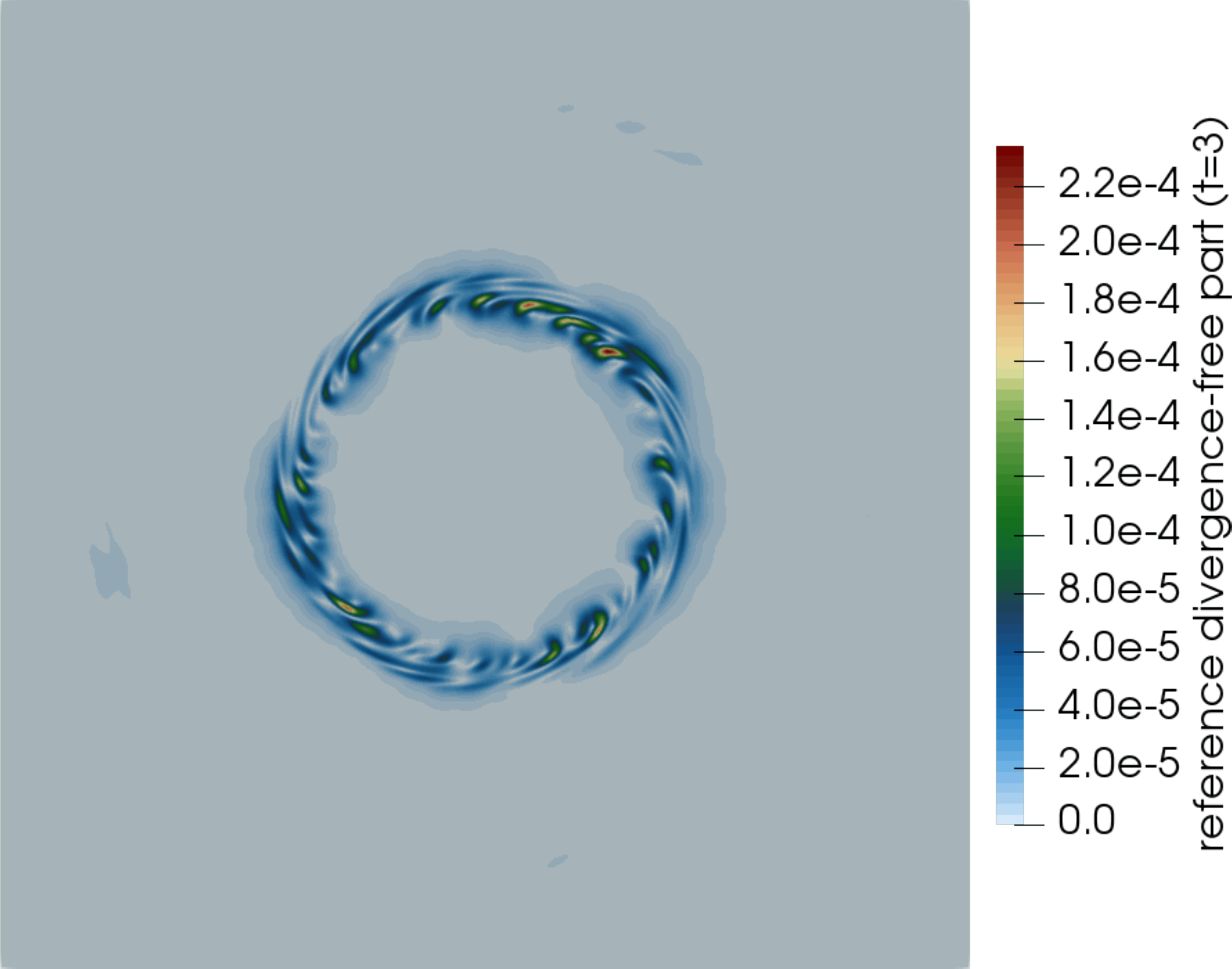}			
\caption{Reference solution for standing Gresho vortex at $t=3$. Vorticity (left), gradient part $\abs{\nabla\phi_h}_\nf{3}{2}^\nf{3}{2}$ (middle) and Helmholtz projection $\abs{\HLSV{\ff_h}}_\nf{3}{2}^\nf{3}{2}$ (right) of discrete convection term $\ff_h=\rb{\uu_h\ip\nabla_h}\uu_h$.}
\label{fig:Standing_Gresho_ref}
\end{figure}

In view of the explanations in the previous sections, we would expect that a pressure-robust method is in general superior to a non-pressure-robust method for this kind of flow.
And indeed, in the following, we show that the pressure-robust SV method is effortlessly able to preserve the vortex structure of the problem, whereas the non-pressure-robust TH method has certain difficulties with this task.

\subsection{Standing Gresho vortex}

At first, we consider the standing Gresho problem with $\ww_0=\rb{0,0}^\dag$ in \eqref{eq:GreshoInitialVortex}.
Figure~\ref{fig:Standing_Gresho_H1_vorticity} shows the vorticity of the SV$_k$ (pressure-robust) and TH$_k$ (non-pressure-robust) simulations for different polynomial orders $k$ on different meshes (corresponding to Table~\ref{tab:DOFs}) for the standing Gresho problem.
At first, one can observe that the pressure-robust method is able to preserve the structure of the initial condition, whereas the non-pressure-robust methods completely fails to give a reasonable approximation for this (seemingly) easy flow problem.
~\\

\begin{figure}[h]
\centering
	\includegraphics[width=0.225\textwidth]
		{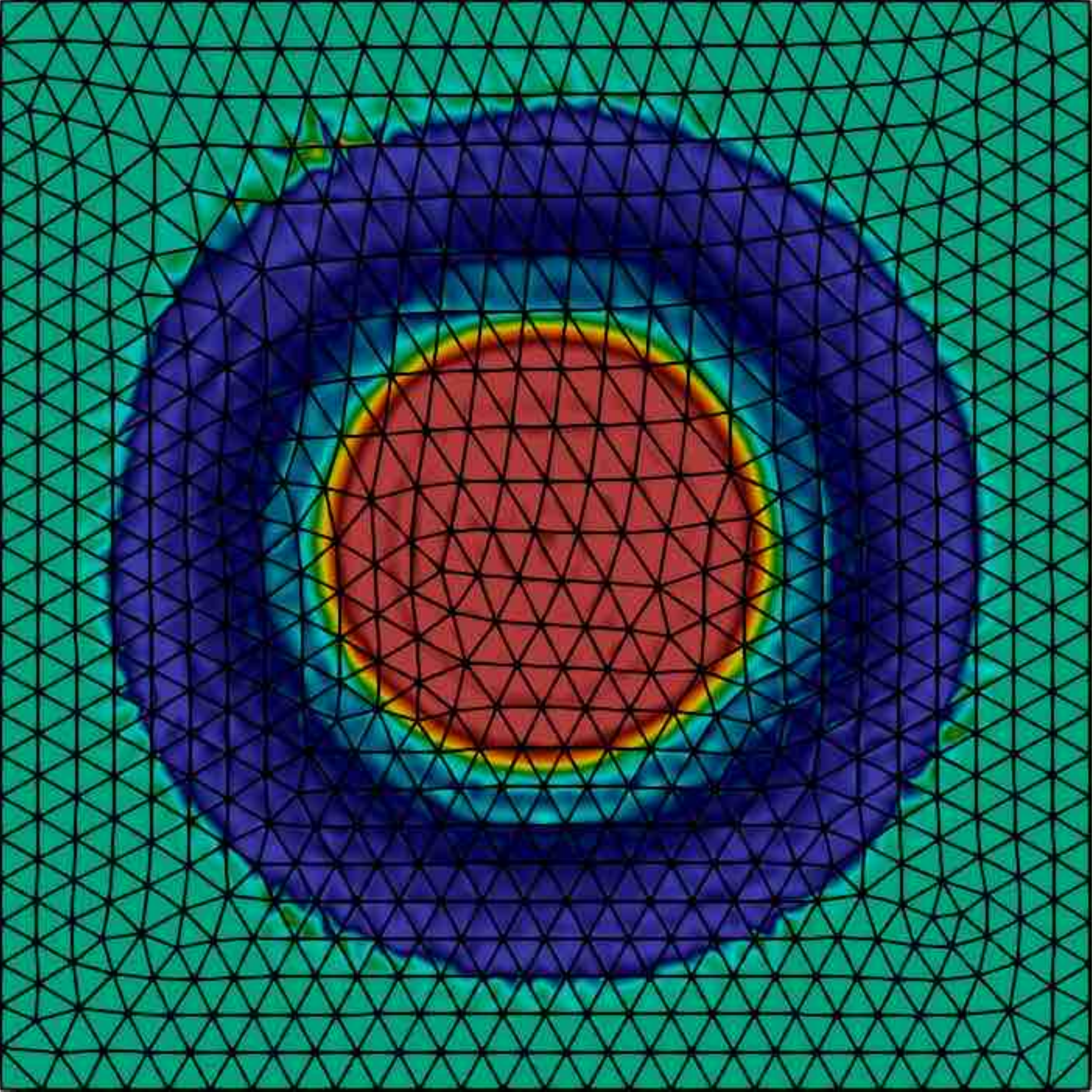} \hspace{5pt}
	\includegraphics[width=0.225\textwidth]
		{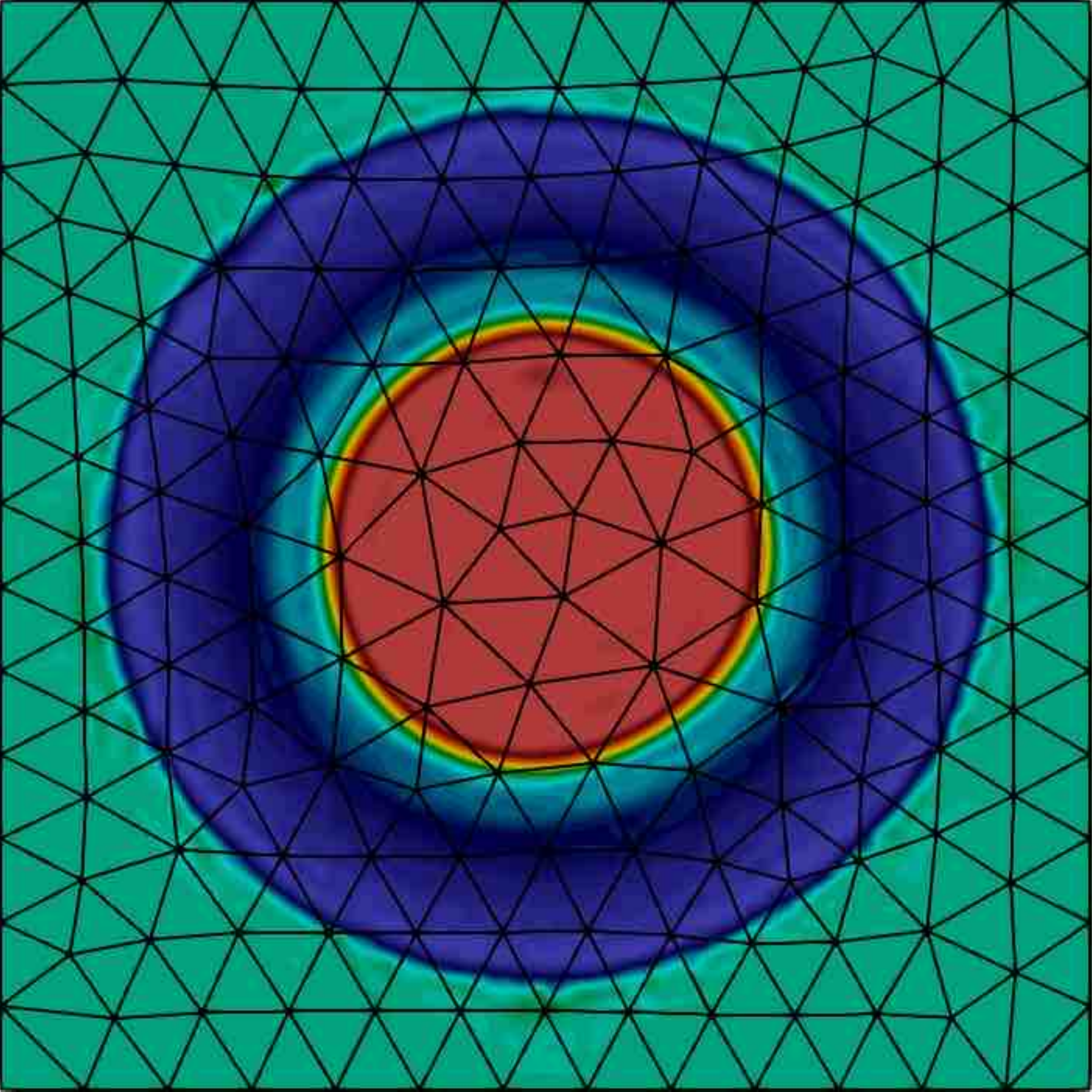} \hspace{5pt}	
	\includegraphics[width=0.225\textwidth]
		{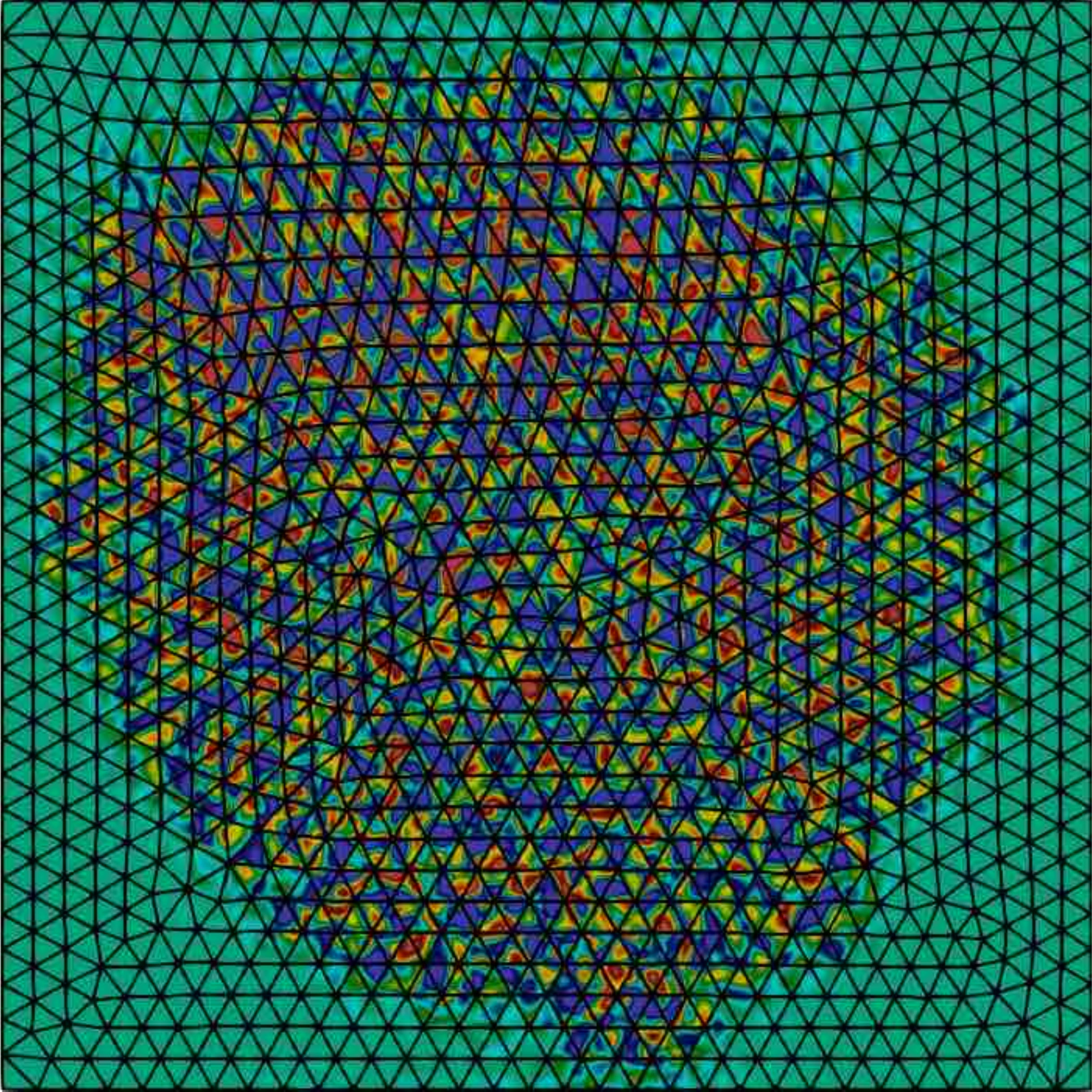} \hspace{5pt}
	\includegraphics[width=0.225\textwidth]
		{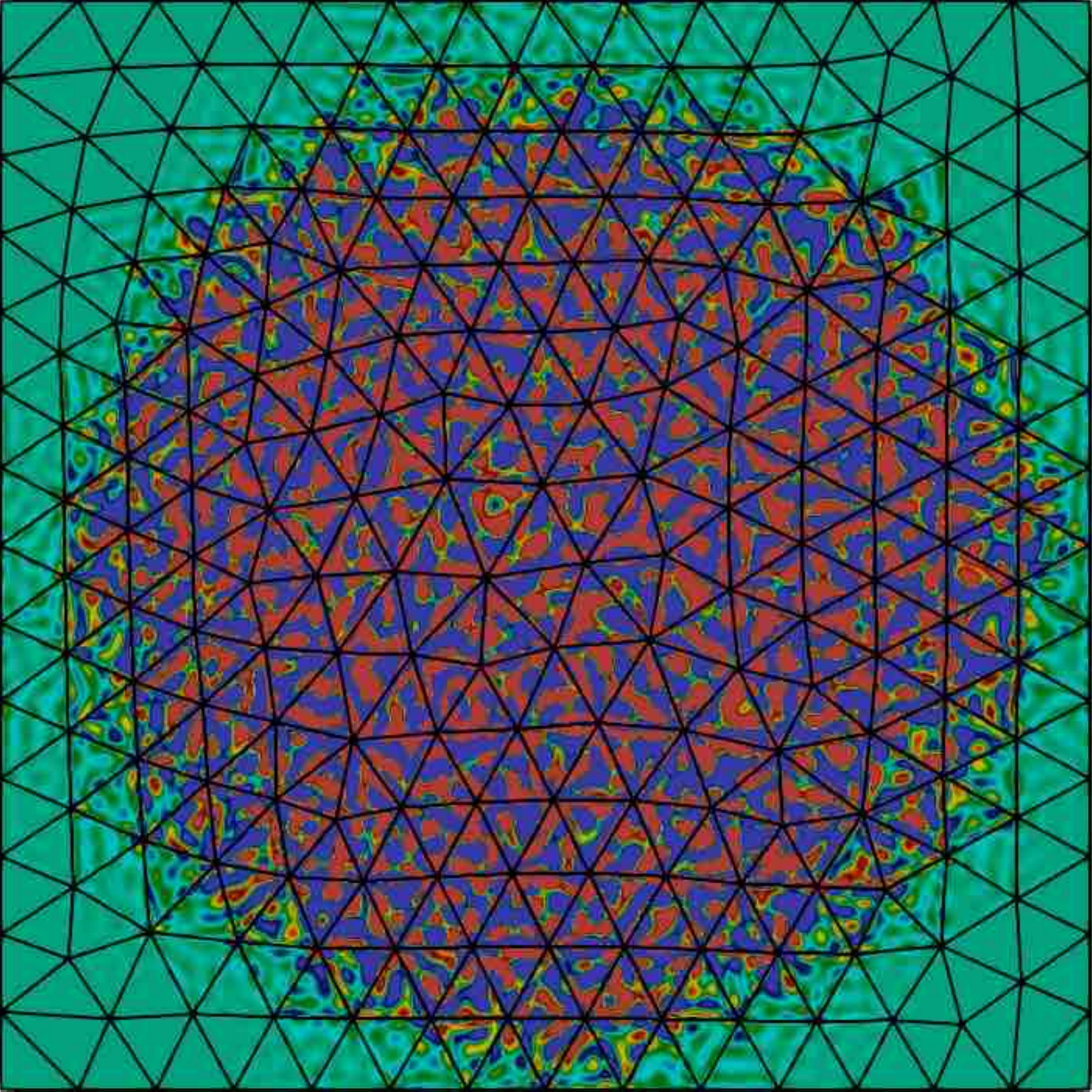} \\ \vspace*{5pt}	
	\includegraphics[width=0.3\textwidth]
		{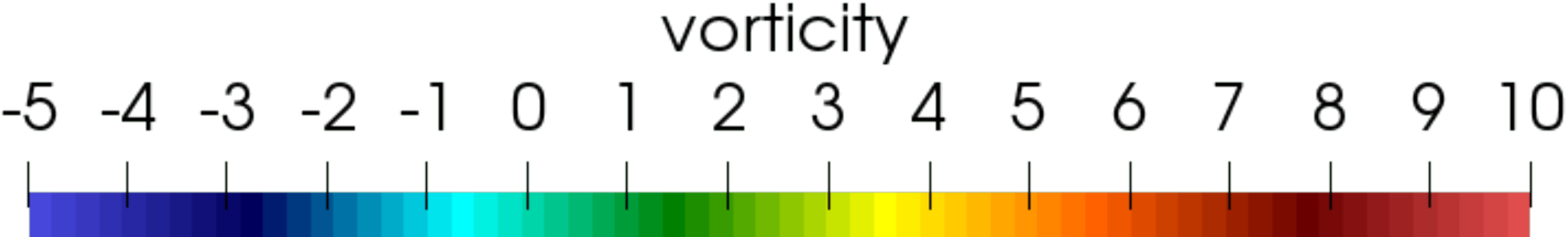}				
\caption{Vorticity of standing Gresho vortex simulations at $t=3$. SV$_4$ (left); SV$_8$ (second from left); TH$_4$ (second from right); TH$_8$ (right). The used meshes, corresponding to Table~\ref{tab:DOFs}, can also be seen.}
\label{fig:Standing_Gresho_H1_vorticity}
\end{figure}

In view of our reference solution in Figure~\ref{fig:Standing_Gresho_ref}, it becomes clear that it is of utmost importance to be able to handle the occurring curl-free gradient part of the convection term accurately.
Furthermore, one can see that the higher-order method with $k=8$ gives slightly better results in the Scott--Vogelius case.
This is very surprising, as one would generally not expect high-order regularity for the corresponding exact solution in this case.

\subsection{Moving Gresho vortex}

Now, we are dealing with the moving Gresho problem with $\ww_0=\rb{1/3,1/3}^\dag$.
After moving in the top-right direction through the periodic domain, this means that at $\tend=3$, the vortex is intended to be again centred around $\cc=\rb{0.5,0.5}^\dag$.
Before taking a closer look at the vorticity (analogous to Figure~\ref{fig:Standing_Gresho_H1_vorticity}), let us consider important flow quantities monitored over the course of the particular simulation.
~\\

\begin{figure}[h]
\centering
	\includegraphics[width=0.32\textwidth]
		{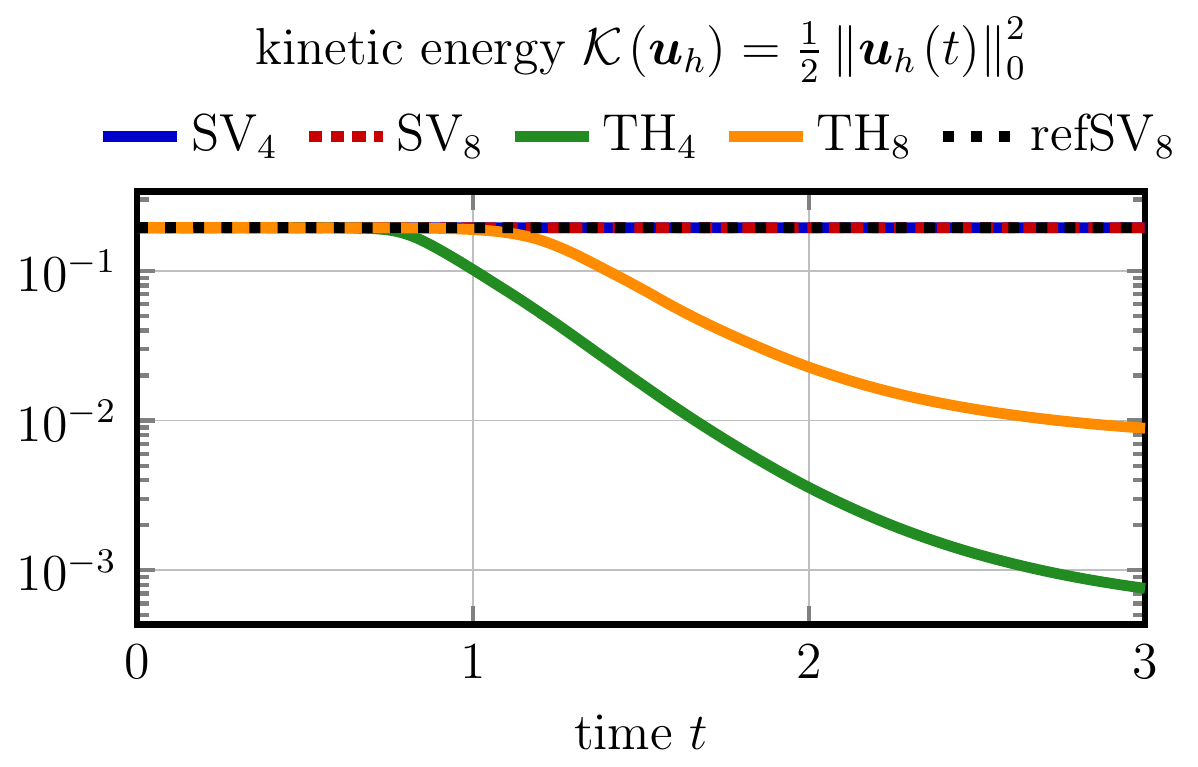} 
	\includegraphics[width=0.32\textwidth]
		{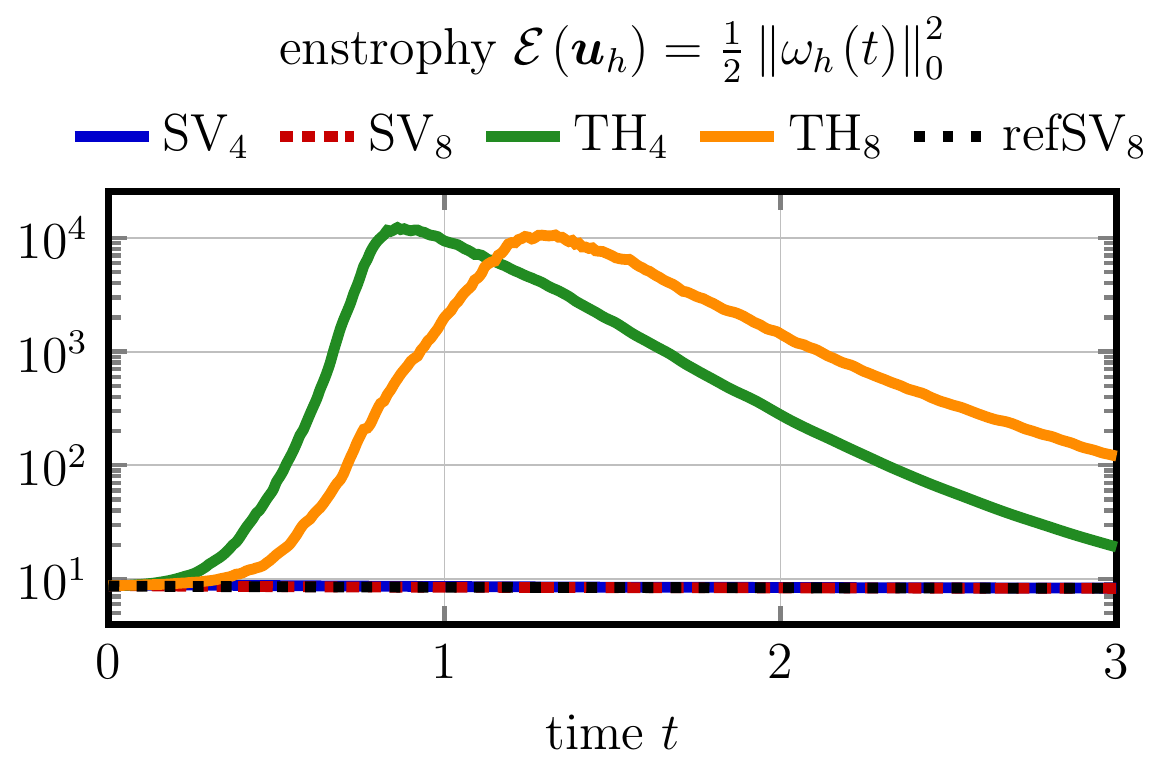} 	
	\includegraphics[width=0.32\textwidth]
		{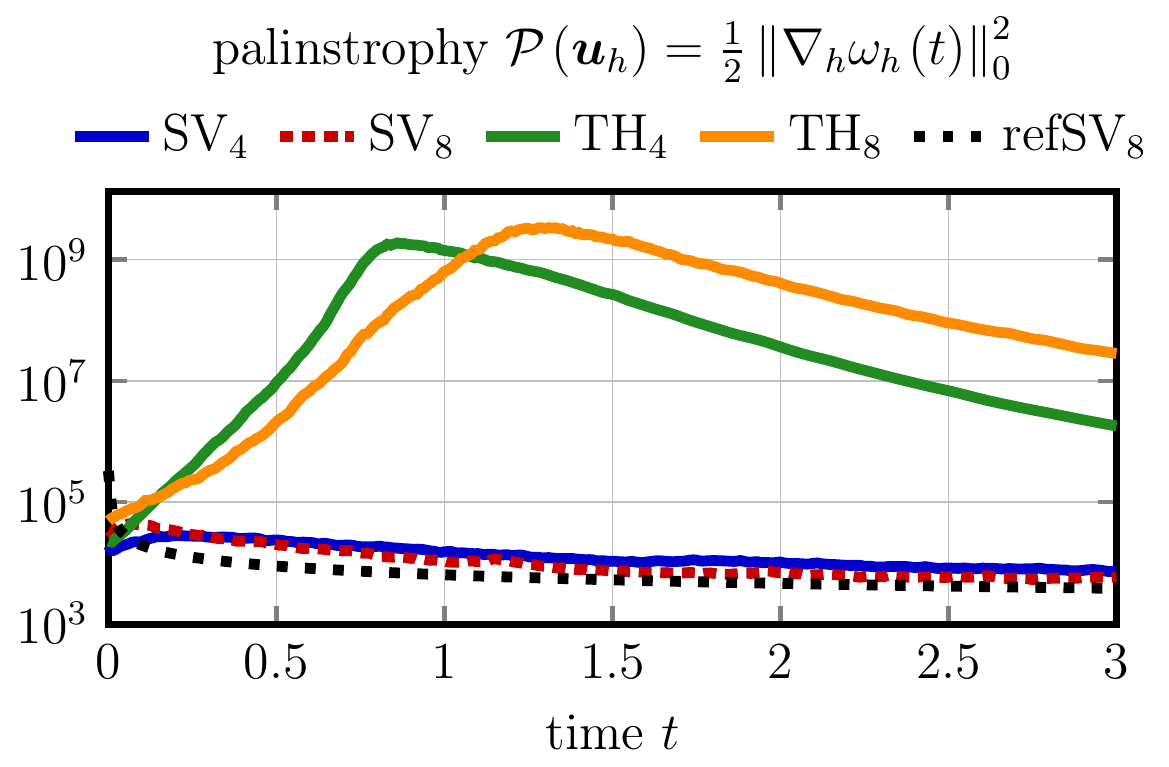} 
\caption{Evolution of kinetic energy $\Kin{\uu_h}$ (left), enstrophy $\Ens{\uu_h}$ (middle) and palinstrophy $\Pal{\uu_h}$ (right) for SV$_k$/TH$_k$ ($k\in\set{4,8}$) and refSV$_8$; see also Table~\ref{tab:DOFs}.}
\label{fig:GreshoMove_quantities}
\end{figure}

In Figure~\ref{fig:GreshoMove_quantities}, one can see the evolution of kinetic energy $\Kin{\uu_h}=\frac{1}{2}\norm{\uu_h}_\LP{2}{}^2$, enstrophy $\Ens{\uu_h}=\frac{1}{2}\norm{\omega_h}_\Lp{2}{}^2$ and palinstrophy $\Pal{\uu_h}=\frac{1}{2}\norm{\nabla_h\omega_h}_\LP{2}{}^2$ over time.
Again, we have computed a reference solution and SV$_k$/TH$_k$ solutions for $k\in\set{4,8}$ (cf.\ Table~\ref{tab:DOFs}).
For the kinetic energy, one can observe that while the non-pressure-robust solutions coincides with the pressure-robust solutions at the beginning, they show a much earlier rapid decrease in energy.
Here, TH$_8$ seems to be slightly superior to TH$_4$, but still, they are both not satisfactory.
The SV$_k$ solutions, on the other hand, preserve the kinetic energy quite well.
Concerning the enstrophy, again SV$_k$ and refSV$_8$ are undistinguishable.
On the other hand, the TH$_k$ solutions show a spontaneous increase in enstrophy which is absolutely not physical for this freely decaying 2D problem.
Lastly, concerning the palinstrophy, one can observe that the
non-pressure-robust methods yield huge values which are not close to the reference solution.
For the pressure-robust methods, we observe that while SV$_4$ and SV$_8$ do not coincide with refSV$_8$, they at least show qualitatively the same behaviour.
Furthermore, one can see SV$_8$ is slightly more precise than SV$_4$.
This comparison shows that in terms of physically interesting quantities, pressure-robust methods drastically outperform
non-pressure-robust methods for the moving Gresho problem.
~\\

In the following, we want to demonstrate that the material derivative of the moving Gresho problem is indeed approximately a (non-trivial) gradient field and thus velocity-equivalent to a zero force, as suggested by Section~\ref{sec:EulerFlows}.
To this end, we consider the discrete Helmholtz decomposition of the (discrete) material derivative
\begin{equation} \label{eq:HelmholtzMaterialSV}
	\ff_h^t = 
		\partial_t\uu_h + \rb{\uu_h\ip\nabla_h}\uu_h =
		\HLSV{\ff_h^t}+\nabla\phi_h^t,
\end{equation} 

where $\HLSV{\cdot}$ is always the pressure-robust discrete Scott--Vogelius Helmholtz projector.
Note that also for the TH method, the SV Helmholtz projector is used with the intention of being able to distinguish accurately between divergence-free and curl-free forces, also in the non-pressure-robust case.
~\\

\begin{figure}[h]
\centering
	\includegraphics[width=0.32\textwidth]
		{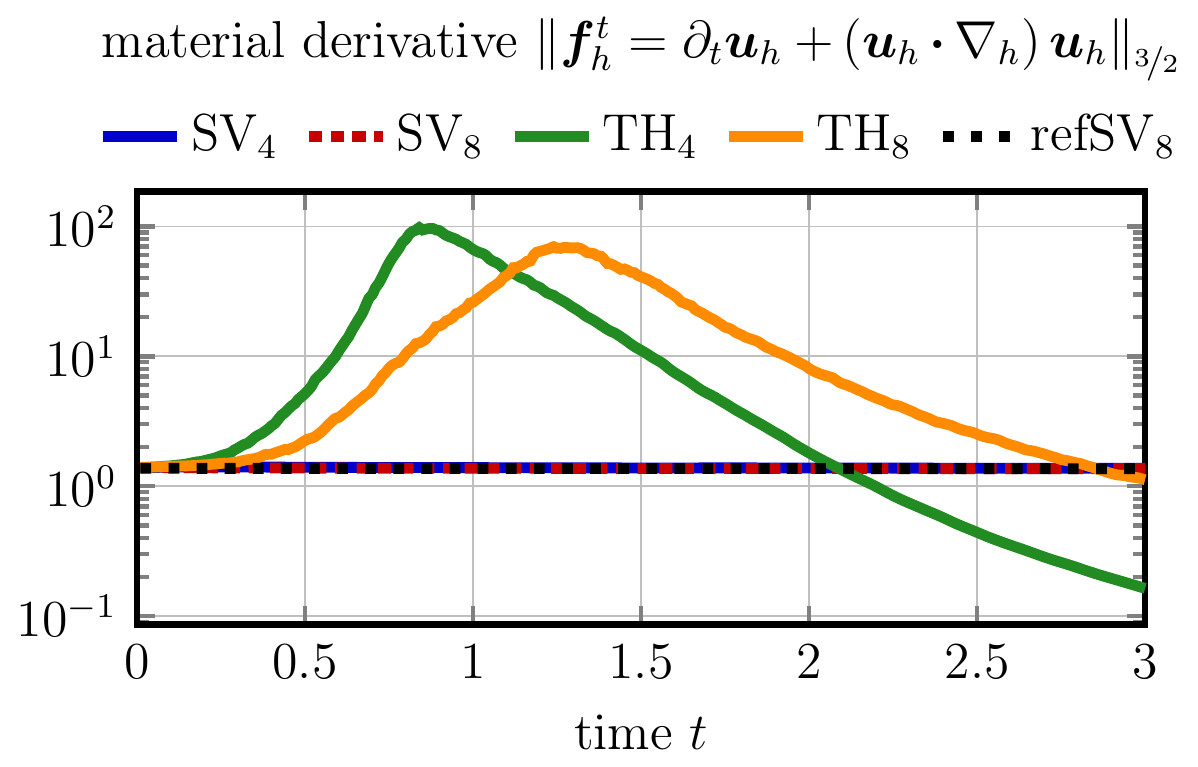} 
	\includegraphics[width=0.32\textwidth]
		{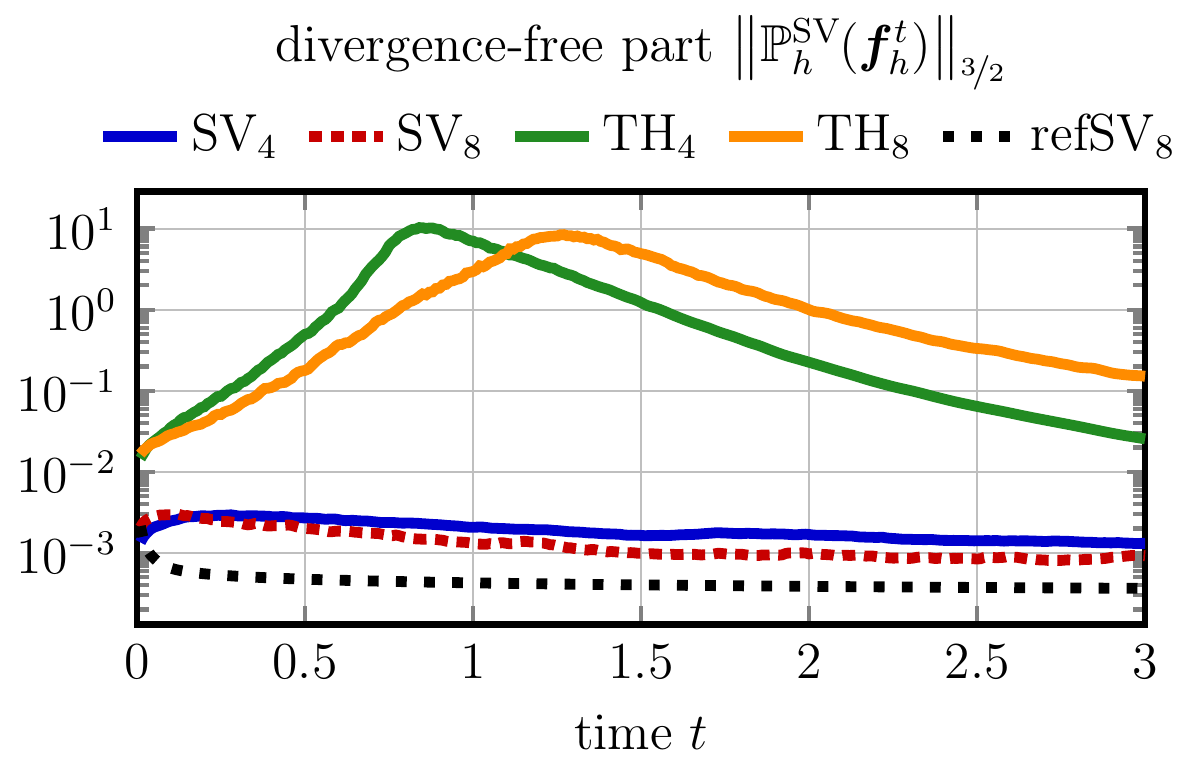} 	
	\includegraphics[width=0.32\textwidth]
		{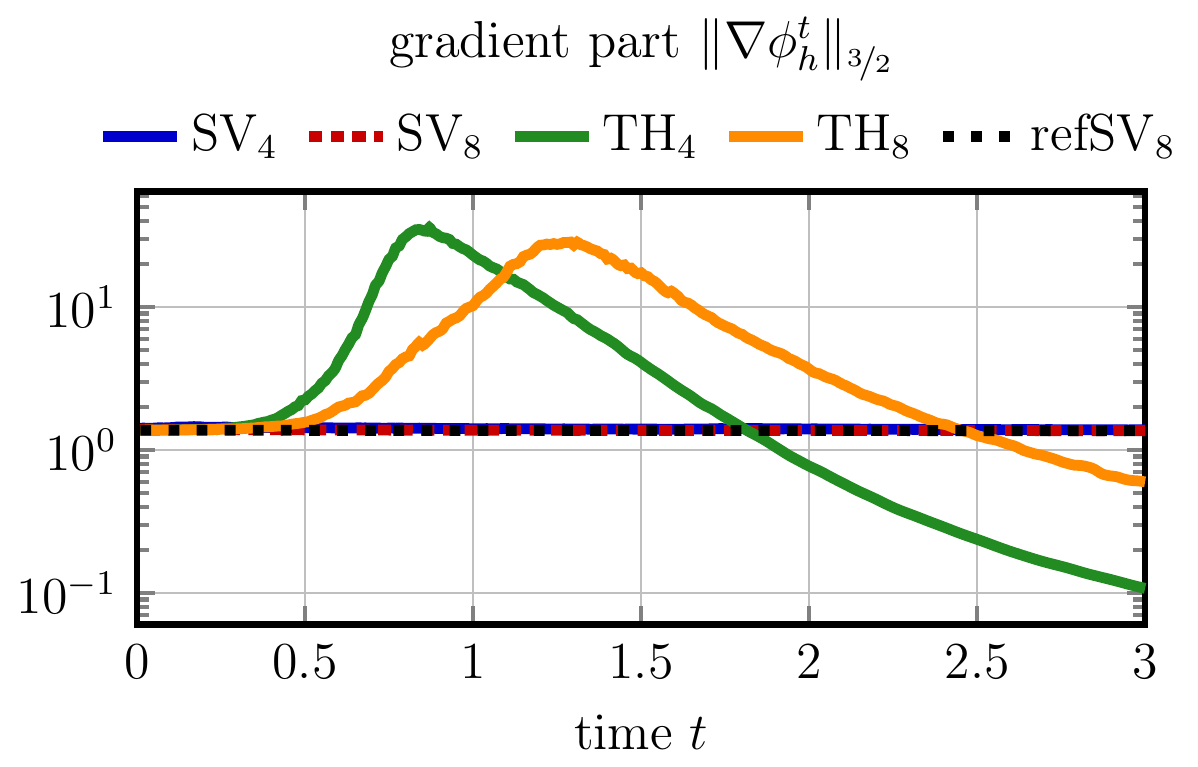} 
\caption{Evolution of $\LP{\nf{3}{2}}{}$-norms of (discrete) material derivative $\ff_h^t=\partial_t\uu_h+\rb{\uu_h\ip\nabla_h}\uu_h$ (left), divergence-free part $\HLSV{\ff_h^t}$ (middle) and gradient part $\nabla \phi_h^t$ (right) for SV$_k$/TH$_k$ ($k\in\set{4,8}$) and refSV$_8$; see also Table~\ref{tab:DOFs}.}
\label{fig:GreshoMove_Helmholtz}
\end{figure}

Now, Figure~\ref{fig:GreshoMove_Helmholtz} shows the evolution of the $\LP{\nf{3}{2}}{}$-norms of the three terms in \eqref{eq:HelmholtzMaterialSV} for the different methods from Table~\ref{tab:DOFs}.
The plot for  $\norm{\ff_h^t}_\nf{3}{2}$ shows that, compared to the reference solution refSV$_8$, the non-pressure-robust TH$_k$ TH$_k$ methods result in a completely inaccurate material derivative.
From regarding $\norm{\HLSV{\ff_h^t}}_\nf{3}{2}$ and $\norm{\nabla \phi_h^t}_\nf{3}{2}$, it becomes clear why this is the case:
At first the divergence-free part $\HLSV{\ff_h^t}$ of the reference solution is very small,
and this can be preserved well by the divergence-free SV$_k$ methods.
However, the TH$_k$ methods result in a much larger divergence-free force which is not the correct behaviour.
A similar observation holds for the gradient term $\nabla \phi_h^t$: the non-pressure-robust methods do not yield accurate results here.
~\\

\begin{figure}[h]
\centering
	\includegraphics[width=0.225\textwidth]
		{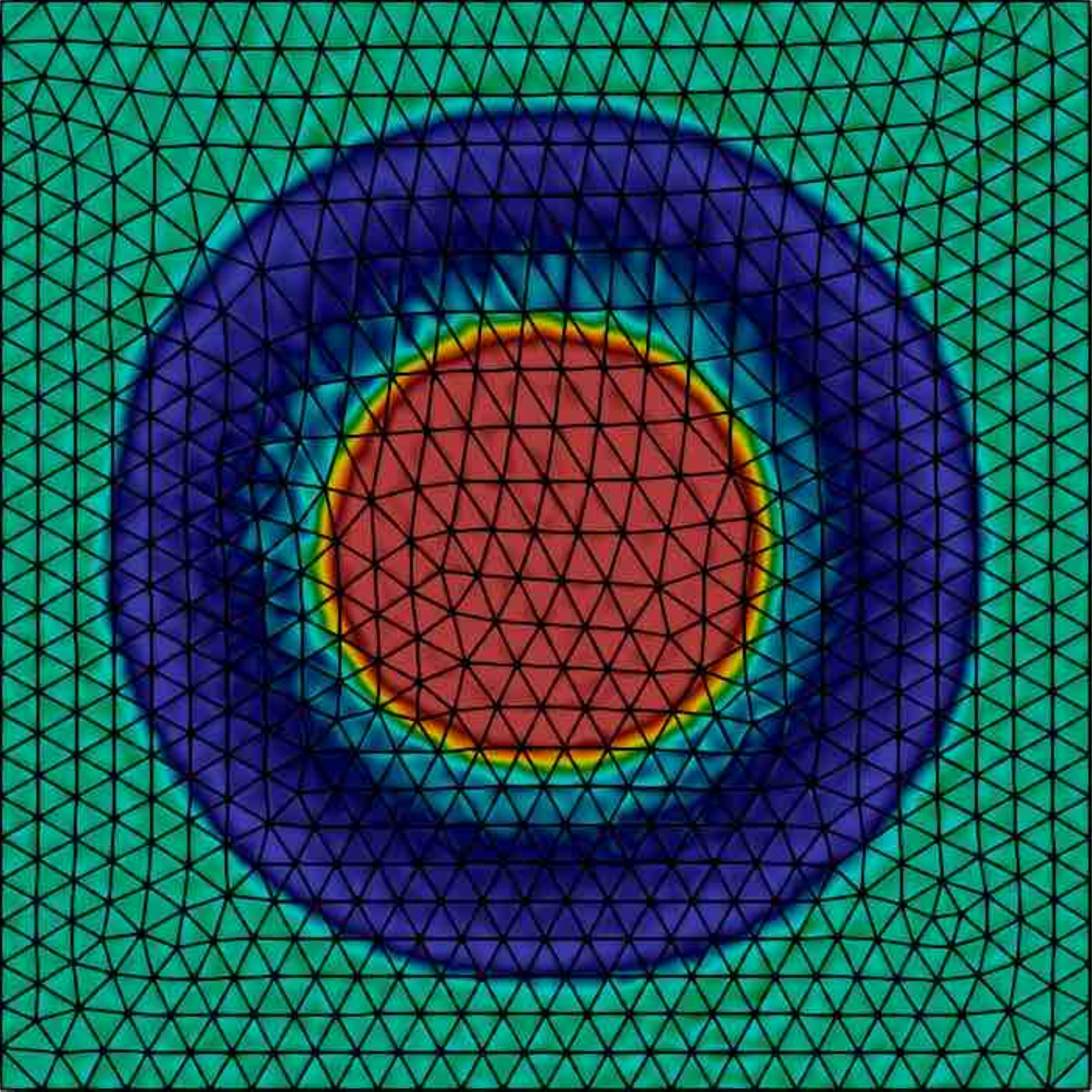} \hspace{5pt}
	\includegraphics[width=0.225\textwidth]
		{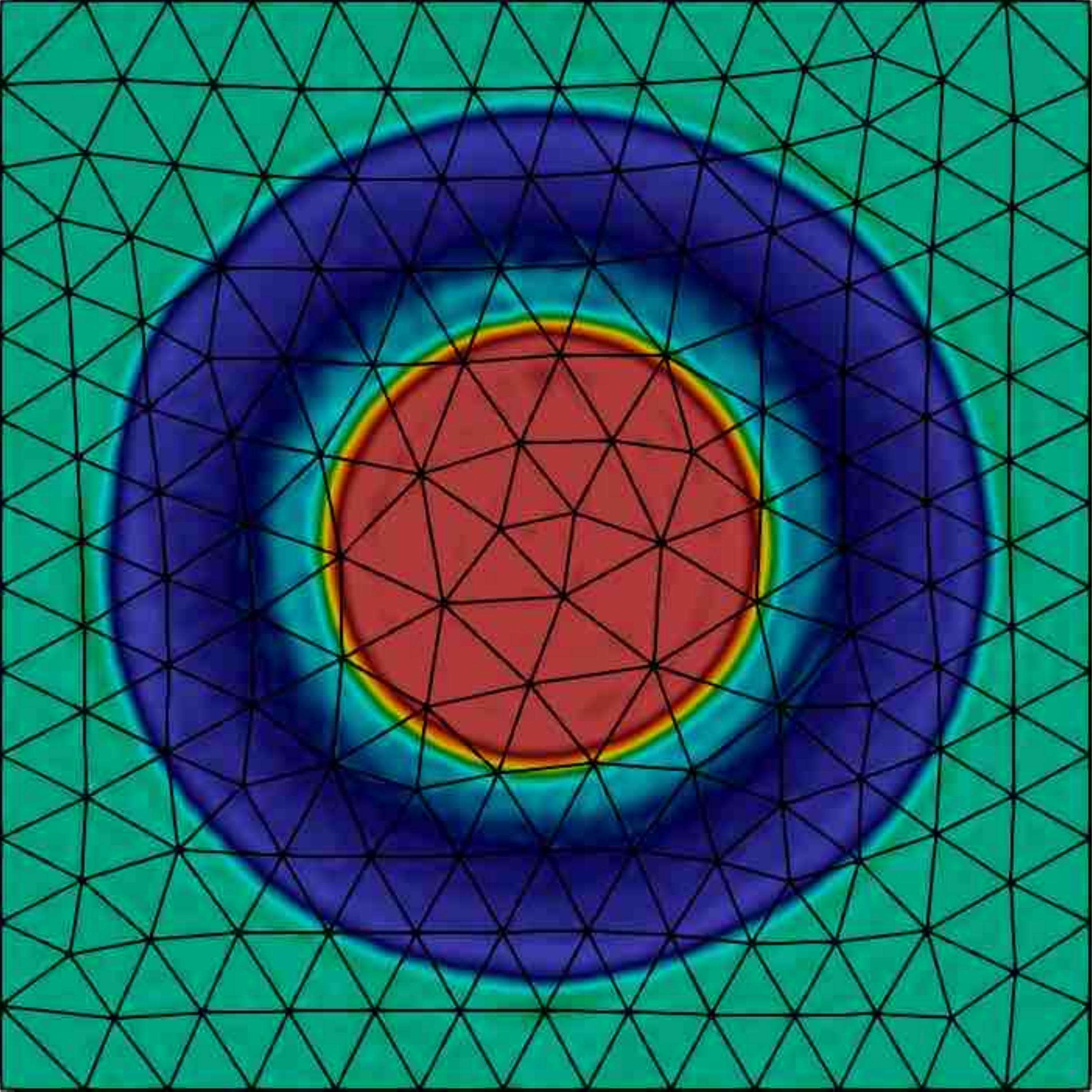} \hspace{5pt}
	\includegraphics[width=0.225\textwidth]
		{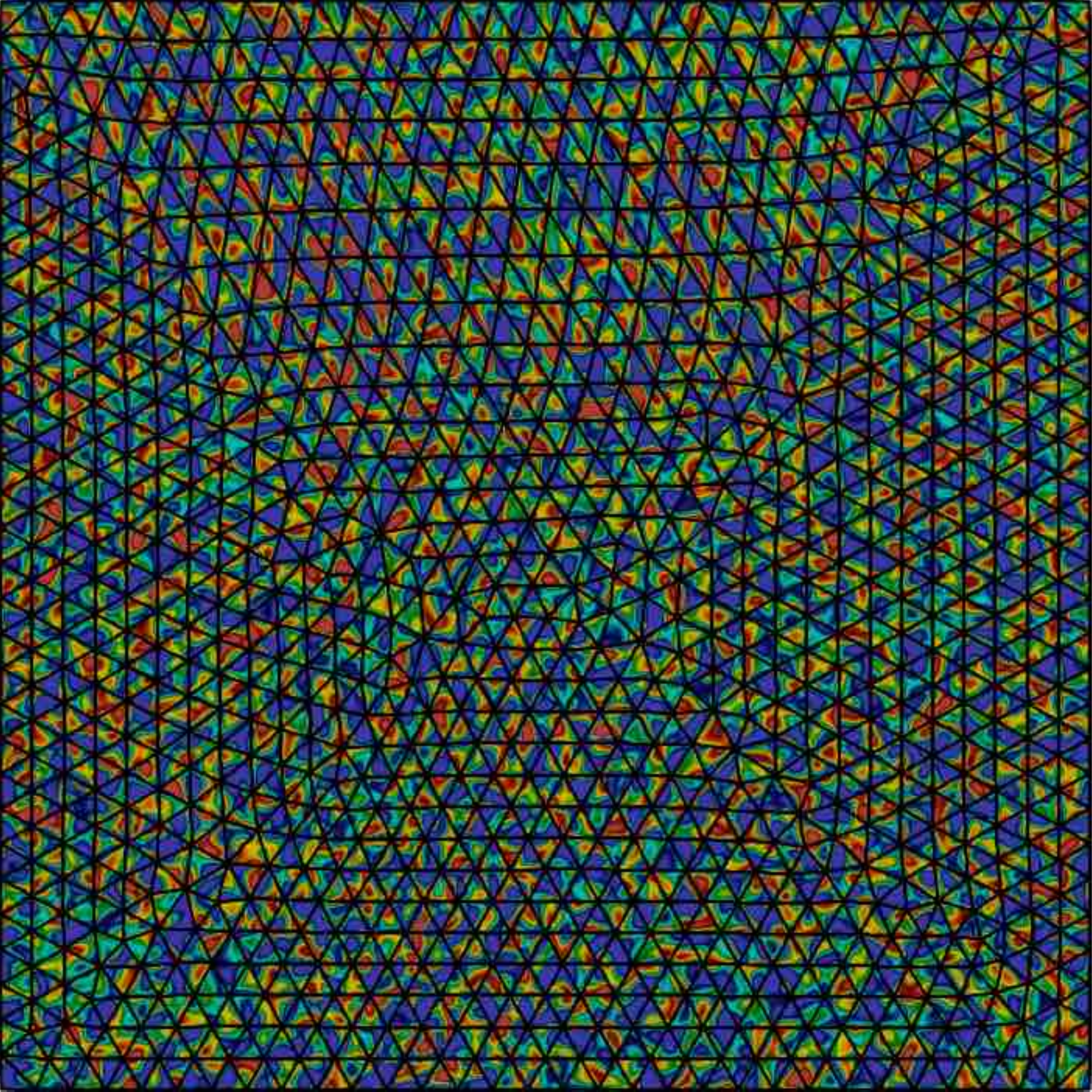} \hspace{5pt}
	\includegraphics[width=0.225\textwidth]
		{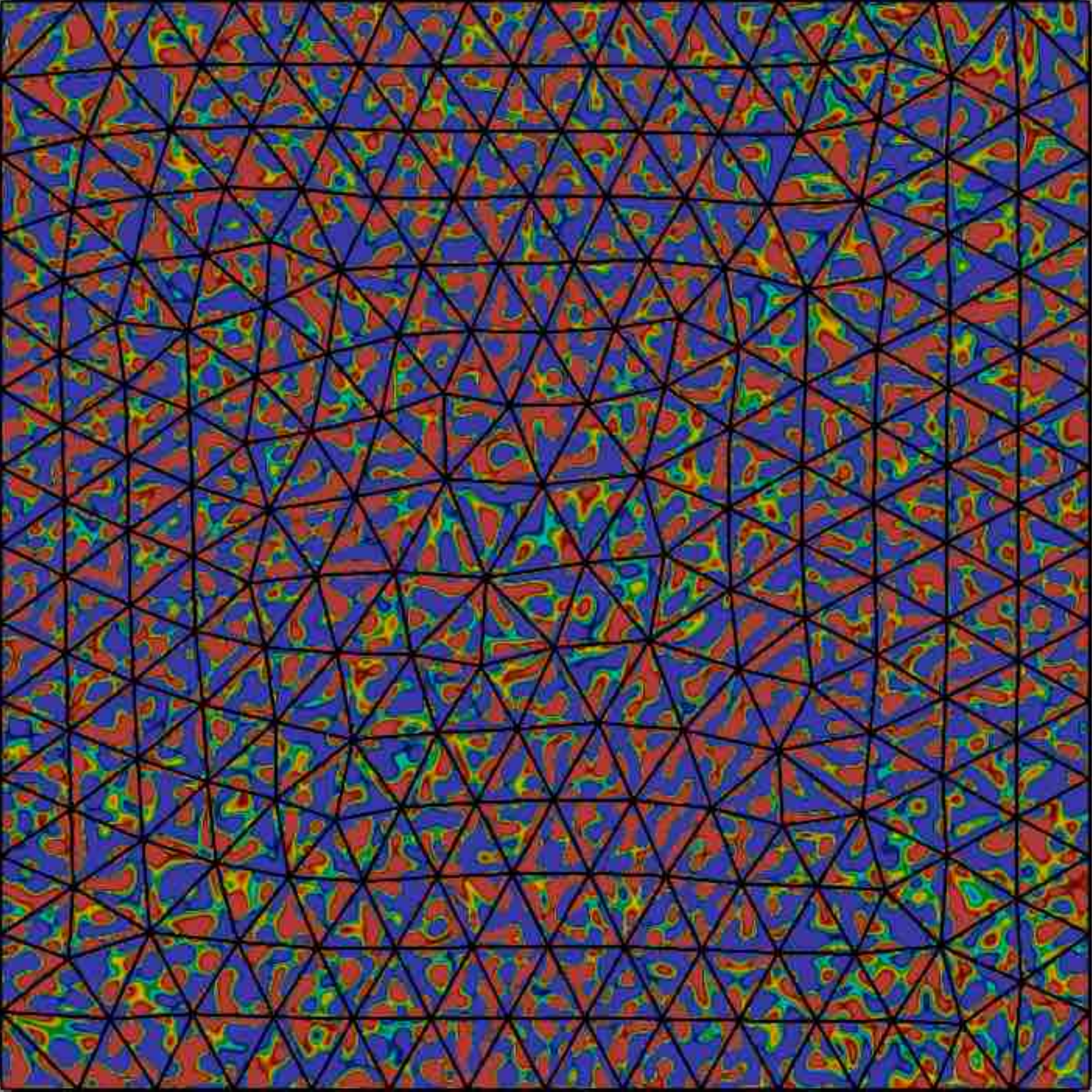} 
			\\ \vspace*{5pt}	
	\includegraphics[width=0.3\textwidth]
		{Gresho/png/pdf/Gresho-vorticity-colour-bar.pdf}				
\caption{Vorticity of moving Gresho vortex simulations at $t=3$. SV$_4$ (left); SV$_8$ (second from left); TH$_4$ (second from right); TH$_8$ (right). The used meshes, corresponding to Table~\ref{tab:DOFs}, can also be seen.}
\label{fig:Moving_Gresho_H1_vorticity}
\end{figure}

All in all, we claim that this deficiency of non-pressure-robust methods is the main reason why they are significantly inferior to
pressure-robust methods for flow problems with large gradient forces.
Let us end our investigations for the Gresho problem with vorticity plots at $t=\tend=3$ for SV$_k$/TH$_k$ ($k\in\set{4,8}$); cf.\ Figure~\ref{fig:Moving_Gresho_H1_vorticity}.
For SV$_k$, practically no difference to the standing Gresho problem can be observed, while the TH$_k$ results are even worse.
These plots underline our previous statement that pressure-robust methods are much better suited to preserve large-scale structures -- even when the particular problem at hand is only Galilean-invariant to an approximate generalised Beltrami flow, and thus more complicated than an approximate generalised Beltrami flow.

\section{\emph{\textbf{H}}(div)- and \emph{$\textbf{L}$}$^\textbf{2}$-DG finite element methods}
\label{sec:dGFEM}

In order to illustrate the numerical analysis developed in Section \ref{sec:ErrorAnalysisH1}, we will present several numerical studies that compare pressure-robust versus classical, non-pressure-robust space discretisations.
In order to make a fair and convincing comparison we will perform the numerical benchmarks from now on with Discontinuous Galerkin (DG) methods.
The reason for this choice is manifold. 
First and most important, with the software package \texttt{NGSolve} \cite{Schoeberl14}, there exists a versatile, well-established and efficient numerical implementation of plenty of different DG methods, allowing especially for high-order space discretisation.
~\\

We choose to compare an exactly divergence-free, pressure-robust $\HDIV$-conforming DG method with a classical, non-pressure-robust DG methods which is only $\LP{2}{}$-conforming (both discretely inf-sup stable).
In this setting, our second reason for using DG methods is that after choosing elementwise polynomials of order $\ku$ for the velocity, both the $\HDIV$- and the $\LP{2}{}$-DG method work with the same (discontinuous) discrete pressure space of polynomial order $k_p=\ku-1$.
Thus, both methods have a roughly comparable number of degrees of freedom and we think that a comparison between these methods is quite fair.
~\\

Third, we will deal with flows at high Reynolds numbers which makes a certain convection stabilisation desirable and/or necessary.
In the DG context, upwind techniques are well-established and $\HDIV$- and $\LP{2}{}$-conforming DG methods allow to apply exactly the same upwind stabilisation, facilitating a fair comparison. 
In Figure \ref{fig:Lattice_Stokes_comparison} the (moderately) positive effect w.r.t.\ the numerical error of a upwind stabilised versus a centred discretisation of the convection term is illustrated --- which is not at all self-evident for generalised Beltrami flows, by the way. 
~\\

Last but not least, we emphasise that a similar numerical analysis as in the case of $\HH^1$-conforming space discretisations is possible for DG methods as well, but would only involve additional technical problems due to the facet terms required for DG discretisations.
In fact, in the PhD thesis \cite{Schroeder2019} of the third author, an error analysis for $\HDIV$-DG methods can be found.
However, note that the thesis uses the discrete Stokes projection (and not the discrete Helmholtz projection) for the error splitting.
For classical DG methods, on the other hand, error analysis can already be found, for example, in the monographs \cite{Riviere08,PietroErn12}.
In this sense, we do not think that a new DG error analysis would bring any new insights, but will qualitatively look similar to the $\HM{1}{}$ estimates provided in Section~\ref{sec:ErrorAnalysisH1}, with the difference that the involved constants would be quantitatively different.
Instead, in the present work the issue of pressure-robustness in DG methods is described and demonstrated in this section. 
Again, one has to investigate the consistency errors of appropriately defined discrete Helmholtz--Hodge projectors for $\HDIV$- and $\LP{2}{}$-conforming DG methods in the $\LP{2}{}$ norm.
In fact, their behaviour is exactly the same as in the context of $\HH^1$-conforming methods, analysed in Theorems~\ref{thm:pr} and \ref{thm:classical}.

\subsection{DG formulation}
\label{sec:DG-setting}

Beginning with the standard setting in DG methods \cite{Riviere08,PietroErn12}, let $\T$ be a shape-regular FE partition (for brevity, we restrict ourselves to simplicial meshes in this work) of $\Omega$ without hanging nodes and mesh size $h=\max_{K\in\T} h_K$, where $h_K$ denotes the diameter of the particular element $K\in\T$. 
The skeleton $\F$ denotes the set of all facets of $\T$, $\FK=\set{F\in\F\colon F\subset\partial K}$ and $h_F$ represents the diameter of each facet $F\in\F$. 
Moreover, $\F=\Fi\cup\Fb$ where $\Fi$ is the subset of interior facets and $\Fb$ collects all Dirichlet boundary facets $F\subset\partial\Omega$.
Facets lying on a periodic surface of $\partial\Omega$ are treated as interior facets. 
To any $F\in\F$ we assign a unit normal vector $\nn_F$ where, for $F\in\Fb$, this is the outer unit normal vector $\nn$. 
If $F\in\Fi$, there are two adjacent elements $K^+$ and $K^-$ sharing the facet $F=\overline{\partial K^+}\cap\overline{\partial K^-}$ and $\nn_F$ points in an arbitrary but fixed direction. 
Let $\phi$ be any piecewise smooth (scalar-, vector- or matrix-valued) function with traces from within the interior of $K^\pm$ denoted by $\phi^\pm$, respectively. 
Then, we define the jump $\jmp{\cdot}_F$ and average $\avg{\cdot}_F$ operator across interior facets $F\in\Fi$ by
\begin{equation}\label{eq:DefInteriorJumps}
	\jmp{\phi}_F= \phi^+-\phi^-	
	\quad \text{and}\quad
	\avg{\phi}_F=\frac{1}{2}\rb{\phi^+ + \phi^-}.
\end{equation}
For boundary facets $F\in\Fb$ we set $\jmp{\phi}_F=\avg{\phi}_F=\phi$. 
These operators act componentwise for vector- and matrix-valued functions. 
Frequently, the subscript indicating the facet is omitted. 
~\\

Let $\VV_h/\Q_h$ be the considered discretely inf-sup stable velocity/pressure (discontinuous) FE pair.
In order to approximate \eqref{eq:TINS}, the following generic semi-discrete DG method is considered:
\begin{subequations} \label{eq:DG-FEM}
	\begin{empheq}[left=\empheqlbrace]{align} 
	\text{Find }\rb{\uu_h,p_h}&\colon\rsb{0,\tend}\to\VV_h\times\Q_h
		\text{ with }\uu_h\rb{0}=\uu_{0h}
		\text{ s.t., }\forall\,\rb{\vv_h,q_h}\in\VV_h\times\Q_h,\\
	\rb{\partial_t\uu_h,\vv_h}
		&+\nu a_h\rb{\uu_h,\vv_h} 
		+ c_h\rb{\uu_h;\uu_h,\vv_h}
		+ b_h\rb{\vv_h,p_h}
		+ b_h\rb{\uu_h,q_h} \\
		&= \rb{\ff,\vv_h}
			+ \nu a_h^\partial\rb{\gD;\vv_h}
			+ c_h^\partial\rb{\gD;\uu_h,\vv_h}
			+ b_h^\partial\rb{\gD,q_h}.
	\end{empheq} 
\end{subequations}

Here, $\uu_{0h}$ denotes a suitable approximation of the initial velocity $\uu_0$.
In the following, based on \cite{Riviere08,PietroErn12}, we introduce the various terms which appear in \eqref{eq:DG-FEM}.
Note that only the particular choice of the discrete velocity space $\VV_h$ will distinguish the pressure-robust from the non-pressure-robust method in the end.
~\\

Denote the broken gradient by $\nabla_h$.
For the discretisation of the diffusion term, we choose the symmetric interior penalty method with a sufficiently large penalisation parameter $\sigma >0$:
\begin{subequations}
\begin{align}
	a_h\rb{\uu_h,\vv_h} 
		&=	\int_\Omega \nabla_h \uu_h \Fip \nabla_h \vv_h \dx
			+\sum_{F\in\F} \frac{\sigma}{h_F} \int_F \jmp{\uu_h} \ip \jmp{\vv_h} \ds \\
		&\quad -\sum_{F\in\F} \int_F \avg{\nabla_h \uu_h} \nn_F \ip \jmp{\vv_h} \ds
			-\sum_{F\in\F} \int_F \jmp{\uu_h} \ip \avg{\nabla_h \vv_h} \nn_F \ds \\
	a_h^\partial\rb{\gD;\vv_h} 
		&=	\sum_{F\in\Fb} \frac{\sigma}{h_F} \int_F \gD \ip \vv_h \ds
		 	-\sum_{F\in\Fb} \int_F \gD \ip \rb{\nabla_h\vv_h} \nn_F \ds		.		
\end{align}
\end{subequations}

Using the broken divergence $\DIVh$, the pressure-velocity coupling is realised by
\begin{subequations}
\begin{align}
	b_h\rb{\uu_h,q_h} 
		&= -\int_\Omega q_h\rb{\DIVh \uu_h} \dx
		+\sum_{F\in\F}\int_F \rb{\jmp{\uu_h}\ip\nn_F}\avg{q_h} \ds, \\
	b_h^\partial\rb{\gD;q_h}
		&=	\sum_{F\in\Fb}\int_F \rb{\gD\ip\nn}q_h \ds.	
\end{align}	
\end{subequations}

For the nonlinear inertia term, we decide to use the following skew-symmetrised (upwind) discretisation in convection form:
\begin{subequations}
\begin{align}
	c_h\rb{\ww_h;\uu_h,\vv_h}	
		&= 
			\int_\Omega \rb{\ww_h\ip\nabla_h}\uu_h\ip\vv_h \dx
			+ \frac{1}{2}\int_\Omega \rb{\DIVh\ww_h}\uu_h\ip\vv_h \dx \\ 
		&\quad - \sum_{F\in\Fi} \int_F \rb{\avg{\ww_h}\ip\nn_F} \jmp{\uu_h}\ip\avg{\vv_h} \ds 
		- \frac{1}{2}\sum_{F\in\F} \int_F \rb{\jmp{\ww_h}\ip\nn_F} \avg{\uu_h\ip\vv_h} \ds \\ 
		&\quad +\sum_{F\in\Fi} \int_F \frac{\theta}{2} \abs{\avg{\ww_h}\ip\nn_F}\jmp{\uu_h}\ip\jmp{\vv_h} \ds\\
	c_h^\partial\rb{\gD;\uu_h,\vv_h}
		&= -\frac{1}{2}\sum_{F\in\Fb} \int_F \rb{\gD\ip\nn} \rb{\uu_h\ip\vv_h} \ds.
\end{align}	
\end{subequations}

Here, the parameter $\theta\in\set{0,1}$ decides whether the considered method term uses an upwind stabilisation ($\theta=1$) for the convection term or not ($\theta=0$).
~\\

Concerning the particular choice of FE spaces, let $\Pk{k}{\rb{K}}$ denote the local space of all polynomials on $K$ with degree less or equal to $k$.
Then, given $k\geqslant 2$, the pressure space for both the $\HDIV$- and the $\LP{2}{}$-DG method coincides:
\begin{equation}
	\Q_h=
		\set{q_h\in\Lpz{2}{\OMEGA}\colon \restr{q_h}{K}\in\Pk{k-1}{\rb{K}},~\forall\, K\in\T},
\end{equation}
i.e., it holds $k_p=k-1$.
While the discrete pressure spaces for the
$\HDIV$- and $\LP{2}{}$ conforming DG methods are the same,
the velocity spaces, however, differ.
Recalling
\begin{equation*}
\Hdiv = \set{\vv\in\LP{2}{\OMEGA}\colon \DIV\vv\in\Lp{2}{\OMEGA}},
\end{equation*}

the FE velocity spaces are defined as follows:
\begin{subequations}
\begin{align}
	\text{$\HDIV$-DG}\colon \qquad
		\VV_h &= \set{\vv_h\in\Hdiv\colon\restr{\vv_h}{K}\in \PPk{k}{\rb{K}},~\forall\, K\in\T;~\restr{\rb{\vv_h-\gD}\ip\nn}{\partial\Omega}=0} \label{eq:VelSpaceHdiv},\\
	\text{$\LP{2}{}$-DG}\colon \qquad  
		\VV_h &=	\set{\vv_h\in\LP{2}{\OMEGA}\colon\restr{\vv_h}{K}\in \PPk{k}{\rb{K}},~\forall\, K\in\T}.
		\label{eq:VelSpaceL2}
\end{align}	
\end{subequations}

Note that the $\HDIV$ space is thus based on the Brezzi--Douglas--Marini element \cite{BoffiEtAl13} and the resulting divergence-free DG method is strongly related to \cite{CockburnEtAl07}.
Finite element error analysis for the $\LP{2}{}$-DG method can be found, for example, in \cite{Riviere08,PietroErn12}.
For the $\HDIV$-DG method, we refer to \cite{SchroederEtAl18}.
In the DG setting, discretely divergence-free functions are defined using the pressure-velocity coupling $b_h$ by
\begin{align*}
	\VV_h^\dvg 
		= \set{\vv_h\in\VV_h\colon b_h\rb{\vv_h,q_h}=0,~\forall\,q_h\in\Q_h},
\end{align*}

where we note that for the $\HDIV$ method, $\vv_h\in\VV_h^\dvg$ follows $\DIV\vv_h=0$ pointwise.

\begin{remark}
Deriving an analogous statement to Lemma \ref{lem:conv:nonlinear:term} for DG methods is straightforward, but technically demanding. 
The necessary compactness arguments for DG methods are described in \cite{PietroErn12}.
\end{remark}

All computations in this work have been done with the high-order finite element library \texttt{NGSolve} \cite{Schoeberl14}. 

\subsection{Discrete DG Helmholtz--Hodge projectors}

Every discretely inf-sup stable numerical method for the incompressible (Navier--)Stokes equations is intrinsically connected to a particular discrete Helmholtz--Hodge projector.
The aim of this subsection is to investigate the properties of the discrete DG Helmholtz--Hodge projectors of the $\HDIV$- and $\LP{2}{}$-DG methods.
An analogous discussion for the discrete $\HM{1}{}$ projectors has been done in Subsection \ref{sec:DiscreteH1Projectors}.
~\\

The general definition of the discrete Helmholtz--Hodge projector $\helm_h$ of a function $\gbld \in\LP{2}{}$ is given by
\begin{equation} \label{eq:DiscDGHelm}
  \helm_h\colon\LP{2}{\OMEGA} \to \VV_h^\dvg, \quad
		\helm_h(\gbld) = \argmin_{\vv_h\in\VV_h^\dvg} \norm{\gbld-\vv_h}_\LP{2}{\OMEGA}.
\end{equation}

A suitable finite element method for this problem uses the already known pressure-velocity coupling form $b_h$ and additionally, a mass bilinear form defined by $m_h\rb{\uu_h,\vv_h}=\int_\Omega \uu_h\ip\vv_h \dx$.
The discrete weak form reads as follows:
\begin{subequations} \label{eq:DiscreteHelmholtzLeray}
	\begin{empheq}[left=\empheqlbrace]{align} 
	&\text{Find }\rb{\HLh{\gbld},\phi_h}\in\VV_h\times\Q_h
		\text{ s.t., }\forall\,\rb{\vv_h,q_h}\in\VV_h\times\Q_h,\\
		&m_h\rb{\HLh{\gbld},\vv_h} 
		+ b_h\rb{\vv_h,\phi_h}
		+ b_h\rb{\HLh{\gbld},q_h} 
		= \rb{\gbld,\vv_h}.
	\end{empheq} 
\end{subequations}

Now, choosing the discrete `velocity' space according to the $\HDIV$-DG choice \eqref{eq:VelSpaceHdiv} leads to an exactly divergence-free discrete Helmholtz--Hodge projector called $\HLhd{\gbld}$.
On the other hand, choosing $\VV_h$ according to the $\LP{2}{}$-DG choice \eqref{eq:VelSpaceL2}, the resulting discrete Helmholtz--Hodge projector is denoted by $\HLhdc{\gbld}$.
Note that while $\DIV\HLhd{\gbld}=0$, the $\LP{2}{}$-DG projector $\HLhdc{\gbld}$ is not divergence-free.
~\\

Let us now quantify the difference between these two discrete DG Helmholtz projectors more carefully.
For the divergence-free $\HDIV$-DG method, the analogue of Lemma
\ref{lem:DiscHelmDivFreeH1} is the following.
\begin{lemma}
For the pressure-robust (divergence-free) $\HDIV$-DG method, for all gradient fields $\nabla \psi$
with $\psi \in H^1\OMEGA$ it holds
\begin{equation*}
	\helm_h(\nabla \psi) = \zero.	
\end{equation*}
\end{lemma}

\begin{proof}
In this case, due to
$\vv_h \in \VV_h^\dvg$ it follows $\DIV\vv_h=0$ pointwise.
Thus, for $\nabla \psi \in \LP{2}{\OMEGA}$ it holds for all $\vv_h \in \VV_h^\dvg$,
\begin{equation*}
	(\nabla \psi, \vv_h) = -(\psi, \DIV \vv_h) = 0.	
\end{equation*}
\end{proof}

For the non-pressure-robust $\LP{2}{}$-DG method, the analogue of Lemma \ref{lem:DiscHelmLerayNonDivFreeH1} is again less favourable.

\begin{lemma} \label{lem:DiscHelmLerayNonDivFreeDG}%
For the non-pressure-robust $\LP{2}{}$-DG method, for all gradient fields $\nabla \psi$ with $\psi \in \Hm{k_p+1}{\OMEGA}$, it holds
\begin{equation}
	\norm{\helm_h(\nabla \psi)}_\LP{2}{\OMEGA}
		\leqslant  C h^{k_p} |\psi|_\Hm{k_p+1}{\OMEGA}.
\end{equation}
\end{lemma}

\begin{proof}
For all $q \in H^1\OMEGA$ it holds for all $\vv_h\in\VV_h^\dvg$
\begin{equation*}
	\rb{\nabla q,\vv_h}
		= -\int_\Omega q\rb{\DIVh \vv_h} \dx
	+\sum_{F\in\F}\int_F q \rb{\jmp{\vv_h}\ip\nn_F} \ds\\
		= b_h\rb{\vv_h,q}.
\end{equation*}

For $L_h \psi \in H^1\OMEGA \cap Q_h$ it holds thus $(\nabla (L_h \psi), \vv_h) = 0$ due to \eqref{eq:DG-FEM}, leading to
\begin{equation*}
  (\nabla \psi, \vv_h) = (\nabla (\psi - L_h \psi), \vv_h)
   \leqslant  \norm{\nabla (\psi - L_h \psi)}_{\LP{2}{}} \norm{\vv_h}_{\LP{2}{}}	
\end{equation*}

and the result is proved like in the $\HH^1$-conforming case in Lemma \ref{lem:DiscHelmLerayNonDivFreeH1}.
\end{proof}

\section{Numerical experiments with known exact solution}
\label{sec:ExperimentsKnownSol}

In this section, we mainly use the previously described $\HDIV$- and $\LP{2}{}$-DG methods to investigate incompressible flow problems at high Reynolds number where the exact solution is known to be a (generalised) Beltrami flow.
The accuracy of the corresponding results is measured in the $\LP{2}{}$ norm exclusively because the errors in $\HM{1}{}$ do not provide any more insight.
More precisely, this means that the errors in the $\HM{1}{}$ norm qualitatively show exactly the same behaviour as the ones measured in the $\LP{2}{}$ norm.
Furthermore, for the 2D lattice flow problem, a brief comparison of pressure-robust and non-pressure-robust $\HM{1}{}$-conforming methods, analogous to Section~\ref{sec:GreshoH1} can be found as well.

\subsection{2D planar lattice flow} 
\label{sec:2DLatticeFlow}

Let us now compare the performance of the pressure-robust $\HDIV$- and non-pressure-robust $\LP{2}{}$-DG methods in terms of accuracy and efficiency in a two-dimensional setting.
In order to do so, we fix $\nu=\num{e-5}$ and solve a problem with $\ff\equiv\zero$, where the exact solution is known and given by
\begin{equation} \label{eq:LatticeExact}
	\uu_0\rb{\xx}=\begin{bmatrix}
		\sin\rb{2\pi x_1}\sin\rb{2\pi x_2}\\
		\cos\rb{2\pi x_1}\cos\rb{2\pi x_2}	
	\end{bmatrix},\quad
	\uu\rb{t,\xx} = \uu_0\rb{\xx}e^{-8\pi^2 \nu t }.
\end{equation}	

We consider the domain $\Omega=\rb{0,1}^2$ with periodic boundary conditions on all edges of $\partial\Omega$ and compute until $\tend=10$.
The exact solution \eqref{eq:LatticeExact} is a classical example of a generalised Beltrami flow as the convective term $\rb{\uu\ip\nabla}\uu$ balances the pressure gradient.
For the simulation, only the initial velocity is prescribed according to $\uu_0$ and the evolution of the flow is observed.
Varying orders $k\in\set{2,3,4,6}$ of the FE spaces with $\ku=k$ and $k_p=k-1$ are used and the resolution of the particular method is controlled via the mesh.
We use unstructured triangular meshes for this 2D example and the SIP penalty parameter $\sigma=\rb{k+1}\rb{k+2}$ is chosen, which has the correct quadratic $k$-dependency; see, for example, \cite[Section 3.1]{Hillewaert13}.
~\\

The time-stepping is based on the second-order multistep implicit-explicit (IMEX) scheme SBDF2 \cite{AscherEtAl95} where the Stokes part of the problem is discretised implicitly with a BDF2 method and the convection part relies on an explicit treatment with second-order accurate extrapolation in time.
The system matrix of the Stokes part is called $M^\ast$ and note that in such an IMEX scheme, only linear system associated with $M^\ast$ have to be solved in every time step.
However, this particular time-stepping scheme is only a choice here and not crucial for the subsequent results.
For the planar-lattice flow we use a constant time step of $\Delta t=\num{e-4}$.
~\\

We will compare the following two different quantities: total (velocity plus pressure) number of degrees of freedom (DOFs) and number of non-zero entries (NZEs) of $M^\ast$. 
While the DOFs indicate how rich the approximation space is, the NZEs are a more suitable measure of how efficient a particular discretisation is.
Indeed, the NZEs of $M^\ast$ indicate how expensive solving linear systems is; usually, this is the most time consuming part of a flow solver (especially in 3D).
In Figure \ref{fig:Lattice_errors}, the corresponding $\Lp{2}{\rb{0,\tend;\LP{2}{\OMEGA}}}$ errors can be seen for the two different methods introduced in Section \ref{sec:dGFEM}.
~\\

\begin{figure}[h]
\centering
	\includegraphics[width=0.35\textwidth]{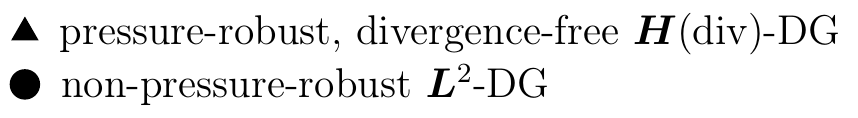} \\
	\includegraphics[width=0.48\textwidth]{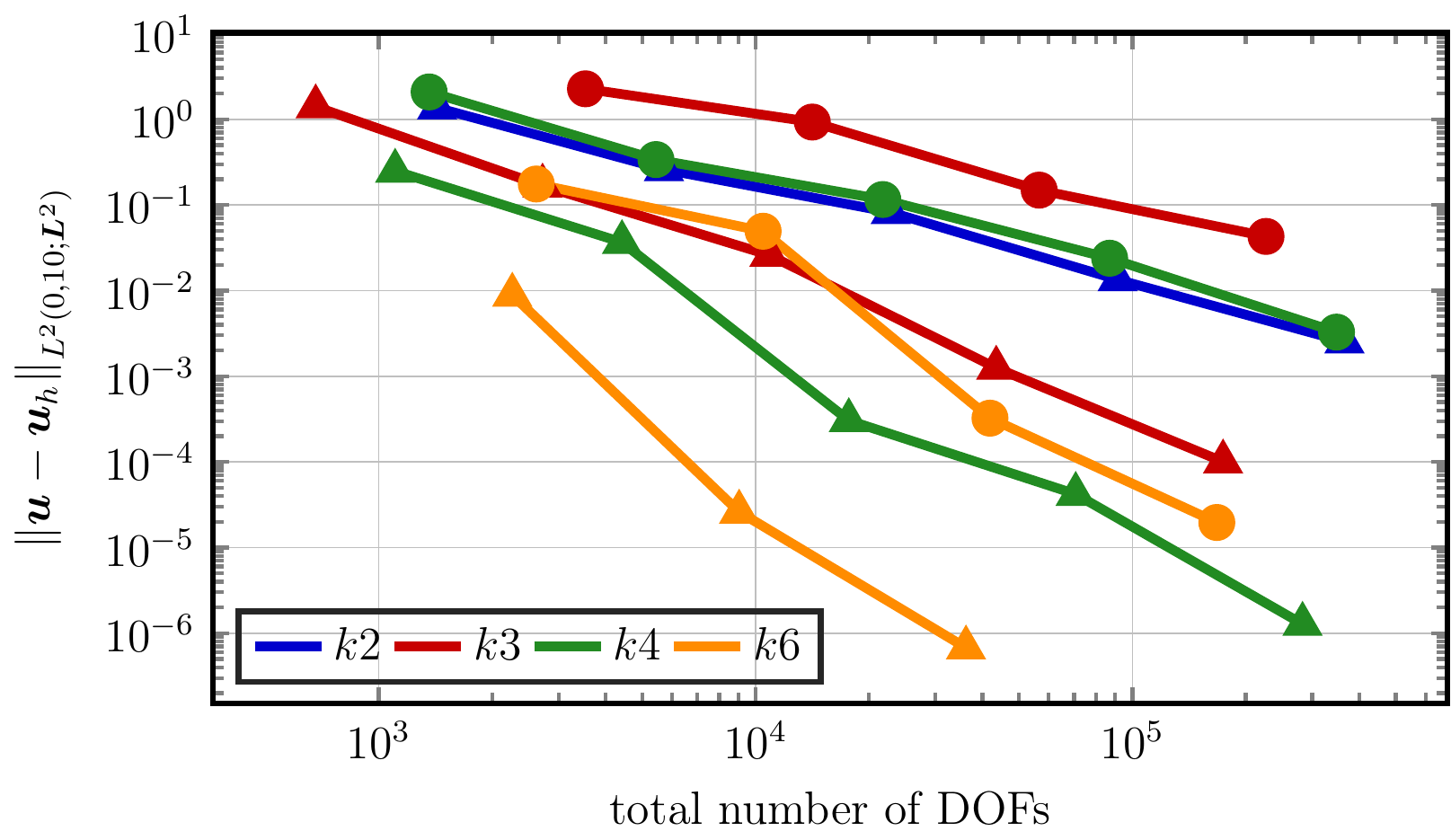} \hspace{5pt}
	\includegraphics[width=0.48\textwidth]{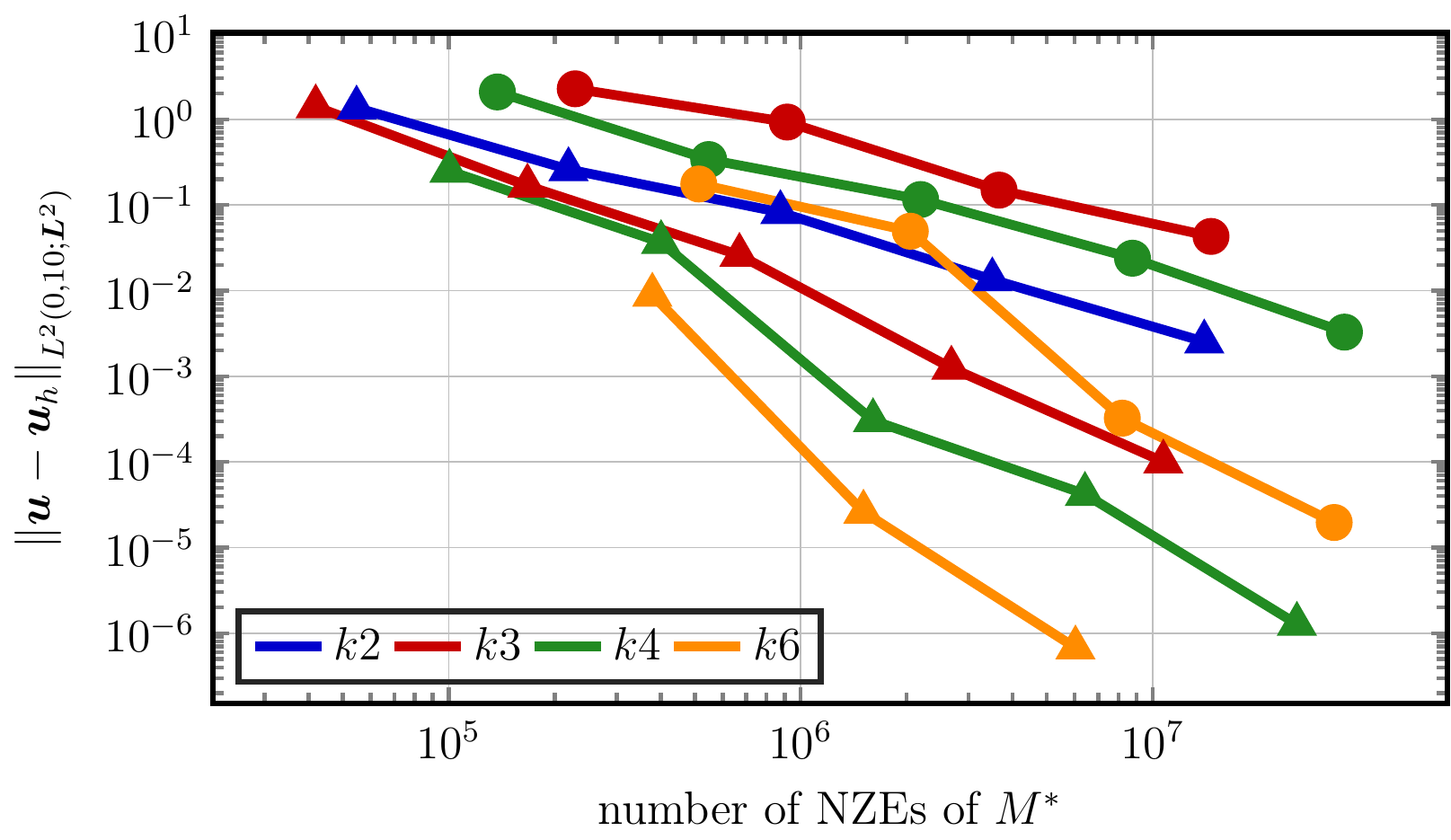} 
\caption{$\Lp{2}{\rb{0,\tend;\LP{2}{}}}$, $\tend=10$, errors for the 2D lattice flow ($\nu=\num{e-5}$). Comparison of pressure-robust $\HDIV$- and non-pressure-robust $\LP{2}{}$-DG methods with $\Delta t=\num{e-4}$, both using upwinding ($\theta=1$). The abscissae show the total number of DOFs (left) and number of NZEs of $M^\ast$ (right). }
\label{fig:Lattice_errors}
\end{figure}

Firstly, one can observe that whenever a fixed polynomial degree of the FE space is considered, the pressure-robust $\HDIV$-FEM always leads to an at least ten times smaller error, both in terms of DOFs and NZEs.
For higher-order and on finer meshes, this offset increases to such an extent that for $k=6$ the $\HDIV$-DG method's solution has an error which is at least \num{e-3} times smaller than the corresponding $\LP{2}{}$-DG's.
Even more remarkably, in terms of fixing DOFs, at least on coarse meshes the pressure-robust $k=2$ $\HDIV$-DG method results in a comparable accuracy as the $k=4$ $\LP{2}{}$-DG method while, at the same time, it leads to fewer NZEs. 
A similar observation holds for $k=3$ $\HDIV$-DG and $k=6$ $\LP{2}{}$-DG.
In practice, as higher-order methods usually lead to more NZEs, being able to use a method of order $k$ instead of $2k$, without loosing accuracy, means a considerable improvement with respect to performance.
~\\

In Section \ref{sec:GenBeltramiFlows}, it was argued that the velocity solution $\uu$ of a generalised Beltrami flow is simultaneously a solution of the Navier--Stokes and the Stokes problem (only the pressures are different).
We now want to consider, in a time-dependent setting, instead of the nonlinear Navier--Stokes problem with $\ff\equiv\zero$, the corresponding Stokes problem with $\ff=-\rb{\uu\ip\nabla}\uu$, solved with the  pressure-robust $\HDIV$-DG method for $k=6$.
~\\

The evolution of $\LP{2}{}$ errors can be seen in Figure \ref{fig:Lattice_Stokes_comparison}, where the left-hand side plot shows the behaviour of the Navier--Stokes solution with upwinding (solid lines) against the corresponding Stokes solution (dashed lines) for different refinement levels $r0,\dots,r3$.
The main observations are that the Stokes error is a lower bound for the Navier--Stokes error at all times and that on sufficiently fine meshes, both errors are not too far apart (at least for short times).
A second interesting issue is the question of the influence of upwinding on the solution.
The right-hand side subfigure Figure \ref{fig:Lattice_Stokes_comparison} shows that upwinding is indeed helpful for this kind of convection dominated problems.
More precisely, one can observe that the Navier--Stokes errors without upwinding (dashed lines) are always larger than with upwinding (solid lines).
~\\

\begin{figure}[h]
\centering
	\includegraphics[width=0.48\textwidth]{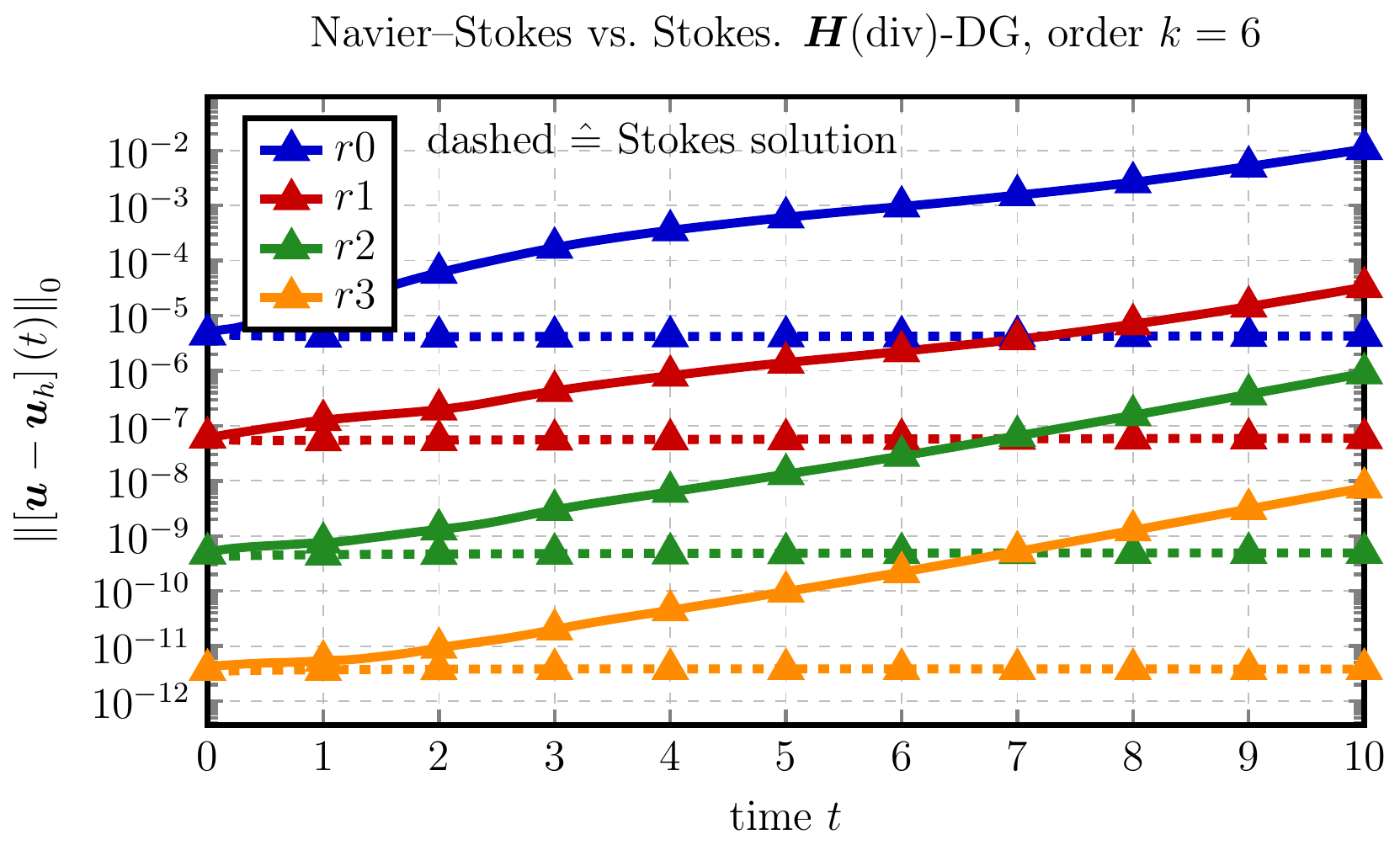} \hspace{5pt}
	\includegraphics[width=0.48\textwidth]{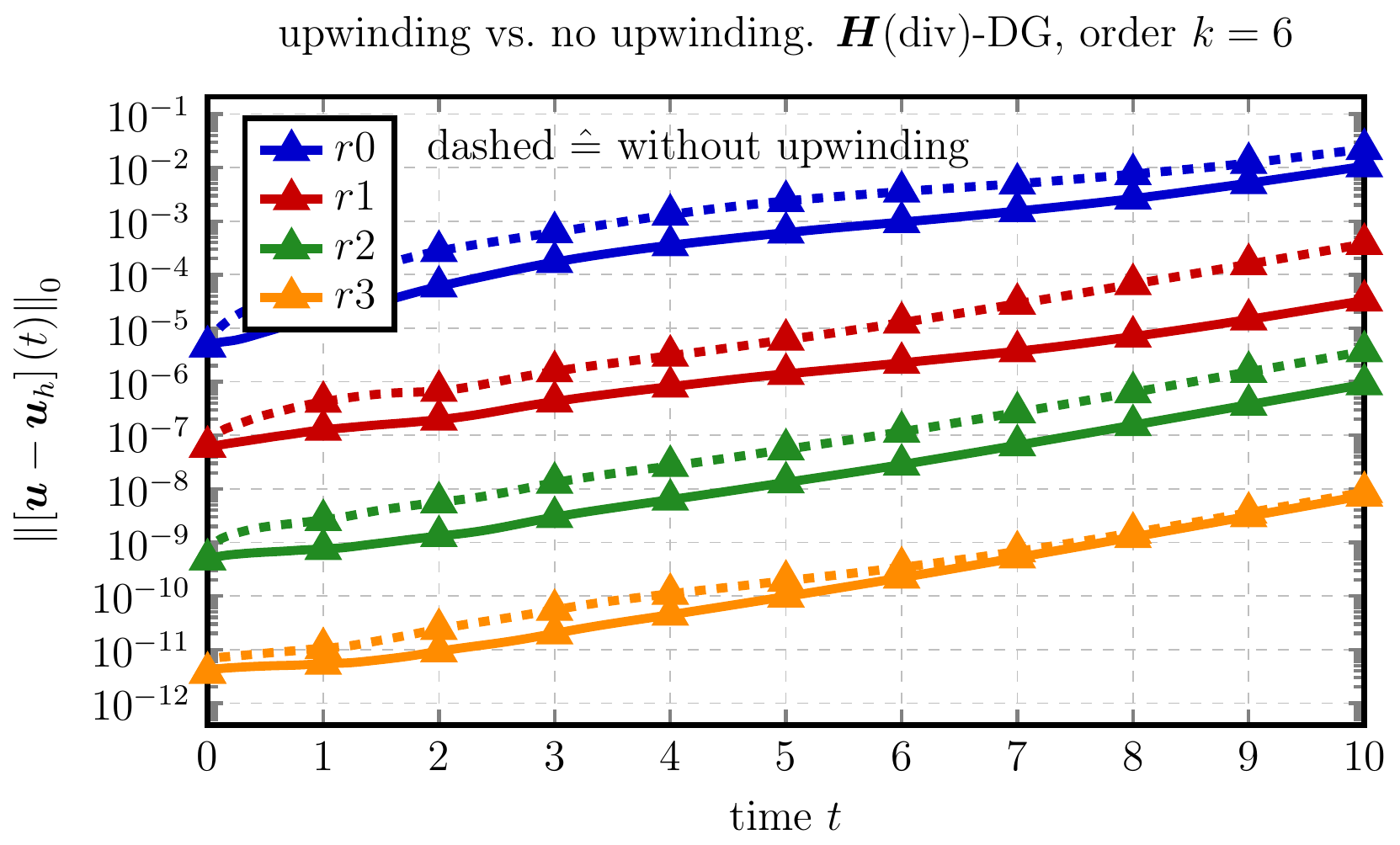} 
\caption{Evolution of $\LP{2}{}$ errors for the 2D lattice flow ($\nu=\num{e-5}$). Comparison of pressure-robust $\HDIV$ solution ($k=6$) with $\Delta t=\num{e-4}$ of the Navier--Stokes problem with upwinding (solid lines) and the Stokes problem (dashed lines) on the left-hand side. The colours indicate different levels of mesh refinement $r$t. Right-hand side: Navier--Stokes solution with upwinding $\theta=1$ (solid), and without upwinding $\theta=0$ (dashed).  }
\label{fig:Lattice_Stokes_comparison}
\end{figure}

Lastly, let us consider what happens if $\HM{1}{}$-conforming FEM are used instead of the (much more flexible) DG methods.
In order to do so, we repeat the experiment presented in Figure~\ref{fig:Lattice_errors} now for the pressure-robust, divergence-free Scott--Vogelius method SV$_k$ and the non-pressure-robust Taylor--Hood element TH$_k$; cf.\ also the corresponding explanations for the Gresho problem in Section~\ref{sec:GreshoH1}.
~\\

Figure~\ref{fig:Lattice_errors_H1} shows the results of this test scenario.
The left-hand side plot shows the $\Lp{2}{\rb{0,\tend;\LP{2}{}}}$ errors for $\tend=10$ for both SV and TH simulations for different polynomial orders $k\in\set{4,8,12}$.
One can observe that the Taylor--Hood method is not stable for this kind of flow problem as the solutions become unstable and long-time predictions seem to be impossible with a non-pressure-robust method.
The Scott--Vogelius method, on the other hand, shows the expected convergence behaviour also for longer time integration.
In order to also show that our results are consistent for short-time simulations, the right-hand side plot shows the $\Lp{2}{\rb{0,\tend;\LP{2}{}}}$ errors for $\tend=2$.
Here, one can see that the Taylor--Hood method is still stable and gives convergent results under mesh-refinement.
Interestingly, one can see that the TH$_8$ method yields comparable results as the SV$_4$ method for the same amount of DOFs.
This halving of approximation order has already been shown previously in the context of DG methods.
In this sense, we demonstrated that the question whether $\HM{1}{}$-conforming or DG methods are used is not essential but the important distinction has to be made with respect to whether a method is pressure-robust or not.
Consequently, we will not show any more numerical results with $\HM{1}{}$-FEM but only use DG methods henceforth.
~\\

\begin{figure}[h]
\centering
	\includegraphics[width=0.35\textwidth]{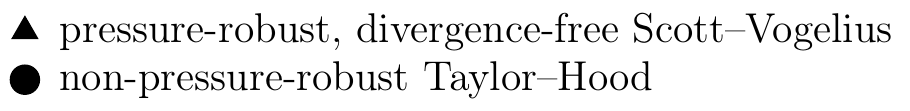} \\
	\includegraphics[width=0.48\textwidth]{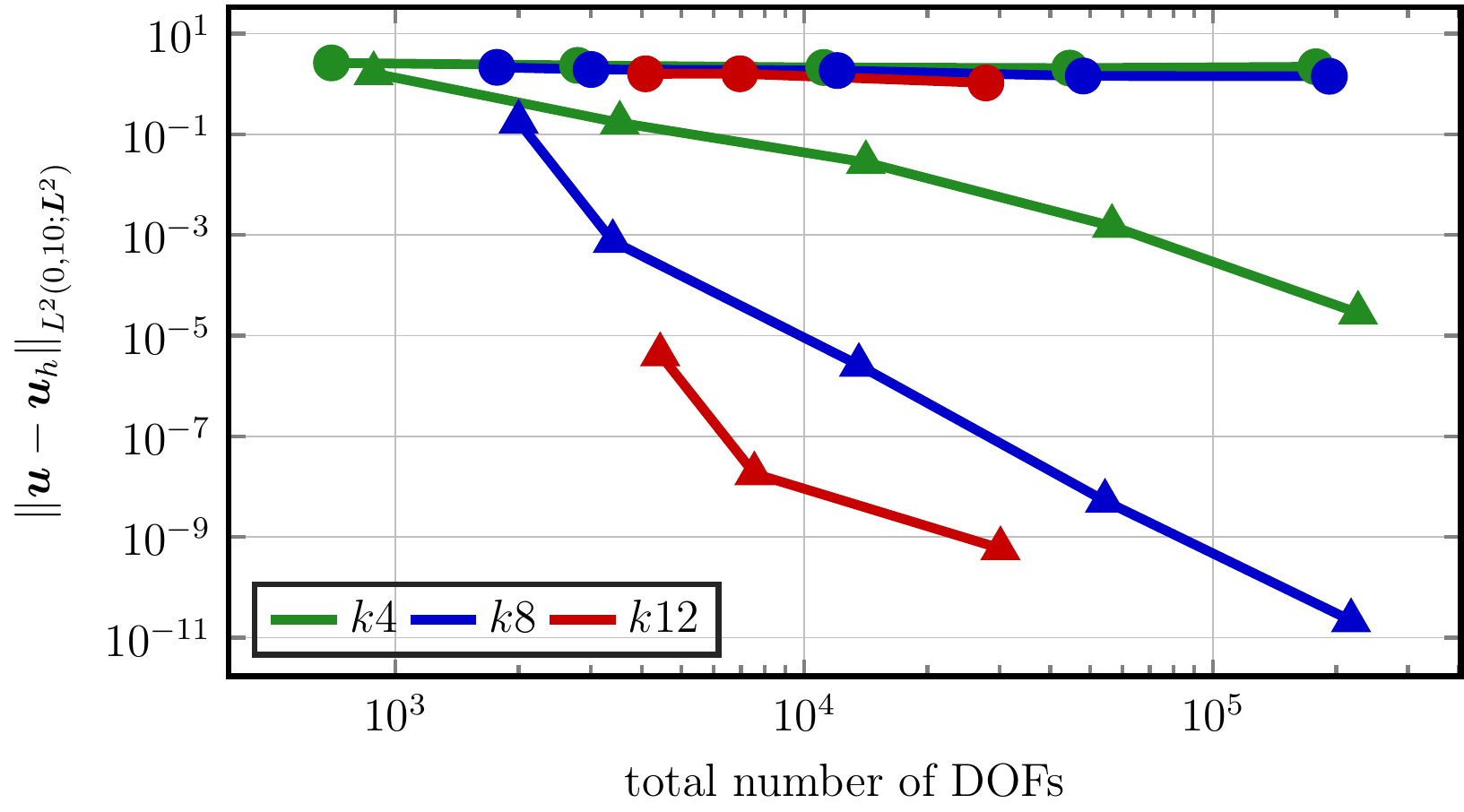} \hspace{5pt}
	\includegraphics[width=0.48\textwidth]{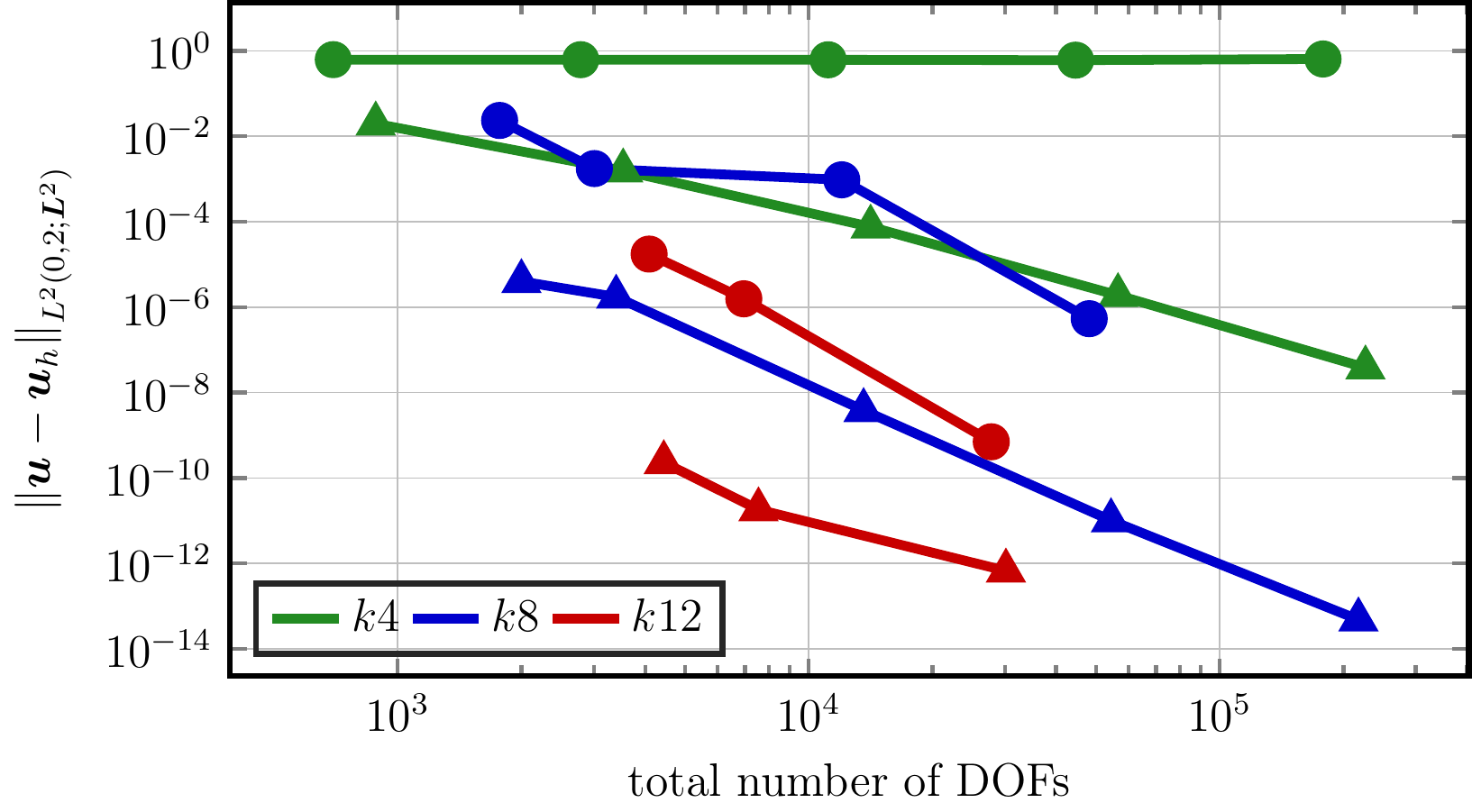} 
\caption{$\HM{1}{}$-FEM: $\Lp{2}{\rb{0,\tend;\LP{2}{}}}$, errors for the 2D lattice flow ($\nu=\num{e-5}$). Comparison of pressure-robust SV- and non-pressure-robust TH methods with $\Delta t=\num{e-4}$. Long-time simulation with $\tend=10$ (left) and short-time simulation with $\tend=2$ (right). }
\label{fig:Lattice_errors_H1}
\end{figure}

\subsection{Classical 3D Ethier--Steinman} 
\label{sec:3DEthierSteinman}

Now, we want to repeat our investigation from Section \ref{sec:2DLatticeFlow} with a three-dimensional flow problem ($\ff\equiv\zero$).
Thus, consider the exact velocity of the Ethier--Steinman problem \cite{EthierSteinman94}:
\begin{equation} \label{eq:3DES-exact}
	\uu_0\rb{\xx}= -a\begin{bmatrix}
		 e^{ax_1}\sin\rb{ax_2 + bx_3} + e^{ax_3}\cos\rb{ax_1 + bx_2}  \\
		 e^{ax_2}\sin\rb{ax_3 + bx_1} + e^{ax_1}\cos\rb{ax_2 + bx_3}  \\	
		 e^{ax_3}\sin\rb{ax_1 + bx_2} + e^{ax_2}\cos\rb{ax_3 + bx_1}  
	\end{bmatrix},\quad
	\uu\rb{t,\xx} = \uu_0\rb{\xx}e^{-\nu d^2 t }.
\end{equation}	

Here, the parameters $a=\nf{\pi}{4}$, $b=\nf{\pi}{2}$ and $\nu=0.002$ are used and we are dealing with a Beltrami flow.
As domain, the cube $\Omega=\rb{-1,1}^3$ is chosen where the time-dependent Dirichlet boundary condition $\gD$ according to the exact solution is imposed.
Again, the initial velocity $\uu_0$ is prescribed at $t=0$ and the evolution of the flow is observed.
We stop the simulations at $\tend=1$ and the time-stepping is again based on the SBDF2 method, again using $\Delta t=\num{e-4}$.
Unstructured tetrahedral meshes are used in the 3D example and the SIP penalty parameter $\sigma=8 k^2$ is chosen.
The resulting $\Lp{2}{\rb{0,\tend;\LP{2}{\OMEGA}}}$ errors can be seen in Figure \ref{fig:ES_errors}.
~\\

\begin{figure}[h]
\centering
	\includegraphics[width=0.35\textwidth]{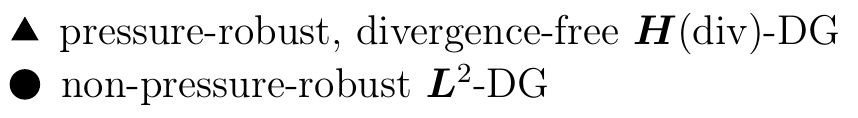} \\
	\includegraphics[width=0.48\textwidth]{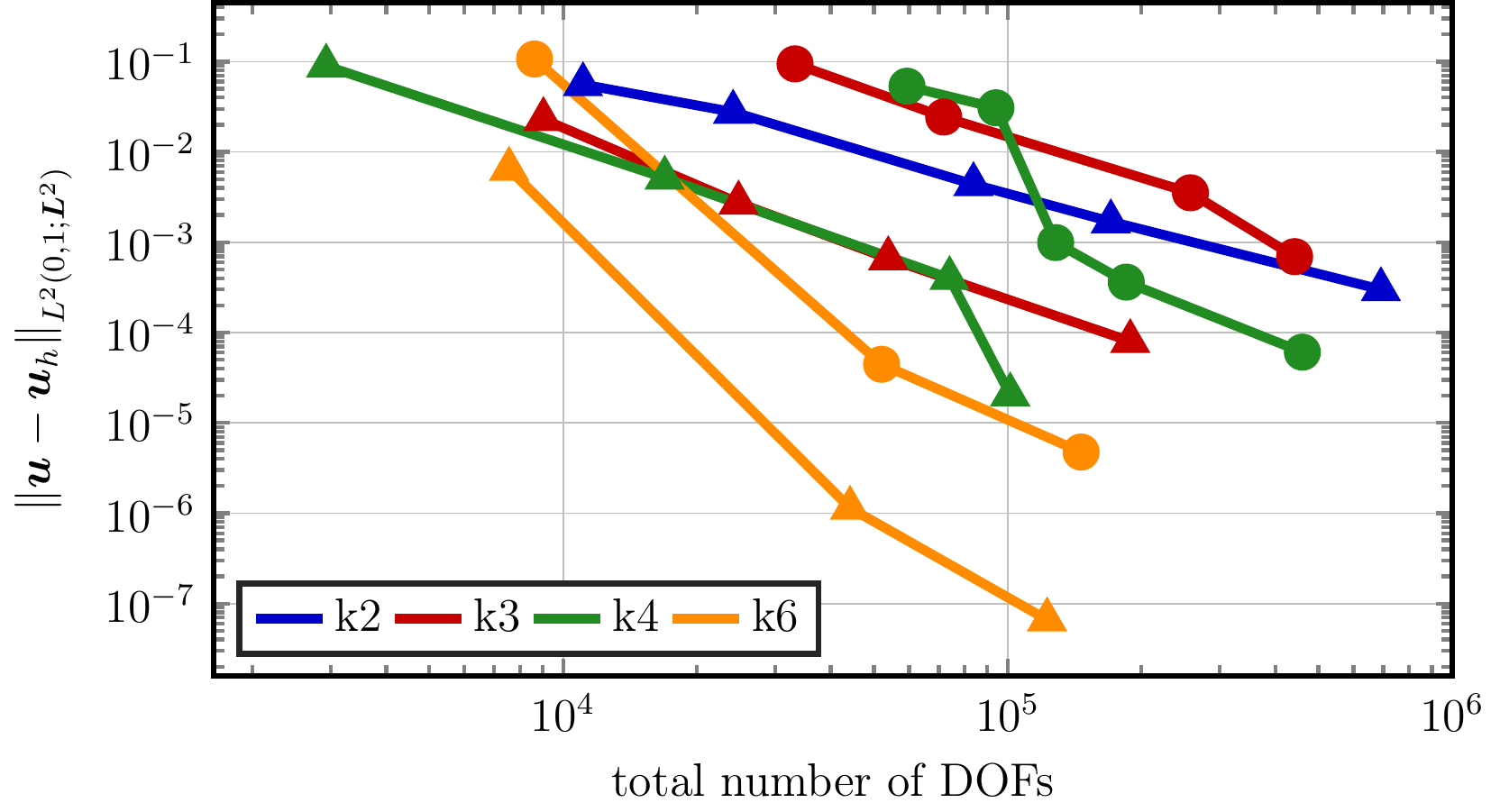} \hspace{5pt}
	\includegraphics[width=0.48\textwidth]{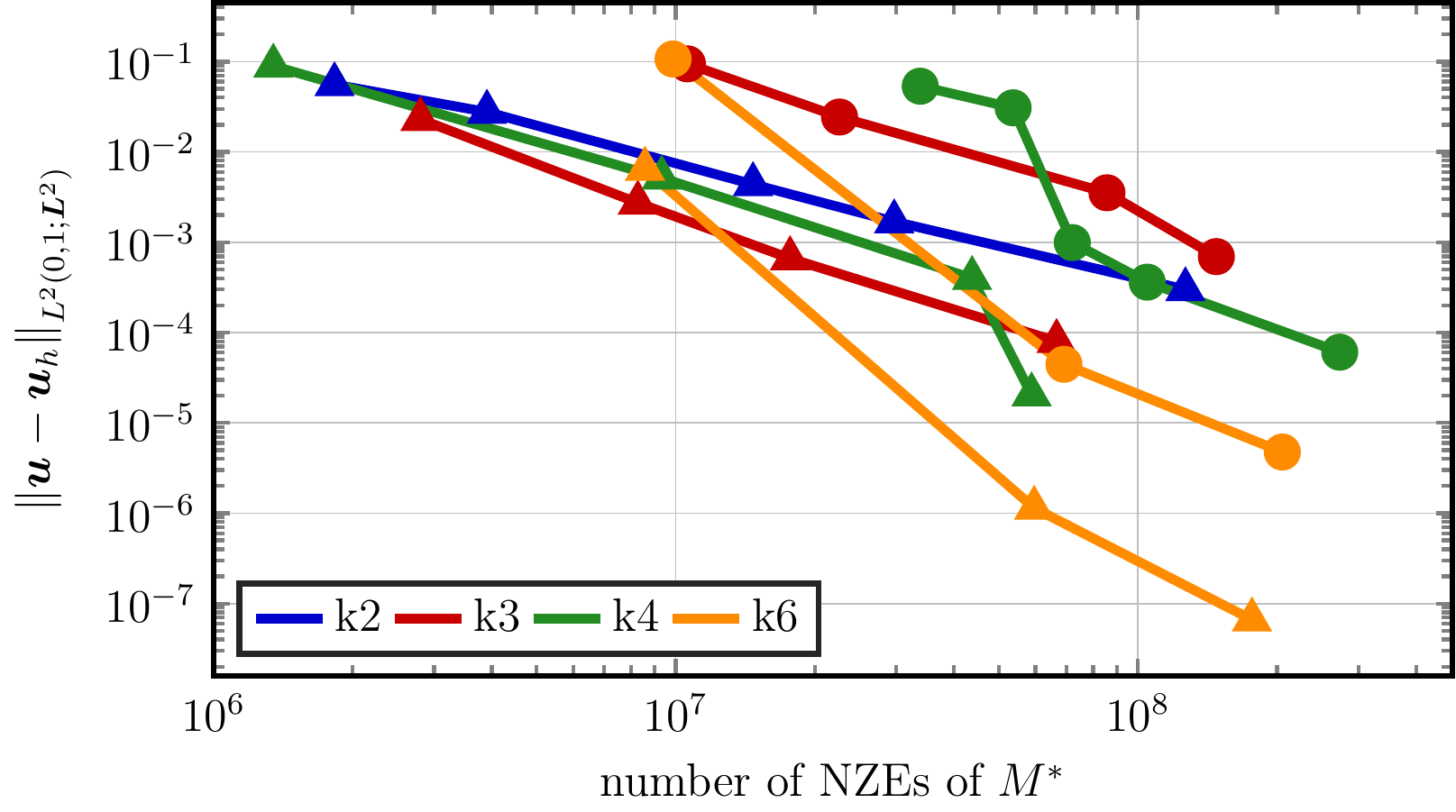}
\caption{Errors for the classical 3D Ethier--Steinman problem ($\nu=0.002$). Comparison of pressure-robust $\HDIV$- and non-pressure-robust $\LP{2}{}$-DG methods with $\Delta t=\num{e-4}$, both using upwinding ($\theta=1$). The abscissae show the total number of DOFs (left) and number of NZEs of $M^\ast$ (right). }
\label{fig:ES_errors}
\end{figure}

Again, for each fixed $k$, the divergence-free and pressure-robust $\HDIV$-DG method always yields the smallest errors.
In terms of efficiency, the plot of the number of NZEs again shows that by using a pressure-robust method, a significant amount of computational effort can be saved.
However, for sufficiently fine meshes, higher-order will always be superior because of the smoothness of the problem.
The next section shall pick up at exactly this point.

\subsection{3D Ethier--Steinman with inaccurate Dirichlet BCs} 
\label{sec:3DEthierSteinmanBC}

In applications, it is very rare that one has an exact geometry description of the underlying domain as, for example, curved boundaries make it necessary to approximate  also the geometry to a certain accuracy.
Therefore, the imposition of the \emph{correct} Dirichlet BCs is usually made on the \emph{approximated} boundary.
Unavoidably, this a source for errors and in this section we want to mimic such a situation in the following equivalent way.
Instead of prescribing the correct Dirichlet BCs on the approximated boundary, we will impose inaccurate BCs on the correct boundary.
~\\

\begin{figure}[h]
\centering
	\includegraphics[width=0.39\textwidth]{Ethier-Steinman/tikz/ES-legend.pdf} \\
	\includegraphics[width=0.48\textwidth]{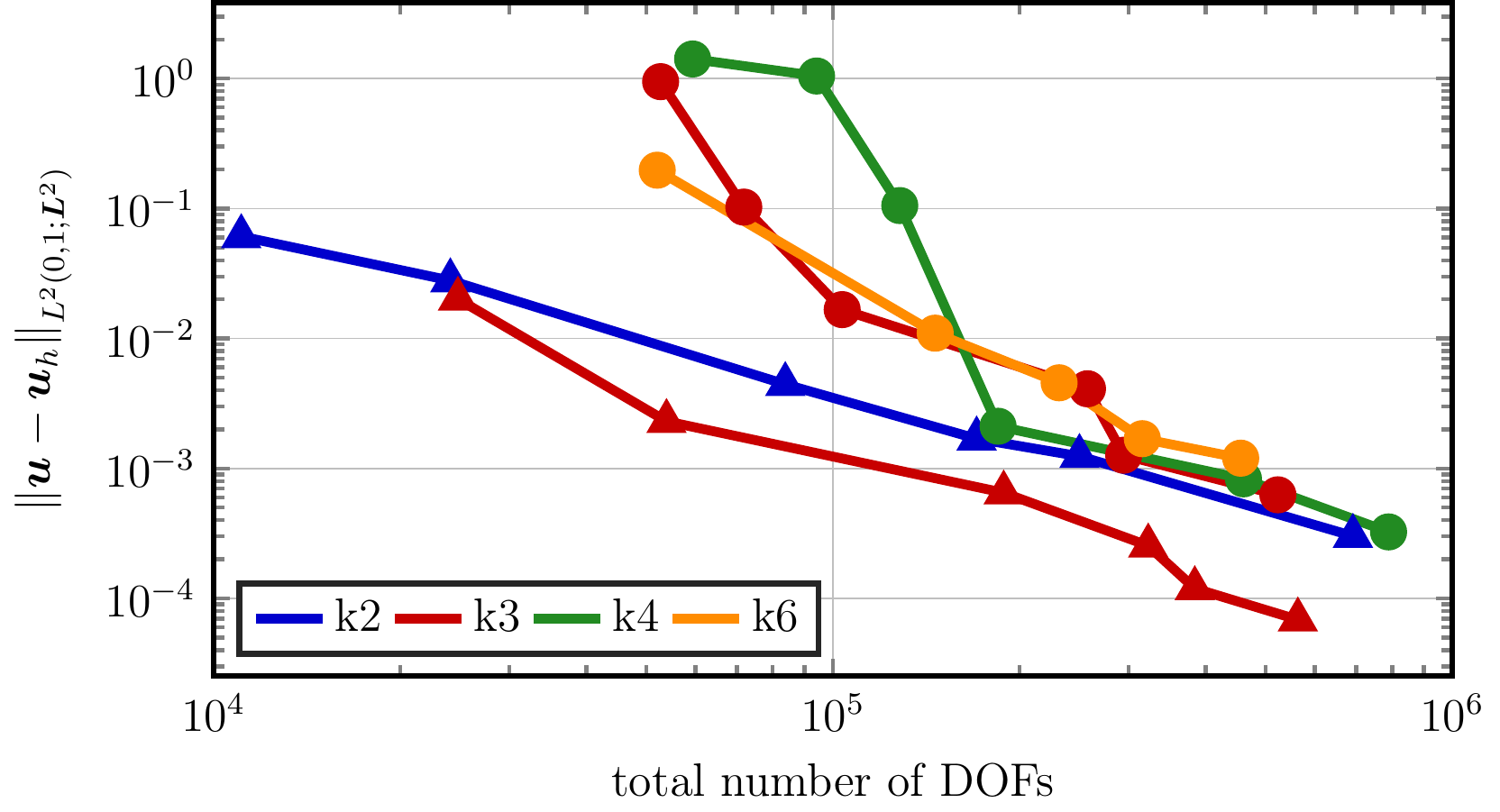} \hspace{5pt}
	\includegraphics[width=0.48\textwidth]{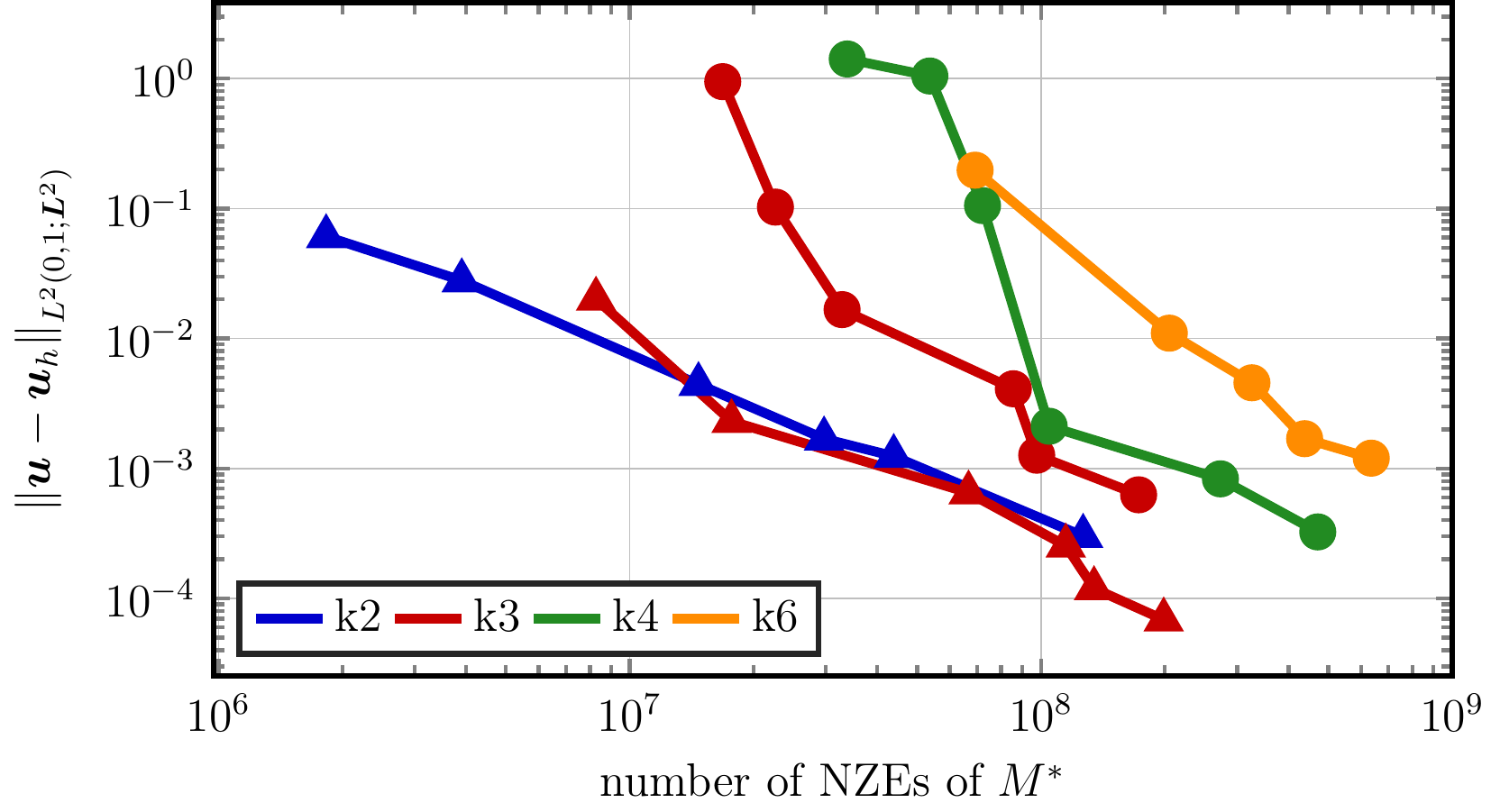} 
\caption{Errors for the 3D Ethier--Steinman problem ($\nu=0.002$) with inaccurate BCs. Comparison of pressure-robust $\HDIV$- and non-pressure-robust $\LP{2}{}$-DG methods with $\Delta t=\num{e-4}$, both using upwinding ($\theta=1$). The abscissae show the total number of DOFs (left) and number of NZEs of $M^\ast$ (right). }
\label{fig:ES_errors_BC}
\end{figure}

Let us revisit the 3D Ethier--Steinman example from Section \ref{sec:3DEthierSteinman}.
The exact solution is also given by \eqref{eq:3DES-exact}, exactly the same parameters are used and again, time-dependent Dirichlet boundary conditions are imposed.
However, as hinted at above, instead of choosing $\gD$ according to the exact solution $\uu$, we will use a piecewise quadratic approximation of it instead.
Figure \ref{fig:ES_errors_BC} shows the resulting $\Lp{2}{\rb{0,\tend;\LP{2}{\OMEGA}}}$ errors.
~\\

First of all, note that the pressure-robust method is again always more accurate for a fixed number of DOFs and a fixed number of NZEs.
In contrast to Section \ref{sec:3DEthierSteinman}, the asymptotical behaviour for higher $k$ and finer meshes is now, by construction, dominated by the accuracy of $\gD$.
Thus, one can see that all methods roughly lead to the same result when the resolution is high enough.
More interestingly though, on coarse meshes, the pressure-robust method is always significantly more accurate.

\section{Material derivative of a K\'arm\'an vortex street}
\label{sec:KarmanBeltrami}

Above it was shown that for generalised Beltrami flows, and flows which are Galilean-invariant to a generalised Beltrami flow, pressure-robust mixed methods can be much more accurate than classical mixed methods. 
As a main reason, strong and complicated gradient fields in the Navier--Stokes momentum balance have been identified, which explain the superior accuracy of pressure-robust methods.
However, a natural question is whether pressure-robust mixed methods are also superior for `real world flows'. 
Therefore, it is interesting to investigate whether strong gradient fields will appear also in such a situation. 
Indeed, Subsection \ref{sec:EulerFlows} suggests that the material derivative of transient high Reynolds number flows is approximately a non-trivial gradient field. 
In the following, we will confirm this conjecture for a periodic K\'arm\'an vortex street.
~\\

As an example of a practically relevant flow, let us consider the flow around an obstacle in a 2D channel of dimensions $\rb{L,H}=(3,1.01)$.
Here, the obstacle is chosen as a circle with radius $r=0.1$ whose centre is placed at $\xx=\rb{0.4,0.5}^\dag$.
In Figure~\ref{fig:Karman_VelMag}, such a flow can be seen at a time instance where the characteristic vortex shedding of a periodic K\'arm\'an vortex street has formed.
Here, all flow computations are performed with the divergence-free and pressure-robust $\HDIV$-DG with order $k$.
The Dirichlet inflow BC on the left part of the boundary is given by the constant free flow $\uu\rb{t,0,x_2}=\rb{1,0}^\dag$, which, together with $\nu=\num{e-3}$ leads to a Reynolds number $\Rey=\overline{u}r/\nu=1\times0.1\times 1000=100$.
No-slip is prescribed on the boundary of the circle, whereas top, bottom and right part of the boundary represent outflow boundaries (do-nothing).
Computations are performed up to $\tend=10$ and all subsequent snapshots and norms are taken at exactly this instance in time.
\\

\begin{figure}[h]
\centering
\includegraphics[width=0.475\textwidth]{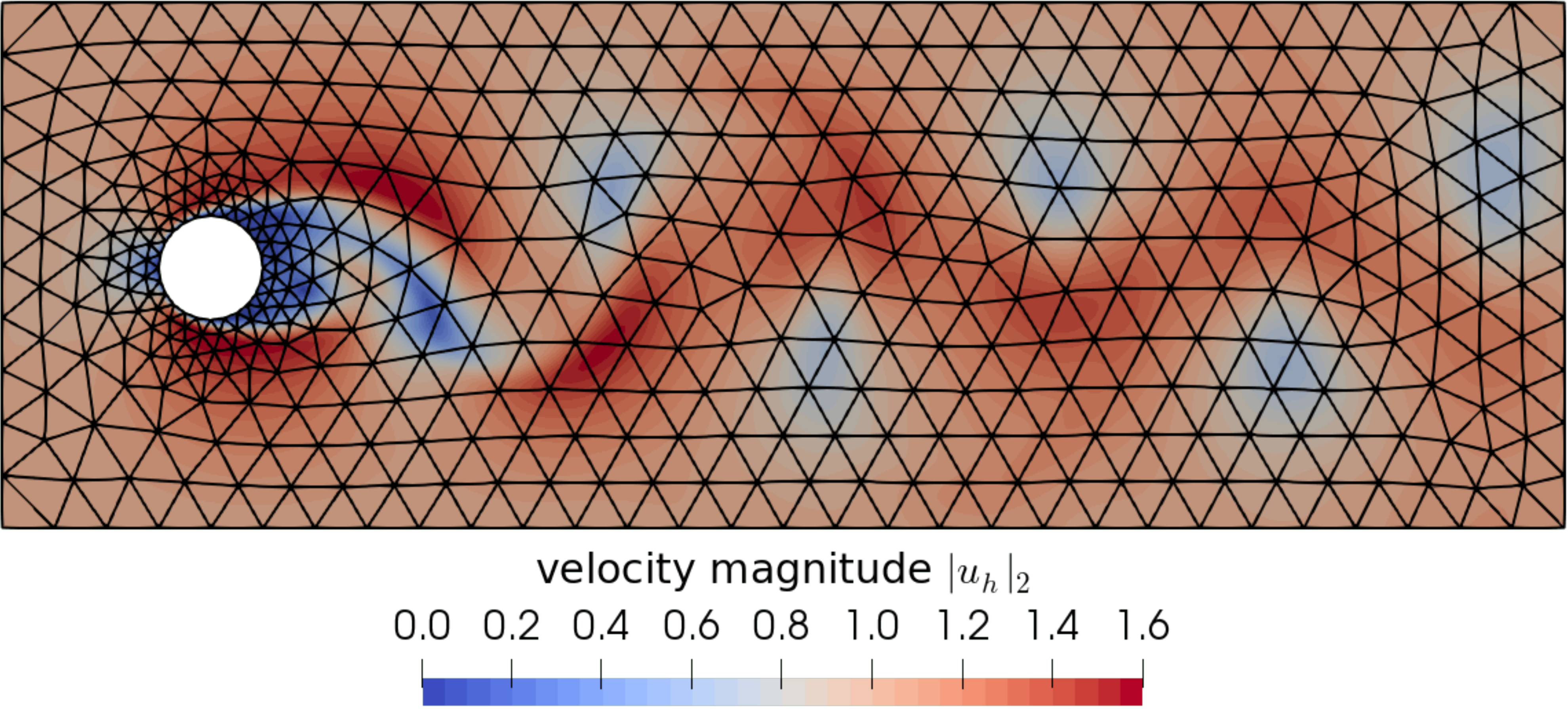} \hspace{5pt}
\includegraphics[width=0.475\textwidth]{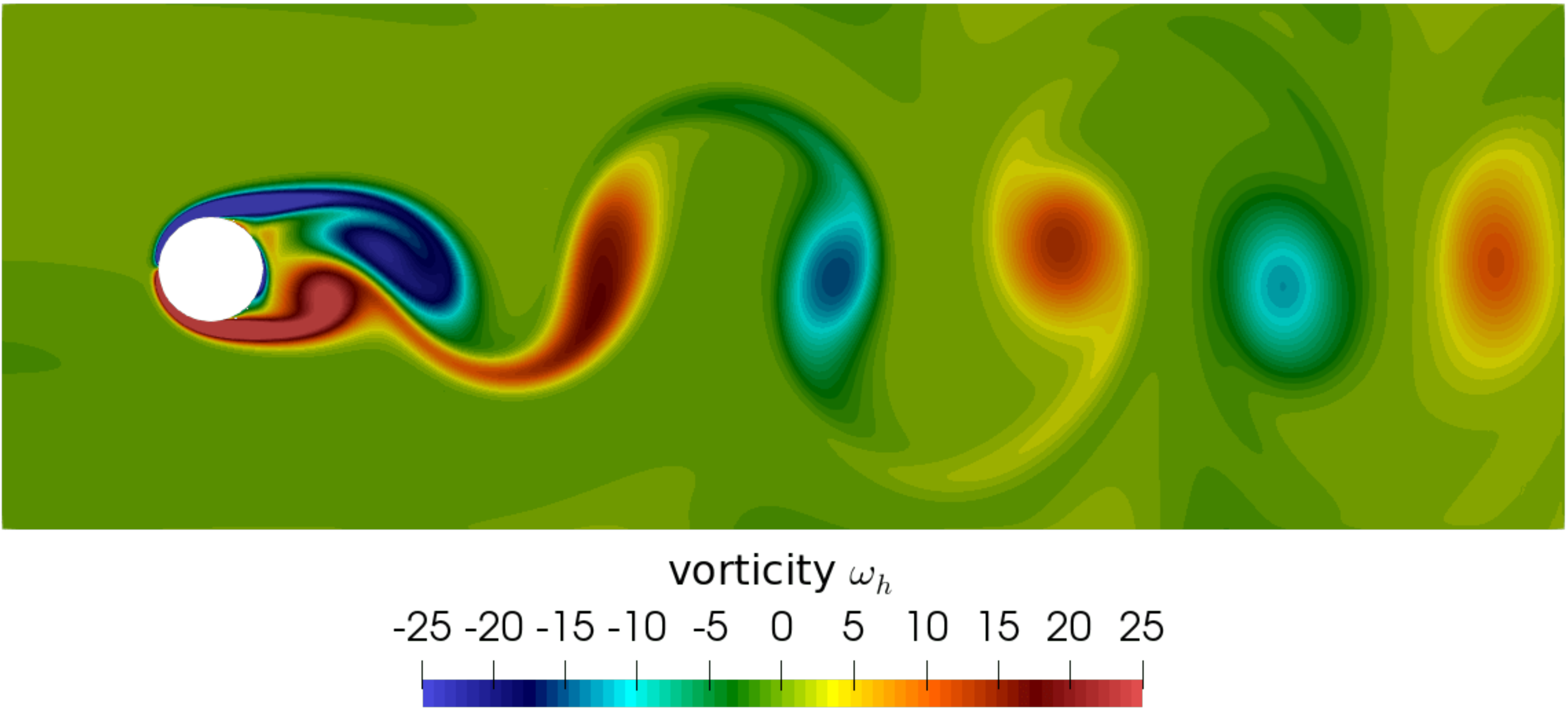}
\caption{Periodic K\'arm\'an vortex shedding in the wake of a square. Visualisation of the underlying geometry, velocity magnitude $\abs{\uu_h}_2$ and computational mesh (left), and vorticity $\omega_h$ (right).  The computations are done with eight-order elements ($k=8$) and upwinding ($\theta=1$), unless stated otherwise.}
\label{fig:Karman_VelMag}
\end{figure}

Being in this situation and having such a flow at hand, two main questions arise:
\begin{enumerate}
	\item In which part of the domain is the Helmholtz--Hodge projector of the material derivative small, as predicted by Subsection \ref{sec:EulerFlows}?
	\item Where does a pressure-robust method locally outperform a non-pressure robust one?
\end{enumerate}

These questions will be answered in the following two subsections.

\subsection{Investigation of the material derivative}

The most obvious approach for answering the first question is to begin with inspecting the (discrete) material derivative $\ff_h^t=\partial_t\uu_h+\rb{\uu_h\ip\nabla_h}\uu_h$ in a suitable norm.
Note that this investigation is based on the fact that with $k=8$ and this Reynolds numbers, all essential flows features are resolved.
In this way, statements about $\uu_h$ should also holds for $\uu$. 
We have chosen to investigate it in the $\LP{3/2}{}$ norm, since even for 3D flows at least the nonlinear convection term $\rb{\uu_h\ip\nabla_h}\uu_h$ is (for almost all times in the sense of Bochner spaces) in $\LP{3/2}{}$.
Qualitatively, however, the same insights can be obtained by, for example, considering the $\LP{2}{}$ norm.
~\\

\begin{figure}[h]
\centering
\includegraphics[width=0.55\textwidth]{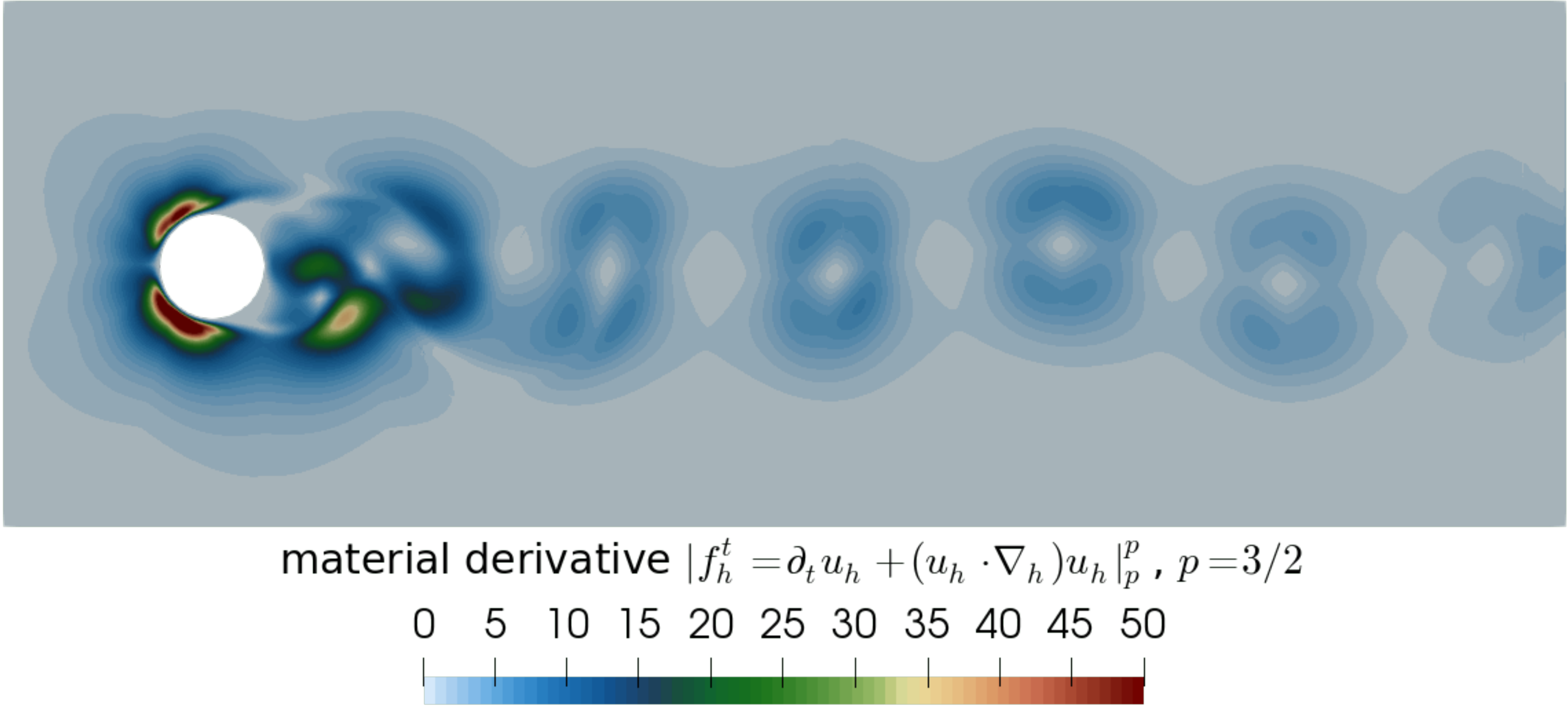}
\caption{Discrete material derivative $\abs{\ff_h^t=\partial_t\uu_h+\rb{\uu_h\ip\nabla_h}\uu_h}_\nf{3}{2}^\nf{3}{2}$. Note that the colour bar is chosen in such a way that all values above 50 are displayed red. }
\label{fig:Karman_ffh}
\end{figure}

Figure~\ref{fig:Karman_ffh} shows such a qualitative approach, where we observe that in a large part of the domain the material derivative (approximately) vanishes.
There, the material derivative is indeed a (trivial) gradient field. 
However, pressure-robust methods will only be superior to non-pressure-robust methods in parts of the domain, where the material derivative is a non-trivial and strong gradient field.
However, there are also some regions in the flow where the material derivative itself is large (basically the direct vicinity of the obstacle and parts of the wake).
~\\

\begin{figure}[h]
\centering
\includegraphics[width=0.55\textwidth]{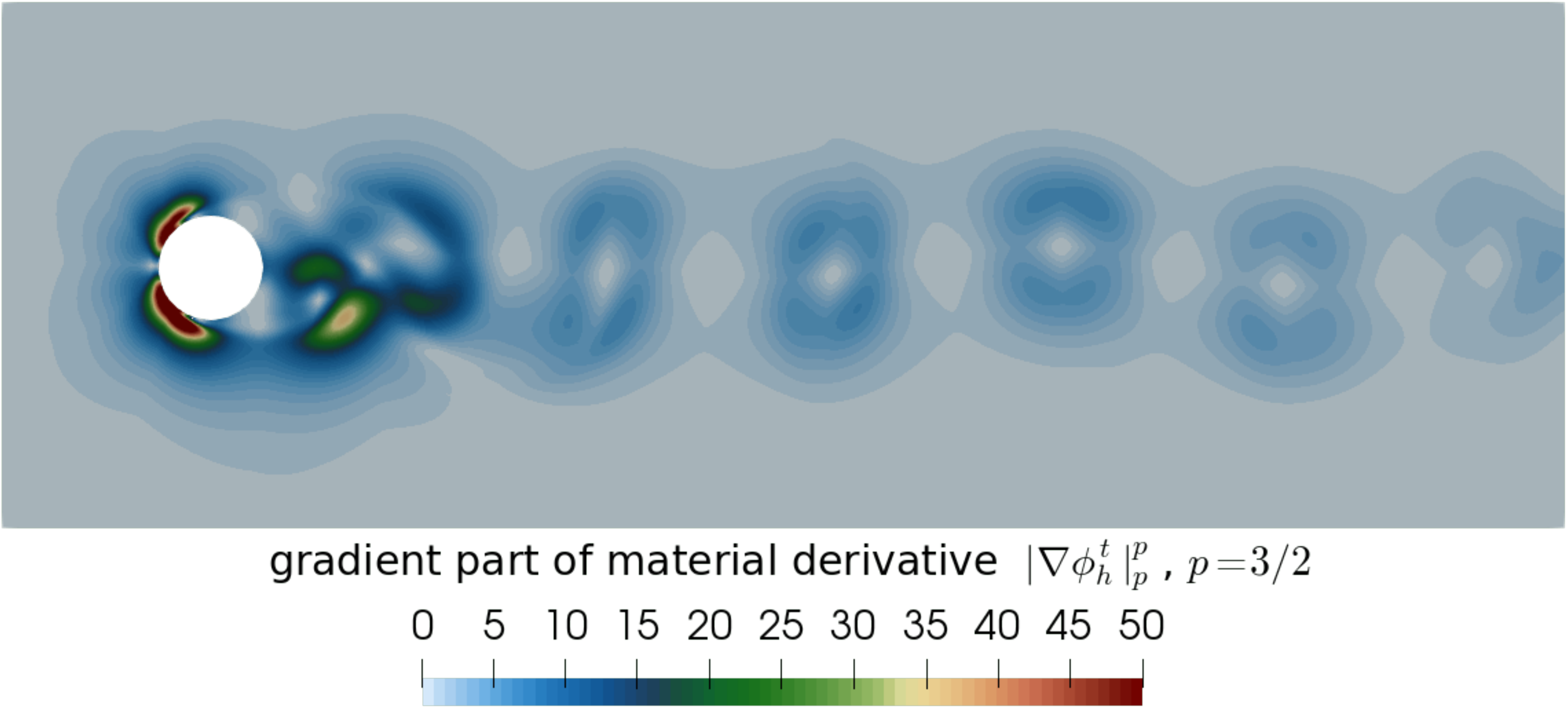}
\caption{Gradient part $\abs{\nabla\phi_h^t}_\nf{3}{2}^\nf{3}{2}$ of discrete Helmholtz--Hodge decomposition of $\ff_h^t$ with exactly divergence-free FEM. The colour bar scaling is chosen identically to that of Figure \ref{fig:Karman_ffh}.}
\label{fig:Karman_GradPhih}
\end{figure}

Thus, we will investigate, where $\ff_h^t=\partial_t\uu_h+\rb{\uu_h\ip\nabla_h}\uu_h$ is locally a gradient.
The gradient contribution of the discrete material derivative $\ff_h^t$ can be seen in Figure~\ref{fig:Karman_GradPhih} where, for a better comparison, the colour bar scaling is chosen identically to that of Figure~\ref{fig:Karman_ffh}.
Note that the Helmholtz decomposition $\ff_h^t=\HLhd{\ff_h^t}+\nabla\phi_h^t$ is also based on the divergence-free $\HDIV$-DG method.
One can observe that especially in the direct vicinity of the obstacle, there is indeed a significant gradient contribution in the material derivative $\ff_h^t$.
~\\

\begin{figure}[h]
\centering
\includegraphics[width=0.55\textwidth]{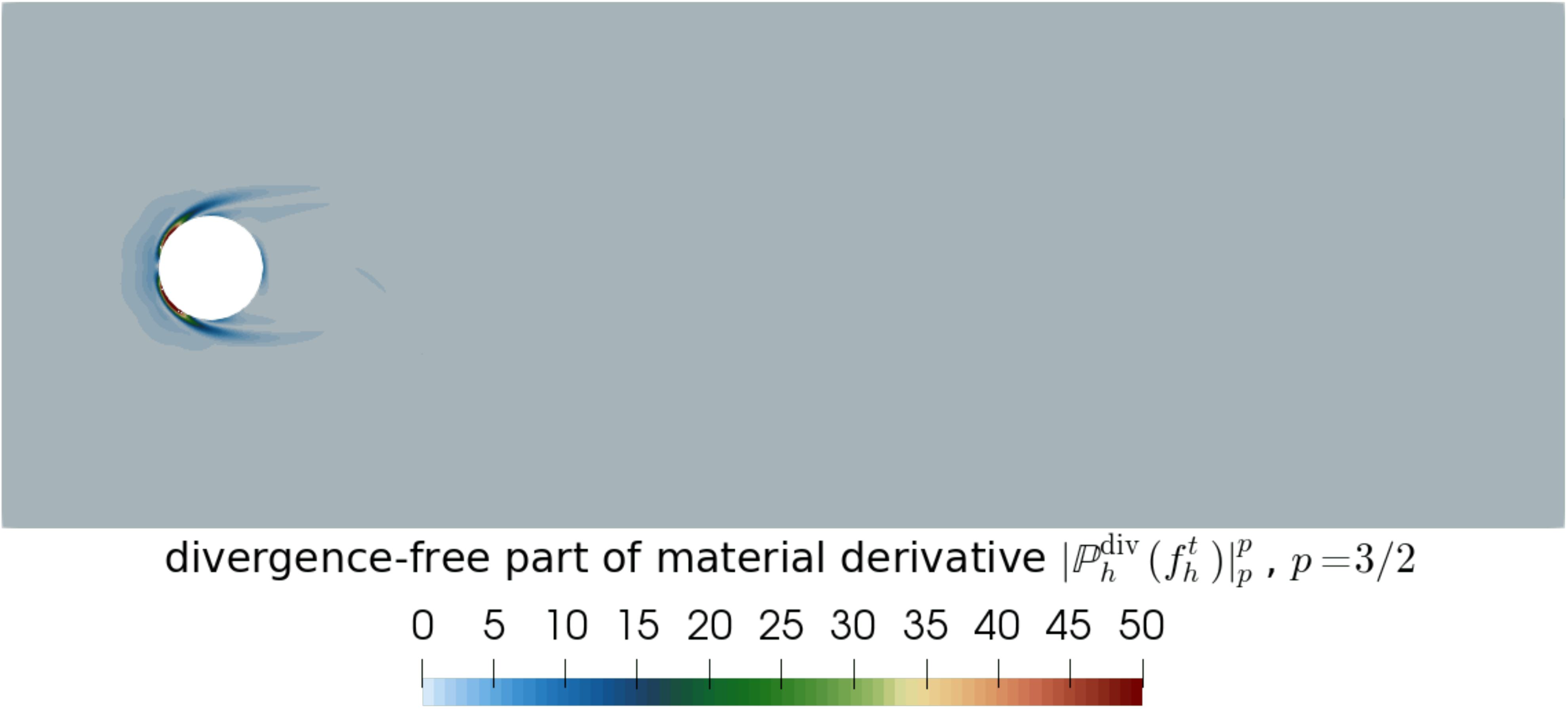}
\caption{Divergence-free part $\abs{\HLhd{\ff_h^t}}_\nf{3}{2}^\nf{3}{2}$ of discrete Helmholtz--Hodge decomposition of $\ff_h^t$ with exactly divergence-free FEM. The colour bar scaling is chosen identically to that of Figure \ref{fig:Karman_ffh}.}
\label{fig:Karman_Phdiv}
\end{figure}

We have seen that a large part of the material derivative consists of a gradient part.
Consequently, the divergence-free part of $\ff_h^t$ must be small.
Figure~\ref{fig:Karman_Phdiv} shows that this is actually the case for this flow problem.
One can observe that only the upstream side of the obstacle shows a non-zero divergence-free part $\HLhd{\ff_h^t}$.
Note that the friction term $-\nu \Delta \uu$ is indeed a divergence-free contribution to the Navier--Stokes material derivative, which can be strong in the vicinity of the no-slip boundary of the domain.

\subsection{Pressure-robust vs.\ non-pressure-robust methods} 

In contrast to considering the flow itself as in the previous subsection, let us now finally focus on the numerical method for approximating the flow.
In order to answer the second question, namely where a pressure-robust method outperforms a non-pressure-robust one, we solve a second discrete Helmholtz--Hodge decomposition problem of the form \eqref{eq:DiscreteHelmholtzLeray}, but this time with the $\LP{2}{}$-DG `velocity' space choice \eqref{eq:VelSpaceL2}.
Note again that due to using the $\LP{2}{}$-DG method, $\HLhdc{\ff_h^t}$ is not exactly divergence-free, even though $\ff_h^t$ has been computed with the divergence-free $\HDIV$-DG method.
~\\

\begin{figure}[h]
\centering
\includegraphics[width=0.55\textwidth]{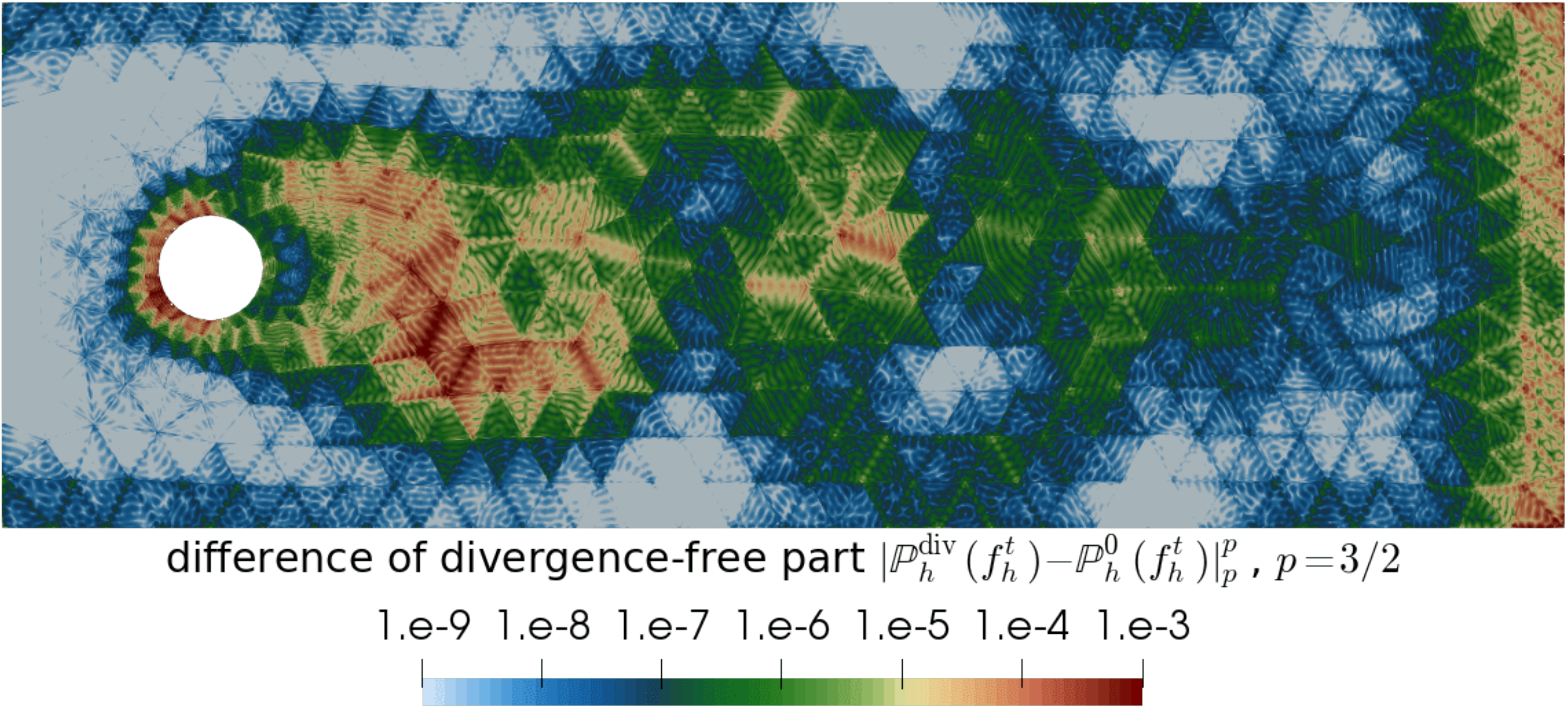}
\caption{Difference $\abs{\HLhd{\ff_h^t}-\HLhdc{\ff_h^t}}_\nf{3}{2}^\nf{3}{2}$ of the two discrete Helmholtz--Hodge projectors of $\ff_h^t$. High values indicate advantageous regions of the pressure-robust discretisation. Note that the colour scale is chosen logarithmically.}
\label{Karman_AdvantagePRobust}
\end{figure}

Now, in Figure~\ref{Karman_AdvantagePRobust}, the difference of the two discrete Helmholtz--Hodge projectors $\HLhd{\ff_h^t}-\HLhdc{\ff_h^t}$ of the discrete material derivative $\ff_h^t$ highlights the regions in the flow where a pressure-robust method performs better than a non-pressure-robust one.
This is due to the fact that the $\LP{2}{}$ Helmholtz--Hodge projector indicates that the corresponding non-divergence-free method would see a wrong force locally.
Consistent with the analysis of the flow characteristics above, the pressure-robust method is advantageous in regions where the material derivative is dominated by a large gradient contribution in the sense of the Helmholtz decomposition.
This is especially the case in the vicinity of the obstacle, the wake, and interestingly, the outflow.
However, note that the difference between the two Helmholtz--Hodge projections seems to be small compared to the total value of the components of the Helmholtz--Hodge decomposition.
The reason for this behaviour is that on the given mesh and with $k=8$, even the non-pressure-robust Helmholtz projector is comparably accurate.
~\\

\begin{table}[h]
\caption{Convergence behaviour for $\LP{\nf{3}{2}}{}$-norms of the (discrete) material derivative $\ff_h^t=\partial_t\uu_h+\rb{\uu_h\ip\nabla_h}\uu_h$ and its discrete Helmholtz--Hodge projections for different polynomial orders $k\in\set{2,\dots,8}$.}
\label{tab:Norms_kConvergence}
\centering 
\begin{tabular}{lllllllll} 
\toprule
	$k$	
		& 2 
		& 3
		& 4
		& 5
		& 6
		& 7
		& 8
		\\ 
\otoprule
	$\norm{\ff_h^t}_\LP{\nf{3}{2}}{}$		
		& \pgfmathprintnumber[precision=5]{2.87349993e+00}
		& \pgfmathprintnumber[precision=5]{2.97578705e+00}
		& \pgfmathprintnumber[precision=5]{2.96509174e+00}
		& \pgfmathprintnumber[precision=5]{2.96072811e+00}
		& \pgfmathprintnumber[precision=5]{2.96032325e+00}	
		& \pgfmathprintnumber[precision=5]{2.96001111e+00}
		& \pgfmathprintnumber[precision=5]{2.95996950e+00}
		\\
	$\norm{\HLhd{\ff_h^t}}_\LP{\nf{3}{2}}{}$		
		& \pgfmathprintnumber[precision=5]{7.96556428e-01}
		& \pgfmathprintnumber[precision=5]{5.48187542e-01}
		& \pgfmathprintnumber[precision=5]{4.38494468e-01}
		& \pgfmathprintnumber[precision=5]{4.18748340e-01}
		& \pgfmathprintnumber[precision=5]{4.13535207e-01}
		& \pgfmathprintnumber[precision=5]{4.12409309e-01}
		& \pgfmathprintnumber[precision=5]{4.12135449e-01}
		\\		
	$\norm{\HLhdc{\ff_h^t}}_\LP{\nf{3}{2}}{}$		
		& \pgfmathprintnumber[precision=5]{1.06804765e+00}
		& \pgfmathprintnumber[precision=5]{5.75710735e-01}
		& \pgfmathprintnumber[precision=5]{4.41380799e-01}
		& \pgfmathprintnumber[precision=5]{4.19200115e-01}
		& \pgfmathprintnumber[precision=5]{4.13624756e-01}
		& \pgfmathprintnumber[precision=5]{4.12418656e-01}
		& \pgfmathprintnumber[precision=5]{4.12139102e-01}
		\\
	$\norm{\sqb{\helm_h^\dvg-\helm_h^0}\rb{\ff_h^t}}_\LP{\nf{3}{2}}{}$		
		& \pgfmathprintnumber[precision=0,sci]{6.78923884e-01}
		& \pgfmathprintnumber[precision=0,sci]{1.58786854e-01}
		& \pgfmathprintnumber[precision=0,sci]{4.10886877e-02}
		& \pgfmathprintnumber[precision=0,sci]{1.63212054e-02}
		& \pgfmathprintnumber[precision=0,sci]{5.92165742e-03}
		& \pgfmathprintnumber[precision=0,sci]{2.10554255e-03}
		& \pgfmathprintnumber[precision=0,sci]{8.31987666e-04}	
		\\		
\bottomrule		
\end{tabular}
\end{table}

Thus, finally, let us demonstrate that by using lower-order methods, the difference between the discrete Helmholtz--Hodge projectors of the $\LTWO$- and the $\HDIV$-DG methods increases.
Table~\ref{tab:Norms_kConvergence} shows the convergence of the $\LP{\nf{3}{2}}{}$-norm of the material derivative $\ff_h^t=\partial_t\uu_h+\rb{\uu_h\ip\nabla_h}\uu_h$, the discrete Helmholtz--Hodge projectors $\HLhd{\ff_h^t}$ and $\HLhdc{\ff_h^t}$, and their difference $\HLhd{\ff_h^t}-\HLhdc{\ff_h^t}$.
One can see that, as expected, for $k=8$, both methods detect a comparable amount of divergence-free forces in the discrete convective term.
This is a possible explanation why non-pressure-robust methods may work comparably good whenever higher-order methods are considered.
For lower-order methods, on the other hand, the difference between the Helmholtz--Hodge projectors increases considerably.
~\\

\begin{figure}[h]
\centering
	\includegraphics[width=0.48\textwidth]{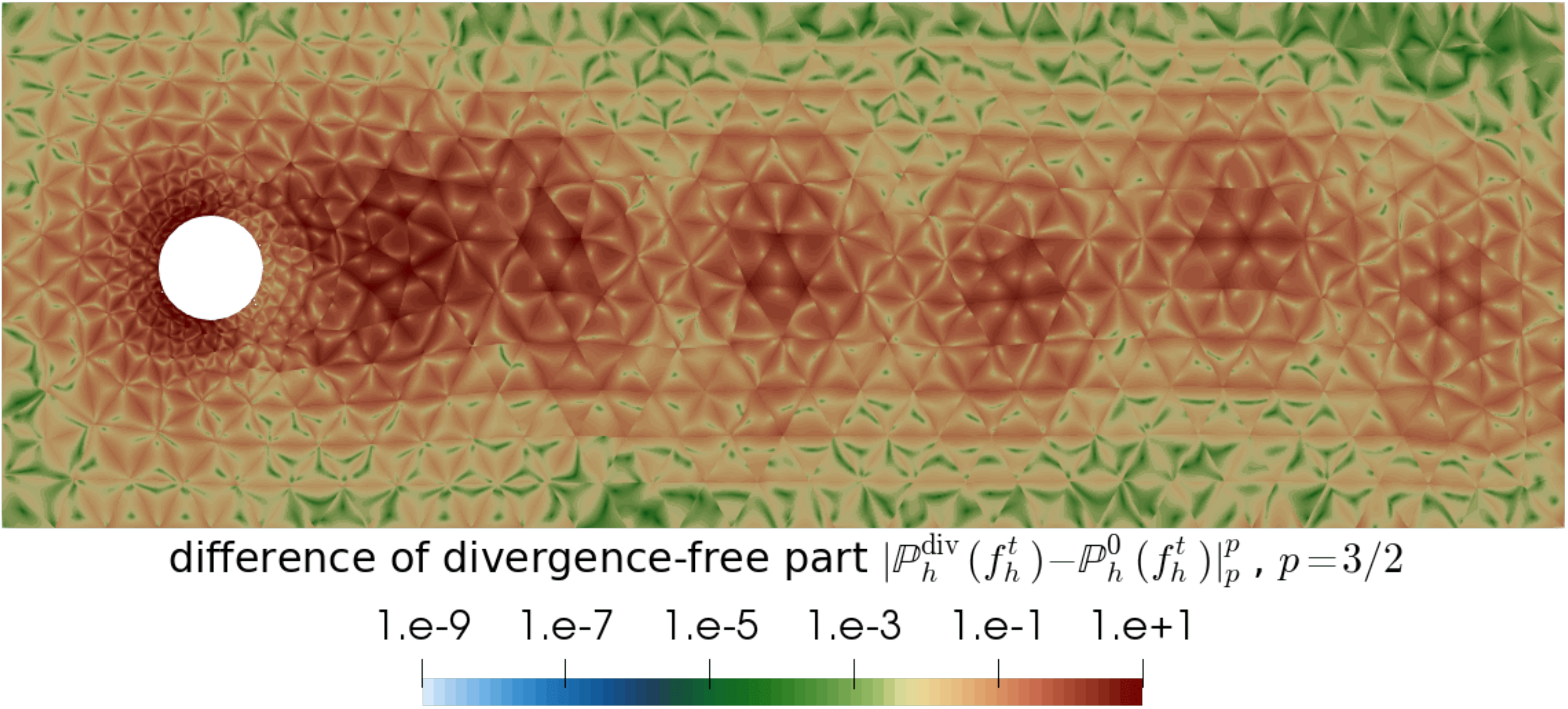} \hspace{5pt}
	\includegraphics[width=0.48\textwidth]{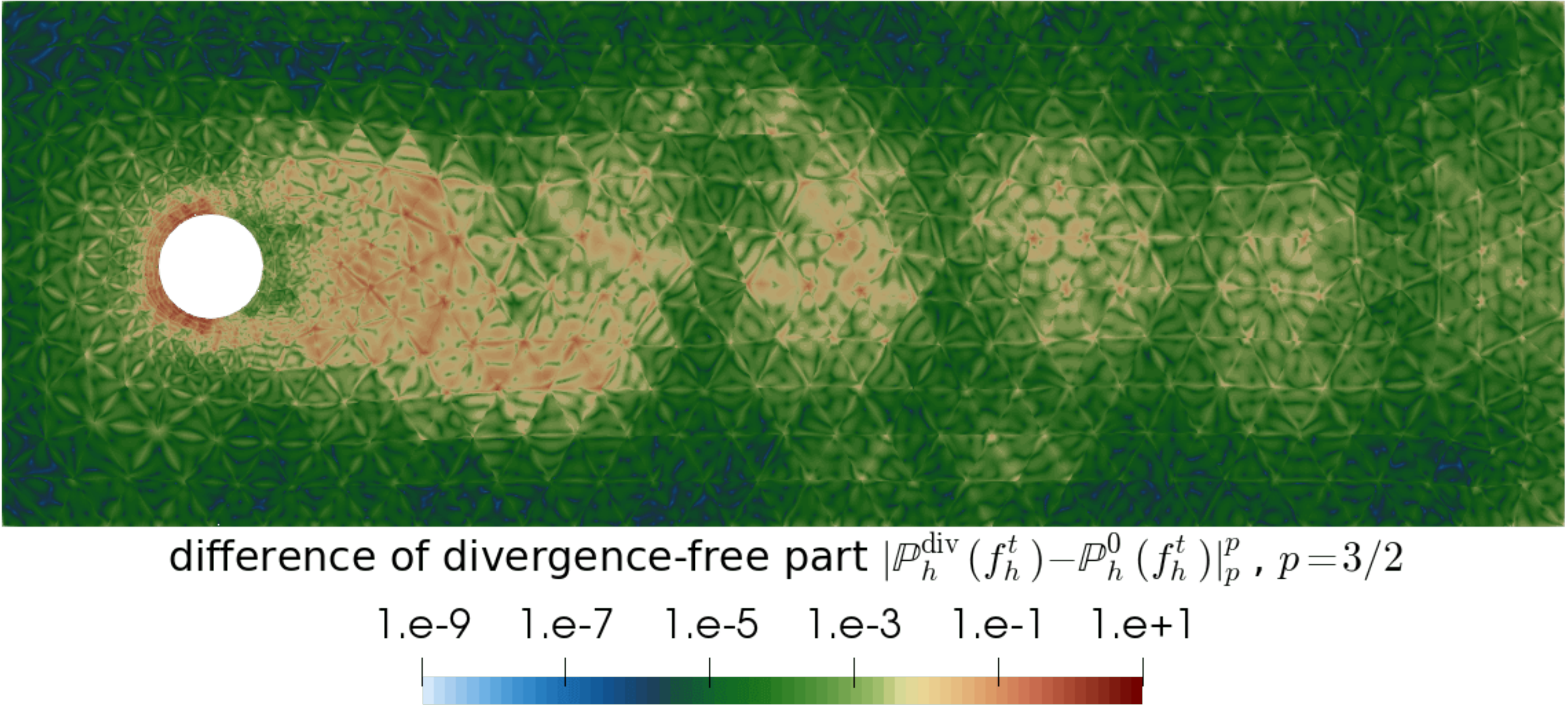} \\
	\includegraphics[width=0.48\textwidth]{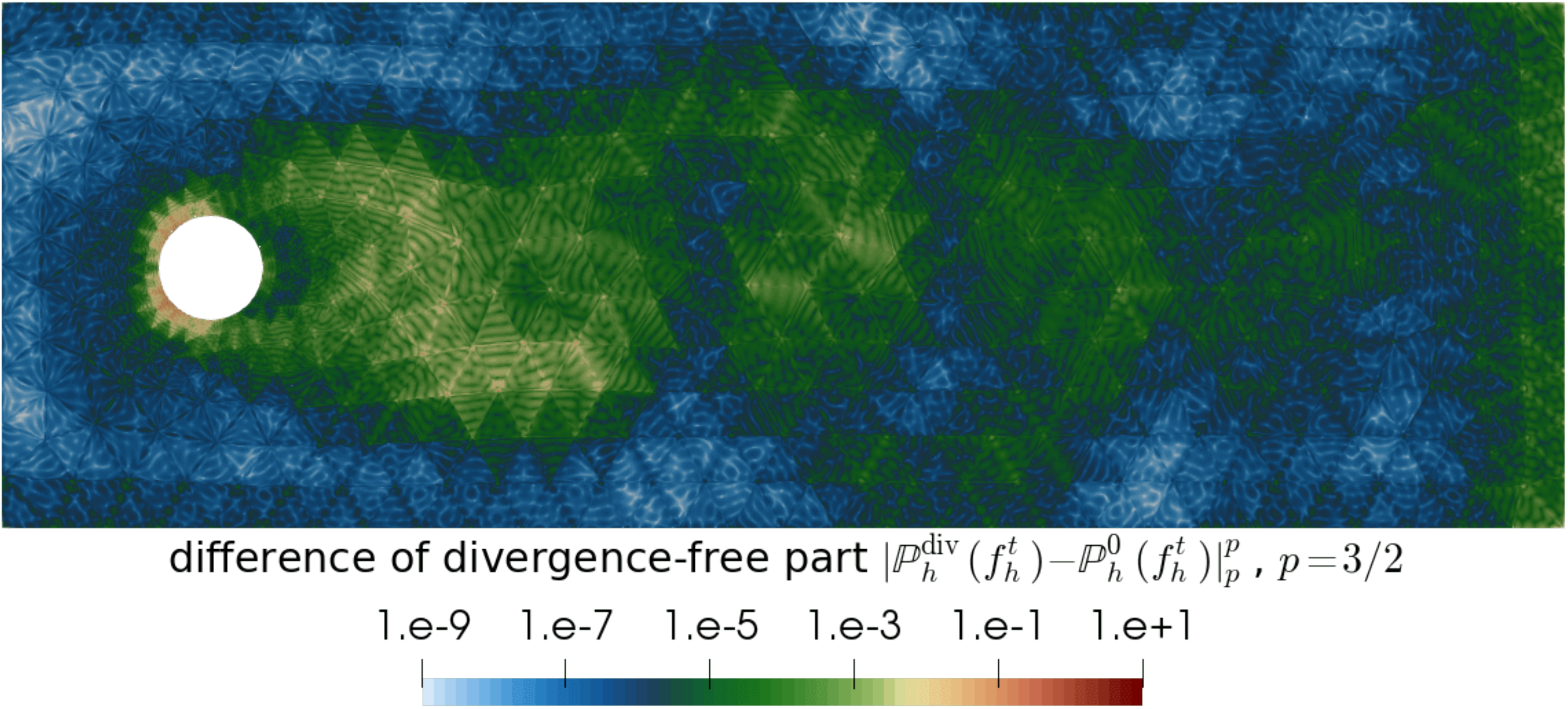} \hspace{5pt}
	\includegraphics[width=0.48\textwidth]{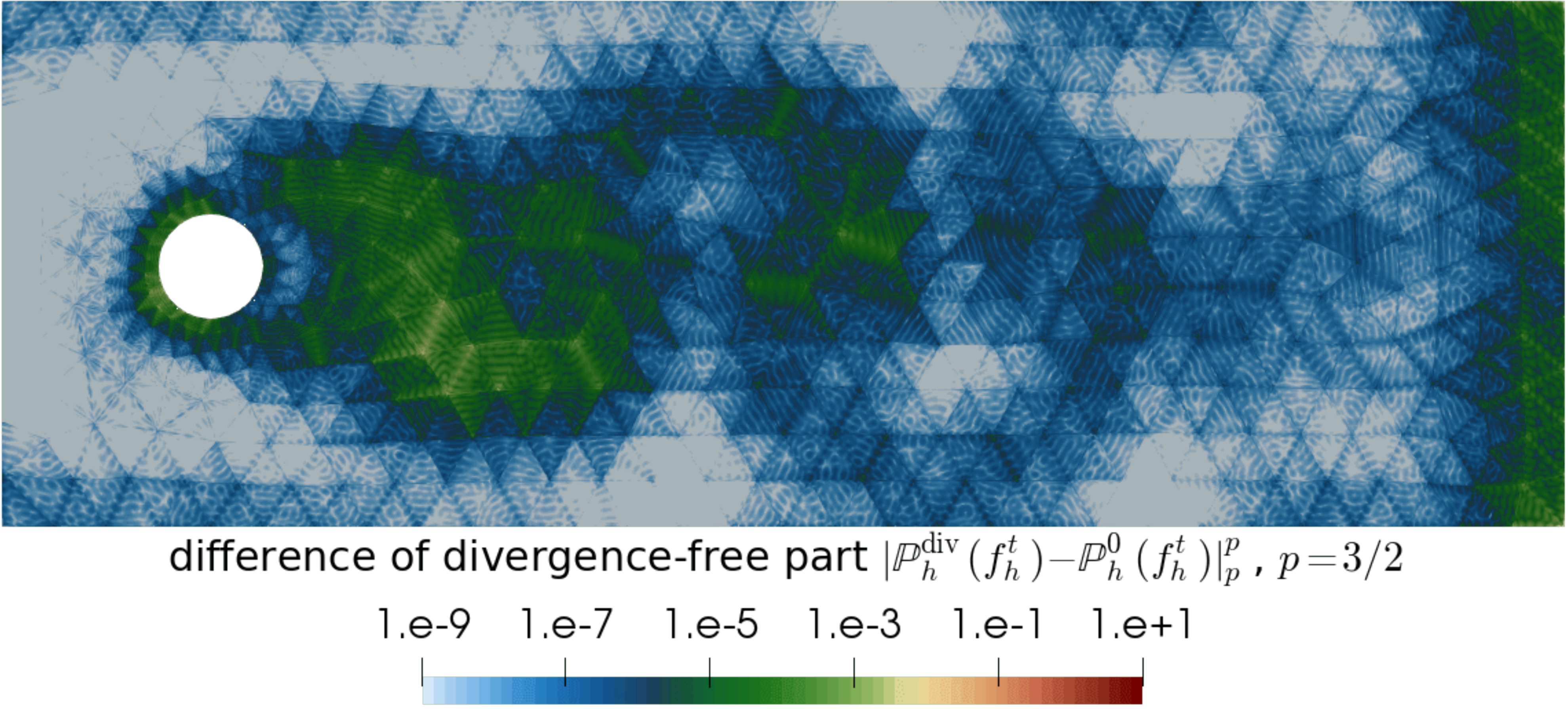} \\
\caption{Difference of discrete Helmholtz--Hodge projectors $\abs{\HLhd{\ff_h^t}-\HLhdc{\ff_h^t}}_\nf{3}{2}^\nf{3}{2}$ of $\ff_h^t$ for different polynomial orders $k\in\set{2,4,6,8}$ (from top left to bottom right). }
\label{fig:Karman_higher_order}
\end{figure}

Consistent with these observation, Figure~\ref{fig:Karman_higher_order} shows the pointwise plots of the difference between the Helmholtz--Hodge projectors.
It is especially interesting that their difference concentrates in the vicinity of the object.
This potentially means that pressure-robust methods have a higher accuracy near objects which are located in a flow.

\section{Conclusion and outlook} 
\label{sec:Conclusions}
The main intention of this contribution is to illustrate that pressure-robust space discretisations outperform non-pressure-robust space discretisations for incompressible Navier--Stokes flows, especially at high Reynolds numbers. 
The main reason is that in the incompressible Euler limit for $\ff=\zero$ the material derivative is a gradient field, which can be handled more appropriately by pressure-robust methods.
Indeed, in high Reynolds number flows the pressure gradient is typically strong and non-trivial, like in vortex dominated flows where the linear pressure gradient balances the quadratic centrifugal force. 
Further, the superior behaviour of pressure-robust methods relies on a more consistent discrete Helmholtz--Hodge projector, which vanishes for arbitrary gradient fields.
~\\

We want to conclude our contribution with some speculations on the relation between high-order and low-order flow solvers. 
In recent years high-order space discretisation was proposed as an efficient means for the simulation of challenging flow problems, like incompressible Navier--Stokes flows at high Reynolds numbers or real-world applications in computational fluid dynamics \cite{2018JHyDy..30....1X, book:ks:2013}.
The potential benefits of high-order discretisations are suggested to be twofold \cite{2018JHyDy..30....1X}:
\begin{itemize}
	\item exponential convergence under certain regularity assumptions which can be achieved from a clever combination of high-quality mesh generation with $p$-refinement away from domain boundaries;
	\item better diffusion and dispersion properties of the spatially discretised differential operators.
\end{itemize}

While the approximation property argument of $h/p$ methods is convincing, there is an inherent prerequisite in the argument: the availability of high-order boundary approximations of 3D mesh generators \cite{2018JHyDy..30....1X}, since the approximation accuracy of the boundary restricts the maximally achievable approximation order of the entire algorithm. 
	Put differently: if such a high quality mesh does not exist for a given complicated domain (as it often happens in practice; e.g.\ boundaries generated by CAD tools yield spline surfaces of third order), high-order methods will simply not achieve high-order convergence rates, see Subsection \ref{sec:3DEthierSteinmanBC}.
~\\

Our contribution shows that the accuracy of all non-pressure-robust space discretisations suffers from a hidden consistency error, namely the consistency error of the corresponding discrete Helmholtz--Hodge projector, see Lemma \ref{lem:DiscHelmLerayNonDivFreeH1} and Lemma \ref{lem:DiscHelmLerayNonDivFreeDG}.
Evidently, this consistency error can be reduced by high-order methods, due to a simple Taylor expansion argument. 
We conjecture that such hidden consistency errors explain the better diffusion and dispersion properties of high-order methods at least partially.



\bibliographystyle{plain}

\bibliography{pRobust-HO_BibTeX}

\end{document}